\documentclass
[
11pt]
{amsart}

\usepackage{hyperref}
\usepackage{cite}
\usepackage{amssymb}
\usepackage{amsmath}
\usepackage{amsthm}
\usepackage{amsfonts}
\usepackage{bbm}
\usepackage{enumitem} 
\usepackage{graphicx}
\usepackage{xcolor}
\usepackage[toc,page]{appendix}
\usepackage{subfigure}
\usepackage{mathtools}
\usepackage[normalem]{ulem}
\usepackage{comment}

\newcommand{\Vol}{\operatorname{Vol}}
\newcommand{\R}{\mathbb{R}}
\newcommand{\supp}{\operatorname{supp}}

\makeatother
\numberwithin{equation}{section}
\newtheorem{theorem}{Theorem}[section]

\newtheorem{proposition}[theorem]{Proposition}

\newtheorem{lemma}[theorem]{Lemma}
\newtheorem{corollary}[theorem]{Corollary}
\theoremstyle{definition}
\newtheorem{definition}[theorem]{Definition}
\newtheorem{example}[theorem]{Example}
\theoremstyle{remark}
\newtheorem{remark}[theorem]{Remark}
\newtheorem{notation}[theorem]{Notation}

\newcommand{\bbR}{{\mathbb R}}

\newcommand{\bbG}{{\mathbb G}}
\renewcommand{\geq}{\geqslant}
\renewcommand{\leq}{\leqslant}
\renewcommand{\epsilon}{\varepsilon}

\renewcommand{\ge}{\geqslant}
\renewcommand{\le}{\leqslant}
\newcommand{\Per}{\operatorname{Per}}

\newcommand{\rohan}[1]
{{\color{red} Rohan says: #1}}

\newcommand{\masha}[1]
{{\color{blue} Masha says: #1}}

\newcommand{\liangbing}[1]
{{\color{green} Liangbing says: #1}}

\newcommand{\andrea}[1]
{{\color{teal} Andrea says: #1}}

\title{Heat content and spectrum for subordinated sub-Laplacians}

\author[Maria Gordina]{Maria Gordina{$^{\dag}$}}
\thanks{\footnotemark {$\dag$} Research was supported in part by NSF Grant DMS-2246549.}
\address[Maria Gordina]{Department of Mathematics\\
University of Rochester\\
Rochester, NY 14627,  U.S.A.}
\email{maria.gordina@rochester.edu}

\author[Liangbing Luo]{Liangbing Luo}
\address[Liangbing Luo]{
Department of Mathematics and Statistics\\
Queen's University\\
Kingston, ON, Canada
K7L 3N6}
\email{liangbing.luo@queensu.ca}

\author[Andrea Pinamonti]{Andrea Pinamonti{{$^\ddagger$}}}
\thanks{\footnotemark {$\ddagger$} Research was supported in part by the University of Trento and by INdAM-GNAMPA 2026 Project \emph{Variational, Geometric, and Analytic Perspectives on Regularity}, CUP E53C25002010001.}
\address[Andrea Pinamonti]{Department of Mathematics, University of Trento, Via Sommarive 14, 38123 Povo (Trento), Italy}
	\email{andrea.pinamonti@unitn.it}

\author[Rohan Sarkar]{Rohan Sarkar}
\address[Rohan Sarkar]{Department of Mathematics
\\
University of Virginia
\\
Charlottesville, VA 22903, U.S.A.
}
\email{kfg3et@virginia.edu}

\begin{document}

\begin{abstract}
We study heat content and spectral properties of subordinated sub-Laplacians on arbitrary Carnot groups. We consider restrictions of such operators to bounded open sets with zero Dirichlet boundary condition and study their spectral properties.  In particular, for a large class of subordinators we give explicit eigenvalue estimates in terms of the subordinator and the eigenvalues of the sub-Laplacian. We also provide large-time asymptotics of the heat content for the subordinated sub-Laplacian in terms of its spectral gap.
For fractional sub-Laplacians we prove short-time asymptotics for the corresponding heat content and relative heat content.
Our approach combines semigroup methods, probabilistic techniques, geometric measure theory, and  heat kernel estimates.
\end{abstract}
\maketitle
\tableofcontents

\section{Introduction}

Carnot groups are Lie groups equipped with a natural sub-Riemannian structure. They have been studied in analysis and  metric geometry, e.g.  \cite{CapognaDanielliPaulsTysonBook, BonfiglioliLanconelliUguzzoniBook, AgrachevBarilariBoscainBook2020, LeDonne2017}. One of the main objects of study on Carnot groups is a sub-Laplacian. The aim of this article is to study such operators under (Bochner) subordination using both analytic and probabilistic techniques.

The main object of interest for us is the heat content for subordinated sub-Laplacians which measures the total energy over a set. It is known that the small-time and large-time asymptotics of the heat content behavior encode geometric and spectral properties of the underlying space and the operator. 

For a smooth Riemannian manifold $(M,g)$, the heat content of a domain $\Omega\subset M$ for the Laplace-Beltrami operator $\Delta$ is defined as
\begin{align*}
    \widetilde{H}_\Omega(t)=\int_{\Omega}u(t,x)d\Vol_g(x),
\end{align*}
where $u:[0,\infty)\times \Omega\longrightarrow \mathbb R$ satisfies 
\begin{align*}
\begin{cases}
    (\partial_t +\Delta)u(t,x)=0 \quad &\mbox{for all $(t,x)\in [0,\infty)\times\Omega$}, \\
    u(t,x)=0  \quad &\mbox{for all $(t,x)\in [0,\infty)\times\partial\Omega$}, \\
    u(0,x) = 1 \quad  &\mbox{for all $x\in\Omega$}.
\end{cases}
\end{align*}
The small-time asymptotics of the heat content for the Laplace-Beltrami operator can be given in terms of the volume, perimeter of the set, and the mean curvature of its boundary. More precisely, van den Berg and Gilkey in \cite{BergGilkey1994} proved that for an open, bounded subset $\Omega$ with a $C^3$ boundary in a Riemannian manifold, as $t\to 0$ one has

\begin{align*}
    \widetilde{H}_\Omega(t)=\Vol_g(\Omega)-\sqrt{\frac{2t}{\pi}}\sigma(\partial\Omega)+\frac{t}{2}\int_{\partial\Omega} h_\Omega(x)\sigma(dx)+ O(t^{3/2}),
\end{align*}
where $h_\Omega$ denotes the mean curvature of the domain at the boundary, and $\sigma$ is the surface measure on $\partial\Omega$. As a generalization of the above result, Savo in \cite{Savo1999} proved that for any open bounded set $\Omega$ with $C^\infty$ boundary in a Riemannian manifold, its heat content admits an asymptotic expansion of any order, namely, for any $k\ge 3$ there exist explicit constants $a_3,\ldots, a_k$ such that
\begin{align*}
    \widetilde{H}_\Omega(t)=\Vol_g(\Omega)-\sqrt{\frac{2t}{\pi}}\sigma(\partial\Omega)+\frac{t}{2}\int_{\partial\Omega} h_\Omega(x)\sigma(dx)+\sum_{i=3}^k a_i t^{i/2} + O(t^{\frac{k+1}{2}}).
\end{align*}

While the heat content asymptotics for Laplace-Beltrami operators have been known for a while, such results for non-local operators have been studied quite recently. If one considers fractional Laplacians on Euclidean spaces with $0<\alpha<2$, then the asymptotic behavior depends on $\alpha$ and one gets the perimeter or the fractional perimeter of $\Omega$ depending on $\alpha\in (0,1)$ or $\alpha\in [1,2)$, see \cite{ParkSong2022}. 

A similar behavior is observed for the fractional sub-Laplacian on Carnot groups. It has been recently proved in \cite{Sarkar2026} that heat content asymptotics for the fractional sub-Laplacian operator defined by 
\begin{align*}
\left(\Delta_{\mathcal H}\right)^{\alpha/2}, \quad \alpha\in (0,2],
\end{align*}
have the same scaling as in the Euclidean case when $\alpha\in [1,2)$, and one recovers the horizontal perimeter of the domain in the limit. 

One of the main goals of this article is to obtain the small-time asymptotics of the heat content of the fractional sub-Laplacian when $\alpha\in (0,1)$, but one needs to consider an appropriate notion of the fractional perimeter, and also appropriate regularity condition for the boundary of the domain intrinsic to Carnot groups.

When $\alpha=2$, the heat content asymptotics in sub-Riemannian geometry was first obtained by Tyson and Wang \cite{TysonWangJ2018} in the case of the 3-dimensional Heisenberg group. They proved that 
\begin{align*}
    \widetilde{H}_\Omega(t)=\Vol(\Omega)-\sqrt{\frac{2t}{\pi}}\Per_{\mathcal H}(\Omega)+\frac{t}{2}\int_{\partial \Omega} h_{\mathcal H}(x)\sigma(dx) + O(t^{3/2}),
\end{align*}
where $h_\mathcal{H}(x)$ denotes the horizontal mean curvature at $x\in\partial\Omega$, assuming that the boundary of $\Omega$ is completely non-characteristic. This result was later established by Rizzi and Rossi in \cite{RizziRossi2021} for any sub-Riemannian manifold, still assuming that $\partial\Omega$ does not contain any characteristic points.
Note that this setting is fundamentally different from Riemannian manifolds, as there is no canonical measure analogous to the Riemannian volume and no canonical connection analogous to the Levi-Civita connection. While for Carnot groups we use a Haar measure, curvature notions are still a challenge for general Carnot groups. 

In this article, we study properties of subordinated sub-Laplacians on a bounded open set $\Omega$ in a Carnot group $\bbG$. These operators are of particular interest because they are non-local and are defined in a setting where the underlying geometry differs substantially from that of Euclidean space. Consequently, identifying suitable geometric conditions under which these operators satisfy analogues of known Euclidean results is a highly non-trivial problem. Moreover, the analytical tools developed in \cite{RizziRossi2021} do not apply to such operators in the context of heat content asymptotics. The main results of the paper can be summarized in three parts as follows. Forthe precise assumptions we refer to later parts of the paper.

\begin{enumerate}[leftmargin=*]
    \item \textbf{Eigenvalues of subordinated sub-Laplacian.} In Theorem~\ref{thm:eigenvalue_comparison}, we establish an eigenvalue comparison between the subordinated sub-Laplacian and the sub-Laplacian, both subject to Dirichlet boundary conditions on $\Omega$. In the Euclidean setting, an analogous estimate was first obtained by Chen and Song in \cite{ChenSong2005} for domains satisfying an exterior cone condition.

We show that this comparison result extends to the setting of Carnot groups. The upper bound \eqref{ineq.UpperBd.Eigenvalue} requires no regularity assumptions on the boundary of $\Omega$. For the lower bound \eqref{ineq.LowerBd.Eigenvalue}, we assume that $\Omega$ satisfies an \emph{intrinsic exterior cone condition}, which provides a natural generalization of the classical exterior cone condition to the sub-Riemannian geometry of Carnot groups. We refer to Section~\ref{s:cones} for the precise definition and further details. In particular, the eigenvalue comparison theorem applies to every bounded domain with compact $C^1$ boundary and no \emph{characteristic points} (see Definition~\ref{def:characteristic_pt}).
\item \textbf{Heat content asymptotics for subordinated sub-Laplacian.} We then prove explicit large and small time asymptotics of the heat content for subordinated sub-Laplacians. 
    
In Theorem~\ref{thm:heat_content_large} we obtain the large-time asymptotic behavior of the heat content for a class of subordinated sub-Laplacians, including fractional sub-Laplacians, without requiring any regularity assumptions on the boundary of $\Omega$. Our result shows that when $t\to\infty$, the heat content at time $t$ scales like $e^{-\lambda_{1,\phi} t}$, where $\lambda_{1,\phi}$ is the spectral gap of the subordinated sub-Laplacian with Dirichlet boundary condition. The proof of this result uses a careful application of the Krein-Rutman theorem and irreducibility of the corresponding semigroups, developed in \cite{CarfagniniGordina2024} for sub-Laplacians on Carnot groups.

For $\alpha\in[1,2)$, the small-time asymptotics of the heat content for the fractional sub-Laplacian were proved in \cite{Sarkar2026}. Theorem~\ref{thm:spectral_heat} complements this result by treating the range $0<\alpha<1$, therefore establishing the existence of the corresponding limits for all $\alpha\in(0,2]$.

In the study of small-time heat content asymptotics in sub-Riemannian geometry, it is typically assumed that the boundary of the domain is smooth and non-characteristic; see, for instance, \cite{TysonWangJ2018, RizziRossi2021, Sarkar2026}. In Theorem~\ref{thm:spectral_heat}, we instead prove the asymptotics for domains satisfying the \emph{volume density condition}. This condition is flexible enough to include various domains with non-smooth boundaries, as well as domains whose boundaries contain (infinitely many) characteristic points; see Examples~\ref{ex:1}--\ref{ex:3}.
\item  \textbf{Relative heat content asymptotics for fractional sub-Laplacian.} An explicit small time asymptotics formula of the relative heat content for fractional sub-Laplacians with parameter $0<\alpha<2$ is given in Theorem~\ref{thm:heat_content}. This generalizes the results obtained by Agrachev, Rizzi, and Rossi in \cite{AgrachevRizziRossi2024} in the context of relative heat content asymptotics for Carnot groups. When $1\leqslant \alpha <2$, we assume that $\Omega$ has a $C^\infty$ boundary without any characteristic points. When $\alpha=1$, our results hold with no regularity assumptions on the boundary of $\Omega$, and for any set $\Omega$ with finite fractional horizontal perimeter. 
\end{enumerate}

The proofs of our main results are based on probabilistic techniques. From the standard theory of Markov processes, if $\mathcal{L}$ is the generator of a Feller process $(X_t)_{t\ge 0}$ on a metric measure space $(M,d,\mu)$, the solution to the Dirichlet boundary value problem 
\begin{align*}
    \begin{cases}
        \partial_t u(t,x)=\mathcal{L}u(t,x) \quad &\mbox{for all $(t,x)\in [0,\infty)\times\Omega$} \\
        u(t,x)=0 \quad & \mbox{for all $(t,x)\in [0,\infty)\times\partial\Omega$}\\
        u(0,x)=1 \quad & \mbox{for all $x\in\Omega$}
    \end{cases}
\end{align*}
under sufficient regularity conditions on the boundary of $\Omega$ can be represented as 
\begin{align*}
    u(t,x)=\mathbb P_x(\tau_\Omega>t),
\end{align*}
where $\tau_\Omega$ is the first exit time of $X_t$ from $\Omega$. Therefore, the heat content of $\Omega$ for the operator $\mathcal L$ can be written as 
\begin{align}\label{eq:heat_content_def}
    \widetilde{H}_\Omega(t)=\int_{\Omega}\mathbb P_x(\tau_\Omega>t)\mu(dx).
\end{align}
We note that the above probabilistic representation of the heat content is well defined for any measurable set $\Omega$, and does not require any  regularity condition on the boundary of $\Omega$. In this article, we adopt \eqref{eq:heat_content_def} as the definition of heat content. In particular, 
the heat content for the subordinated sub-Laplacian can be written as
\begin{align*}
    \widetilde{H}^\phi_\Omega(t)=\int_{\Omega}\mathbb P_x(\tau_\Omega>t) dx, \quad \mbox{and} \quad \Vol(\Omega)- \widetilde{H}^\phi_\Omega(t)=\int_{\Omega}\mathbb P_x(\tau_\Omega\le t) dx,
\end{align*}
where $\tau_\Omega$ is the first exit time of the subordinated hypoelliptic Brownian motion on the Carnot group, and $\phi$ is the Bernstein function associated with the subordinator. We refer to Section~\ref{s:probability} for detailed discussion on such processes. This shows that the asymptotic behavior of $H^\phi_\Omega(t)$ is closely related to the asymptotic behavior of the exit probability 
\begin{align*}
    \mathbb P_x(\tau_\Omega>t).
\end{align*}
This probabilistic approach was also used in \cite{TysonWangJ2018, Sarkar2026}. 

Let us briefly explain why the argument used in \cite{Sarkar2026} to obtain the fractional heat content asymptotics\footnote{For brevity, we use this term to indicate the heat content for fractional sub-Laplacian.} for $\alpha\in[1,2]$ cannot be directly adapted to the case $\alpha\in(0,1)$. First, the proof of \cite[Theorem~1.1]{Sarkar2026} relies heavily on the $C^2$-regularity of the boundary and the absence of characteristic points, assumptions that are substantially stronger than those imposed in Theorem~\ref{thm:spectral_heat}. Another key ingredient in the argument of \cite{Sarkar2026} is the following estimate\footnote{\cite[Lemma~4.3]{Sarkar2026} states the estimate in a different way but has the same implication.} proved in Lemma~4.3: for any $\Omega_1\Subset\Omega$ and every $\alpha\in(0,2)$,
\begin{align*}
\int_{\Omega_1}\mathbb P_x(\tau_\Omega\le t)dx = O(t) \quad \mbox{as $t\to 0$}.
\end{align*}
When the fractional heat content subtracted from the volume of the domain is normalized by $t^{1/\alpha}$ or $t\log(1/t)$, depending on $\alpha>1$ or $\alpha=1$, this estimate implies that the contribution from points away from the boundary is negligible in the small-time limit. Consequently, the asymptotic behavior is governed by the geometry of the boundary, leading to the horizontal perimeter in the limiting expression. In contrast, for $\alpha\in(0,1)$, Theorem~\ref{thm:spectral_heat} shows that the appropriate scaling is of order $t$. Hence, the estimate above is no longer sufficient to identify the precise limiting behavior, and one should also expect the effect of points away from the boundary in the limit.

To overcome the above difficulty, we first obtain the relative heat content asymptotics for the fractional sub-Laplacian; see Theorem~\ref{thm:heat_content}. For $\alpha\in(0,1)$, this asymptotics formula holds for every bounded measurable set with a finite fractional horizontal perimeter. We then show that the corresponding asymptotic relation also holds for the fractional heat content. This step relies on the Ikeda-Watanabe theorem (see \cite{Ikeda-Watanabe}) to determine the distribution of the subordinated hypoelliptic Brownian motion at the time of its exit from the domain. Another observation that plays a central role in this part of the argument is Proposition~\ref{prop:skip_boundary}, where we prove that, \emph{under the volume density condition, the subordinated hypoelliptic Brownian motion does not hit the boundary of the domain upon exiting}. An analogous property is known for certain L\'evy processes on Euclidean spaces under suitable technical assumptions on their L\'evy measures; see \cite{Bogdan1997, Sztonyk2000, Bogdan_et_al2020}. However, when viewed as a Markov process on the underlying Euclidean space, subordinated hypoelliptic Brownian motion on a Carnot group is not a L\'evy process. Therefore, Proposition~\ref{prop:skip_boundary} enlarges the class of jump-valued processes on  Euclidean space that satisfy the boundary skipping property during their exit time from domains.

Apart from the main results discussed above, a substantial portion of the paper is devoted to developing suitable boundary regularity conditions for domains in Carnot groups that are sufficient for our main results. In Section~\ref{s:cones}, we extend the Euclidean exterior cone condition to the setting of Carnot groups using the intrinsic cones introduced in \cite{FranchiSerapioniSerra_Cassano2003a}. In Euclidean spaces, cones are isometric under translations and rotations, which considerably simplifies the analysis of isotropic L\'evy processes\footnote{A L\'evy process whose distribution is invariant under isometries.} when restricted to cones. In Carnot groups, however, it is generally difficult to identify transformations that both preserve the volume of intrinsic cones and leave the distribution of the subordinated hypoelliptic Brownian motion invariant. To circumvent this difficulty, in Lemma~\ref{lem:volumes} we establish uniform bounds for the volumes of intrinsic cones and for the corresponding probabilities, which are sufficient for the proofs of our main results.

We also introduce a volume density condition in Carnot groups as a natural extension of its Euclidean counterpart considered in \cite{Jang-MeiWu2002}. We investigate the relationship between the intrinsic exterior cone condition and the volume density condition and establish sufficient geometric conditions under which these regularity assumptions are satisfied; see, for example, Theorem~\ref{thm:noncharacteristic-cones}.

The rest of the paper is organized as follows. After reviewing the necessary background on Carnot groups in Section~\ref{s:preliminaries}, we state our main results in Section~\ref{s:main_results}. In Section~\ref{s:probability}, we discuss the probabilistic interpretation of the subordinated sub-Laplacian. Boundary regularity conditions for domains in Carnot groups are introduced in Section~\ref{sec.Boundary}. Finally, the proofs of the main results are presented in Sections~\ref{s:proof1}, \ref{s.proof2}, and \ref{s:proof3}.

\vspace{1cm}
\textbf{AI disclosure:} The authors used OpenAI's ChatGPT during the preparation of this manuscript
for language editing, bibliography assistance, and preliminary mathematical
discussion. The authors independently verified all mathematical content and
references and assume full responsibility for the final manuscript.
\section{Preliminaries}\label{s:preliminaries}

\subsection{Carnot groups}
We start by introducing the basics of Carnot groups needed for our results. 

\begin{definition}
We say that $\bbG$ is a \emph{Carnot group} of step $k$ if $\bbG$  is a connected and simply connected Lie group whose Lie algebra $\mathfrak{g}$ is \emph{stratified}, that is, there are non-trivial linear subspaces $V_1,\ldots, V_k$ of $\mathfrak{g}$ such that
\begin{align*}
\mathfrak{g}=V_{1}\oplus\cdots\oplus V_{k},
\end{align*}
and
\begin{align}\label{e.Stratification}
& \left[V_{1}, V_{i-1}\right]=V_{i},\hskip0.1in  V_{i} \neq \{0\}, \hskip0.1in 2 \leqslant i \leqslant k,
\notag
\\
& [ V_{1}, V_{k} ]=\left\{ 0 \right\}.
\end{align}
\end{definition}

To exclude trivial cases we assume that the dimension of $\mathfrak{g}$ is at least $3$. 
From now on we denote by $\mathcal{H}:=V_{1}$ the space of \emph{horizontal} vectors that generate the rest of the Lie algebra.

By \cite[Theorem 2.2.18]{BonfiglioliLanconelliUguzzoniBook}, each $n$-dimensional Carnot group $\bbG$ is isomorphic to some homogeneous Carnot group 
 $(\mathbb{R}^n,\star)$ on $\mathbb{R}^n$ (see \cite[Definition 1.4.1]{BonfiglioliLanconelliUguzzoniBook} for the precise definition) with the same Lie algebra $\mathfrak{g}$ as that of $\bbG$, and the isomorphism is given by the exponential map
 \begin{align*}
     \exp:\mathfrak{g}\longrightarrow \mathbb G.
 \end{align*}
 Therefore, $\bbG$ can be identified with some homogeneous Carnot group $(\mathbb{R}^n,\star)$ on $\mathbb{R}^n$. We refer to \cite{BonfiglioliLanconelliUguzzoniBook} for more details on this identification. This allows us to work on homogeneous Carnot groups instead of general Carnot groups. For the sake of simplicity, we will often use the notation $xy$ to indicate the product of the group elements $x,y\in\mathbb G$. The identity element of $\mathbb{G}=(\mathbb R^n,\star)$ is given by $e=(0,\ldots, 0)\in \mathbb{R}^n$. Since we identify any Carnot group with a homogeneous Carnot group via the exponential map, without loss of generality we are going to assume that $x^{-1}=-x$ for any $x\in\mathbb G$.

Set
$m_i=\dim(V_i)$, for $i=1,\dots,k$ and $h_i=m_1+\dots +m_i$, so that $h_k=n$ with $h_0:=0,\ m:=m_1$.
We denote by $Q$ the {\em homogeneous dimension} of $\bbG$, i.e.
 \begin{align*}
 Q:=\sum_{i=1}^{k} i \dim(V_i).
 \end{align*}

For any $\lambda >0$, the {\em dilation} $\delta_\lambda:\bbG\longrightarrow\bbG$, is defined as
\begin{equation}\label{dilatazioni}
\delta_\lambda(\xi_1,\ldots,\xi_n)=
(\lambda \xi_1,\ldots,\lambda^k\xi_k),
\end{equation} 
where $x=(\xi_1,\dots,\xi_k)\in\bbR^{m_1}\times\cdots\times \bbR^{m_k}\equiv\bbG$.

By \cite[Proposition 1.3.21]{BonfiglioliLanconelliUguzzoniBook}, the Lebesgue measure $dx$ on $\mathbb{R}^n$ is invariant with
respect to the left and the right translations on $\mathbb{G}$. If $A\subset \bbG$ is Lebesgue measurable, we
write $\Vol(A)$ to denote its Lebesgue measure. 
\begin{definition}
A continuous function $\|\cdot\|:\mathbb G\longrightarrow [0,\infty)$ is called a \emph{homogeneous quasinorm} on the Carnot group $\mathbb{G}$ if the following conditions hold:
\begin{enumerate}[leftmargin=*]
    \item $\|\delta_\lambda x\|=\lambda \|x\|$ for all $\lambda>0$ and $x\in\mathbb G$,
    \item $\|x\|>0$ if and only if $x\neq 0$.
\end{enumerate}
The quasinorm is called \emph{symmetric} if $\|x\|=\|x^{-1}\|$ for all $x\in\mathbb G$. If $\|\cdot\|$ also subadditive, that is,
\begin{align*}
    \|xy\|\le \|x\|+\|y\|, \quad x,y\in\mathbb G,
\end{align*}
we call it a \emph{homogeneous norm}.
\end{definition}
By \cite[Proposition~5.1.4]{BonfiglioliLanconelliUguzzoniBook} any two homogeneous quasinorms are
equivalent, and from now on we denote by $\|\cdot\|$
an arbitrary homogeneous norm; all the estimates that we give are
then the same up to changes in the constants.
We denote by 
\begin{align*}
\mathbb{B}(x,r)=\{y\in \bbG: \| x^{-1}y\| < r\}
\end{align*}
the open ball centered at $x\in \bbG$ with radius $r>0$ and by $\mathbb{B}(r)=\mathbb{B}(e,r)$. 

\subsection{Sub-Riemannian structure on \texorpdfstring{$\bbG$}{}}
For $x\in \mathbb{G}$, we denote by $L_{x}: \mathbb{G} \longrightarrow \mathbb{G}$ the \emph{left translation}
\begin{align*}
L_{x} y:=x  y,  \text{ for } y \in \mathbb G,
\end{align*}
and the corresponding pushforward (differential) $(L_x)_{\ast}:T\mathbb{G} \longrightarrow T\mathbb{G}$ by
\begin{align*}
\left( L_{x} \right)_{\ast}: T_{y}\mathbb{G} & \longrightarrow T_{x y}\mathbb{G}
\\
v &\longmapsto (L_x)_{\ast}v.
\end{align*}
Here $T\mathbb{G}$ denotes the tangent bundle of $\mathbb{G}$ and $T_x\mathbb{G}$ the tangent space at $x\in \mathbb{G}$.

Now we describe the sub-Riemannian structure on $\mathbb{G}$. We assume that $\mathcal{H}$ is equipped with an inner product $\langle\cdot,\cdot\rangle_{\mathcal{H}}$ which induces a norm $\vert\cdot\vert_{\mathcal{H}}$. One may use left translation to define a \emph{horizontal distribution} $\mathcal{D}$, a sub-bundle of $T\mathbb{G}$, and a metric on $\mathcal{D}$ as follows. First, we identify the space $\mathcal{H} \subset \mathfrak{g}$ with $\mathcal{D}_{e}\subset T_e \mathbb{G}$. Then for any $x\in \mathbb{G}$, we define $\mathcal{D}_{x}:=(L_x)_{\ast}\mathcal{D}_{e}$. Then a metric on $\mathcal{D}$ can be defined by translating back to $\mathcal{H} \subset \mathfrak{g}$, that is,
\begin{align*}
\langle u, v\rangle_{\mathcal{D}_x} &:= \langle  (L_{x^{-1}})_{\ast} u,(L_{x^{-1}})_{\ast} v\rangle_{\mathcal{D}_e} \\
	&= \langle (L_{x^{-1}})_{\ast} u,(L_{x^{-1}})_{\ast} v\rangle_{\mathcal H}  \text{ for all } u, v\in\mathcal{D}_x.
\end{align*}

\begin{definition}
Let $\{X_1,\ldots,X_{m}\}$ be an orthonormal frame for $\mathcal{D}$. Define the \emph{horizontal gradient} of $f$, denoted by
$\nabla_{\mathcal{H}}f$, as the horizontal section
\begin{equation*}
\nabla_{\mathcal{H}}f:=\sum_{i=1}^{m}(X_if)X_i,
\end{equation*}
for any function $f:\bbG\longrightarrow \bbR$ for which the partial derivatives
$X_if$ exist. The second-order differential operator 
\begin{equation}\label{eq:sub-Laplacian}
\Delta_{\mathcal{H}}:=-\sum_{i=1}^{m} X_i^2
\end{equation}
defined on $C^{\infty}(\mathbb{G})$ is called a \emph{sub-Laplacian}.
\end{definition}
By \cite{Hoermander1967a}, since the
vector fields $X_1,\ldots,X_m$ satisfy H\"ormander's
bracket-generating condition, the operator $\Delta_{\mathcal H}$ is
hypoelliptic. Moreover, its definition is independent of the choice of
a left-invariant orthonormal frame and depends only on the
sub-Riemannian metric
$\langle\cdot,\cdot\rangle_{\mathcal H}$, see \cite[Theorem~3.6]{GordinaLaetsch2017}.

The operator $\Delta_{\mathcal H}$ is densely defined, symmetric and
non-negative on $L^2(\mathbb G,dx)$. Indeed, for every
$f\in C_c^\infty(\mathbb G)$,
\begin{align*}
\langle \Delta_{\mathcal H}f,f\rangle_{L^2(\mathbb G, dx)}
=\sum_{j=1}^{m}\int_{\mathbb G}|X_jf|^2\,dx\geq 0.
\end{align*}
Moreover, by \cite{Strichartz1986} and  \cite[p.~950]{DGS-C2009} it is known that $\Delta_{\mathcal H}$ is essentially self-adjoint on
$C_c^\infty(\mathbb G)$. Hence, its Friedrichs extension coincides
with its unique self-adjoint extension. With a slight abuse of
notation, we continue to denote this extension by
$\Delta_{\mathcal H}$.
Consider the quadratic form associated with $\Delta_{\mathcal H}$ given by
\begin{align*}
\mathcal E(f,g)
:= \sum_{j=1}^{m}\int_{\mathbb G}X_jf\cdot X_jg\,dx,
\qquad
f,g\in C_c^\infty(\mathbb G).
\end{align*}
Then its closure is a regular, strongly local Dirichlet form on
$L^2(\mathbb G,dx)$, see \cite[p.~1902]{CarfagniniGordina2024} and the references therein. Consequently, the associated self-adjoint
semigroup
\begin{align*}
P_t:=e^{-t\Delta_{\mathcal H}},
\qquad t\geq0,
\end{align*}
is sub-Markovian and admits consistent extensions to contraction semigroups on $L^p(\mathbb G,dx)$ for every $1\leq p\leq\infty$. Moreover, by Hunt's theorem (see \cite[Theorem~5.1]{Hunt1956a}) the heat semigroup on a Carnot group is conservative, that
is,
\begin{align*}
P_t 1= 1.
\end{align*}
Thus, $P_t$ is Markovian and can be applied, in particular, to bounded
Borel functions on $\mathbb G$. The hypoellipticity of the corresponding parabolic operator implies that $P_t$ admits a smooth positive transition density
\begin{align*}
p:(0,\infty)\times\mathbb G\times\mathbb G
\longrightarrow(0,\infty).
\end{align*}
In particular, for every bounded Borel function
$f:\mathbb G\to\mathbb R$,
\begin{align}\label{eq:transition_kernel}
P_tf(x)
=
\int_{\mathbb G}f(y)p(t,x,y)\,dy.
\end{align}
The left-invariance of $\Delta_{\mathcal H}$, and hence of the semigroup $P_t$, implies that
\begin{align*}
p(t,hx,hy)=p(t,x,y)
\end{align*}
for every $t>0$ and every $h,x,y\in\mathbb G$. Since $P_t$ is self-adjoint on $L^2(\mathbb G,dx)$, its transition kernel is symmetric, that is, $p(t,x,y)=p(t,y,x)$. Writing
\begin{align}\label{eq:heat_kernel}
p(t,x):=p(t,e,x),
\end{align}
the left-invariance of the kernel gives $p(t,x,y)=p(t,x^{-1}y)$. Moreover, symmetry yields $p(t,x)=p(t,x^{-1})$.
It follows that, for every $f\in L^2(\mathbb G,dx)$ and, more
generally, for every bounded Borel function $f$,
\begin{align*}
P_tf(x)
=
\int_{\mathbb G}f(y)p(t,x^{-1}y)\,dy
=
\int_{\mathbb G}f(y)p(t,y^{-1}x)\,dy
=
(f*p(t,\cdot))(x),
\end{align*}
where we use the convention
\begin{align*}
(f*g)(x):=\int_{\mathbb G}f(y)g(y^{-1}x)\,dy.
\end{align*}
The function $(t,x)\mapsto p(t,x)$ is the fundamental solution of the
heat equation
\begin{align*}
\left(\frac{\partial}{\partial t}
+\Delta_{\mathcal H}\right)u=0,
\end{align*}
and will be referred to as the hypoelliptic heat kernel.
Finally, using the equivalence between the norm induced by Carnot--Carath\'eodory
distance and any homogeneous norm $\|\cdot\|$ on $\mathbb G$, it is known from \cite[p.~50]{VaropoulosSaloff-CosteCoulhonBook1992} that there
exists a constant $c\geq1$ such that
\begin{align}\label{eq:heat_bound_1}
c^{-1}t^{-\frac Q2}
\exp\left(-\frac{c\|x\|^2}{t}\right)
\leq
p(t,x)
\leq
ct^{-\frac Q2}
\exp\left(-\frac{\|x\|^2}{ct}\right)
\end{align}
for every $x\in\mathbb G$ and every $t>0$.
\section{Main results}\label{s:main_results}

\subsection{Subordination of the sub-Laplacian and eigenvalue comparison}

Suppose that $\phi: [0,\infty) \longrightarrow [0,\infty)$ is a Bernstein
function, that is, $\phi|_{(0,\infty)}\in C^\infty(0,\infty)$ and
\begin{align*}
(-1)^{n-1}\phi^{(n)}(u)\geq0
\qquad
\text{for every }n\geq1\text{ and }u>0.
\end{align*}
Every Bernstein function admits a unique L\'evy--Khintchine
representation
\begin{align}\label{eq:Bernstein}
\phi(u)
=
\phi(0)+bu+\int_0^\infty(1-e^{-ur})\,\nu(dr),
\end{align}
where $b\geq0$ and $\nu$ is a Radon measure on $(0,\infty)$
satisfying
\begin{align*}
\int_0^\infty\min\{1,r\}\,\nu(dr)<\infty.
\end{align*}
In particular, $\phi$ has a continuous extension to $[0,\infty)$. We call $\phi$ a \emph{complete} Bernstein function if $\nu$ admits a completely monotone density $m$, that is, 
\begin{align*}
    (-1)^n m^{(n)}(r)\ge 0 \quad \mbox{for all $r\ge 0$ and $n=0,1,2,\ldots$}
\end{align*}
Throughout the paper we assume that $\phi$ is a non-constant Bernstein function.
Since $\Delta_{\mathcal H}$ is a non-negative self-adjoint operator on
$L^2(\mathbb G,dx)$, for any Bernstein function (not necessarily complete) $\phi$ the Borel functional calculus (see \cite[Theorem~VIII.5, p.~262]{ReedSimonBook}) allows us to
define
\begin{align}\label{eq:subordinated_subL}
\Delta^\phi_{\mathcal H}
:=
\phi(\Delta_{\mathcal H})
=
\int_{[0,\infty)}\phi(\lambda)\,E(d\lambda),
\end{align}
where $E$ denotes the spectral measure of $\Delta_{\mathcal H}$. The
domain of $\Delta^\phi_{\mathcal H}$ is
\begin{align*}
\mathcal D(\Delta^\phi_{\mathcal H})
=
\left\{
f\in L^2(\mathbb G,dx):
\int_{[0,\infty)}
\phi(\lambda)^2\,d\mu_f(\lambda)<\infty
\right\},
\end{align*}
where
\begin{align*}
\mu_f(B)
:=
\langle E(B)f,f\rangle_{L^2(\mathbb G, dx)}
\end{align*}
for every Borel set $B\subset[0,\infty)$. 
The operator $\Delta^\phi_{\mathcal H}$ is self-adjoint and
non-negative on $L^2(\mathbb G,dx)$. Its associated semigroup is
\begin{align}\label{eq:subordinated_semigroup}
P_t^\phi
:=
e^{-t\Delta^\phi_{\mathcal H}}
=
e^{-t\phi(\Delta_{\mathcal H})}.
\end{align}
This is the semigroup obtained by subordinating the heat semigroup
$P_t=e^{-t\Delta_{\mathcal H}}$ by the Bernstein function $\phi$.
Accordingly, we call $\Delta^\phi_{\mathcal H}$ the \emph{sub-Laplacian
subordinated by $\phi$}. More precisely, there exists a convolution semigroup
$(\eta_t)_{t\geq0}$ of sub-probability measures on $[0,\infty)$ (i.e. $\eta_t([0,\infty))\leq 1$ for every $t\geq 0$) such
that
\begin{align*}
\int_{[0,\infty)}e^{-\lambda s}\,\eta_t(ds)
=
e^{-t\phi(\lambda)}
\end{align*}
and
\begin{align*}
P_t^\phi f
=
\int_{[0,\infty)}P_sf\,\eta_t(ds).
\end{align*}
If $\phi(0)=0$, then each $\eta_t$ is a probability measure and the
subordinated semigroup is conservative. We refer to Section~\ref{s.Subordination} for details about the probabilistic interpretation of $P^\phi_t$. In particular, if $0<\alpha<2$ and
\begin{align}\label{eq:Bernstein_alpha}
\phi(0)=0=b,
\qquad
\nu(dr)
=
\frac{\alpha}{2\Gamma(1-\frac{\alpha}{2})}
r^{-1-\frac{\alpha}{2}}\,dr,
\end{align}
then $\phi(u)=u^{\alpha/2}$, and in this case $\phi$ is a complete Bernstein function. Consequently, $\Delta^\phi_{\mathcal H}=\Delta_{\mathcal H}^{\alpha/2}$.
With our sign convention,
$\Delta_{\mathcal H}^{\alpha/2}$ is the non-negative fractional
sub-Laplacian, while $-\Delta_{\mathcal H}^{\alpha/2}$ is the non-positive generator of the fractional heat semigroup
\begin{align*}
P_t^\phi=e^{-t\Delta_{\mathcal H}^{\alpha/2}}.
\end{align*}
\subsubsection{Subordinated sub-Laplacian with Dirichlet boundary condition} Let us consider the Dirichlet form $\mathcal{E}^\phi(\cdot, \cdot)$ associated with the subordinated semigroup $(P^\phi_t)_{t\ge 0}$ defined in \eqref{eq:subordinated_semigroup}, that is, 
\begin{align}\label{eq:subordinated_DF}
    \mathcal{E}^\phi(f,g)=\langle(\Delta^\phi_{\mathcal H})^{1/2} f, (\Delta^\phi_{\mathcal H})^{1/2}g\rangle_{L^2(\mathbb G, dx)} \quad \mbox{for all $f\in \mathcal{D}((\Delta^\phi_{\mathcal H})^{1/2})$}.
\end{align}
In this case, the domain of Dirichlet form is also given by 
\begin{align*}
\mathcal{D}(\mathcal E^\phi)=\mathcal{D}\left(\left(\Delta^\phi_\mathcal{H}\right)^{\frac12}\right)=\left\{f\in L^2(\mathbb G, dx): \int_{[0,\infty)} \phi(\lambda)d\mu_f(\lambda)<\infty\right\}.
\end{align*}
By Lemma~\ref{lem:form_core_1}, we have $C^\infty_c(\mathbb G)\subset \mathcal{D}(\mathcal E^\phi)$. Hence, $\mathcal{E}^\phi$ is a regular Dirichlet form.
Moreover, $C^\infty_c(\mathbb G)$ is a form core for $\mathcal{E}^\phi$ by Lemma~\ref{lem:form_core_1}. For any open subset $\Omega\subset \mathbb G$, let $(\mathcal{E}^{\phi,\Omega},(\mathcal{F}^\phi)^\Omega)$ denote the \emph{part of the Dirichlet form on $\Omega$}, with
\begin{align*}
    (\mathcal{F}^\phi)^\Omega=\{f\in\mathcal{D}(\mathcal{E}^\phi): \widetilde{f}=0 \quad \mbox{$\mathcal{E}^\phi$-q.e. on $\mathbb{G}\setminus \Omega$}\},
\end{align*}
where $\widetilde{f}$ denotes a quasi-continuous version of $f$,
and we refer to \cite[p.~173]{Fukushima_et_alBook} for details. Then by \cite[Theorem~4.4.3(i)]{Fukushima_et_alBook}, $(\mathcal{E}^{\phi,\Omega}, (\mathcal{F}^\phi)^\Omega)$ is a regular Dirichlet form on $L^2(\Omega, dx)$.
\begin{definition}\label{def:dirichlet_subL}
    For any open set $\Omega\subset \mathbb G$, the generator associated with the Dirichlet form $\mathcal{E}^{\phi,\Omega}$, denoted by $\Delta^{\phi,\Omega}_{\mathcal H}$, is called the \emph{subordinated sub-Laplcian with Dirichlet boundary condition on $\Omega$}.
\end{definition}
\begin{remark}
    The space of smooth functions $C^\infty_c(\Omega)$ is a form core for $\mathcal{E}^{\phi,\Omega}$. As a result, $\Delta^{\phi,\Omega}_{\mathcal H}$ is given by the Friedrichs extension of $(\mathcal{L}^{\phi}_{\mathcal H}, C^\infty_c(\Omega))$, where for any $f\in C^\infty_c(\Omega)$,
    \begin{align*}
        \mathcal{L}^\phi_\Omega f=\left.\Delta^\phi_{\mathcal H} f^0\right|_{\Omega},
    \end{align*}
    with $f^0$ being the zero extension of $f$ on $\mathbb G$.
    We refer to Lemma~\ref{lem:form_core_2} for details.
\end{remark}
\begin{remark}
    The operator $\Delta^{\phi,\Omega}_{\mathcal H}$ should not be confused with $\phi(\Delta^\Omega_{\mathcal H})$, where $\Delta^{\Omega}_{\mathcal H}$ is the sub-Laplacian with Dirichlet boundary condition on $\Omega$. In probabilistic terms, the first operator (up to a change of sign) is the generator of a killed subordinated Markov process (subordinated first, then killed), whereas the latter (up to a change of sign) is the generator of a subordinated killed (killed first, then subordinated) Markov process, see Section~\ref{s:probability} for details.
\end{remark}

Motivated by the work of Chen and Song \cite{ChenSong2005} relating eigenvalues of subordinated Laplacian with Dirichlet boundary condition on Euclidean spaces, we relate the eigenvalues of subordinated Dirichlet sub-Laplacian to the spectrum of the Dirichlet sub-Laplacian on Carnot groups. Unlike in the Euclidean case, the sub-Laplacian on Carnot groups is not elliptic, and the underlying space is not a linear space. 
 
 Some regularity assumptions on the boundary of the domain are needed to prove the eigenvalue comparison results even in the Euclidean case. One such natural condition, often assumed in the mathematical literature is the \emph{exterior cone condition}, which assumes that there exists a cone $K\subset \R^n$ centered at the origin and $r_0>0$ such that for each point $p$ on the boundary of $\Omega$, there is a  cone $K_p\subset \R^n$ with the vertex at $p$ and isometric to $K$, and a ball $B(p,r_0)$ centered at $p$ and having radius $r_0$ that satisfy 
 \begin{align*}
 K_p\cap B(p,r_0)\subseteq \Omega^c.
 \end{align*}

 For Carnot groups, even defining such cones involves subtleties of sub-Riemannian geometry of the group. Moreover, the class of isometric isomorphisms on Carnot groups is quite restrictive compared to Euclidean spaces, and it is not feasible to extend the notion of exterior cone condition from Euclidean spaces to Carnot groups. We refer to Section~\ref{s:cones} for details about the cones and exterior cone conditions on Carnot groups. One of our main results is stated below.

 \begin{theorem}\label{thm:eigenvalue_comparison}
Let $\phi$ be a Bernstein function such that 
\begin{align}\label{eq:Bernstein_cond}
    \int_0^\infty u^{\frac{Q}{2}-1} e^{-t\phi(u)}du<\infty
\end{align}
for some $t>0$. Then,
\begin{enumerate}[leftmargin=*]
\item $\Delta^{\phi,\Omega}_{\mathcal{H}}$ has purely point spectrum in $L^2(\Omega, dx)$ and the first eigenvalue is strictly larger than $\phi(0)$. Moreover, denoting the eigenvalues of $\Delta^{\phi,\Omega}_{\mathcal H}$ by 
\begin{align*}
\lambda_{1,\phi}\le \lambda_{2,\phi}\le \cdots,
\end{align*}
we have
\begin{align} \label{ineq.UpperBd.Eigenvalue_1}
    \lambda_{n,\phi}\le 4\phi(\lambda_n) \quad \mbox{for all $n$},
\end{align}
where $0<\lambda_1\le \lambda_2\le \cdots$ denote the eigenvalues of the Dirichlet sub-Laplacian $\Delta^{\Omega}_{\mathcal H}$. 
\item If $\phi$ is a complete Bernstein function, then 
\begin{align} \label{ineq.UpperBd.Eigenvalue}
    \lambda_{n,\phi}\le \phi(\lambda_n) \quad \mbox{for all $n$}.
\end{align}
If in addition $\Omega$ satisfies the intrinsic exterior cone condition (see Definition~\ref{def:cone_cond}), then 
\begin{align} \label{ineq.LowerBd.Eigenvalue}
    \lambda_{n,\phi}\ge c(\Omega)\phi(\lambda_n) \quad \mbox{for all $n$},
\end{align}
where $c(\Omega)>0$ is a constant independent of $n$ and $\phi$, and depends only on $\Omega$.
\end{enumerate}
\end{theorem}

\begin{remark}
 The bound  \eqref{eq:Bernstein_cond} is satisfied whenever the drift coefficient $b$ in \eqref{eq:Bernstein} is strictly positive, or $\phi(u)\gtrsim \log u$ as $u\to\infty$. In particular, \eqref{eq:Bernstein_cond} holds for $\phi(u)=u^\beta$ for $0<\beta<1$.
\end{remark}
\begin{remark}
In the special case when $\phi(u)=u^{\beta}$, $0<\beta<1$, the above upper and lower bound for the eigenvalues $\lambda_{n,\phi}$ with explicit constants can be found in \cite[Corollary~1.1, Theorem~1.3]{ChenChenLi2026}. Moreover, the authors proved such estimates for domains $\Omega$ satisfying the condition
\begin{align}\label{eq:sup_boundary}
    \sup_{r>0}\frac{|\left\{x\in\Omega: d(x,\partial\Omega)<r\right\}|}{r}<\infty.
\end{align}
While it is difficult to adapt the argument in \cite{ChenChenLi2026} for arbitrary Bernstein functions $\phi$, our proof uses a probabilistic argument motivated by \cite{ChenSong2005}. Our approach requires an exterior cone condition on the boundary of the domain, which is different from \eqref{eq:sup_boundary}. In particular, theorem \ref{thm:eigenvalue_comparison} holds for all bounded open sets with $C^1$ boundary with no characteristic points, see Theorem~\ref{thm:noncharacteristic-cones}.
\end{remark}

\subsection{Heat content of fractional sub-Laplacian} \label{sec.HeatContent}
Let $\Omega$ be a bounded open subset of $\mathbb G$. The heat content
of $\Omega$ associated with the subordinated sub-Laplacian
$\Delta_{\mathcal H}^{\phi}$ is defined by
\begin{align*}
\widetilde H_\Omega^\phi(t)
:=
\int_\Omega
e^{-t\Delta_{\mathcal H}^{\phi,\Omega}}
\mathbbm 1_\Omega(x)\,dx
=
\left\langle
e^{-t\Delta_{\mathcal H}^{\phi,\Omega}}\mathbbm 1_\Omega,
\mathbbm 1_\Omega
\right\rangle_{L^2(\Omega, dx)}.
\end{align*}
Here $\Delta_{\mathcal H}^{\phi,\Omega}$ denotes the subordinated sub-Laplacian with Dirichlet boundary condition introduced in Definition~\ref{def:dirichlet_subL}. By the standard theory of strongly continuous semigroups of linear
operators, see, for instance, \cite[p.~436]{EngelNagelBook2000}, the
function
\begin{align*}
u_\phi(t)
:=
e^{-t\Delta_{\mathcal H}^{\phi,\Omega}}\mathbbm 1_\Omega
\end{align*}
belongs to $C([0,\infty);L^2(\Omega, dx))$ and is the unique mild solution
of the abstract Cauchy problem
\begin{equation}\label{eq:heat_eq_Dirichlet}
\begin{cases}
\partial_tu_\phi(t)
=
-\Delta_{\mathcal H}^{\phi,\Omega}u_\phi(t),
& t>0,\\ 
u_\phi(0)=\mathbbm 1_\Omega,
& \text{in }L^2(\Omega, dx).
\end{cases}
\end{equation}
Moreover, for every $t>0$, the function $u_\phi(t)$ belongs to
$\mathcal D(\Delta_{\mathcal H}^{\phi,\Omega})$, and hence the above
equation holds in $L^2(\Omega, dx)$ in the strong sense.

Equivalently, if $\widetilde u_\phi(t,\cdot)$ denotes the extension of
$u_\phi(t,\cdot)$ by zero to $\mathbb G$, the corresponding non-local
Dirichlet problem can be formally written as
\begin{equation}\label{eq:heat_eq_exterior_Dirichlet}
\begin{cases}
\partial_t\widetilde{u}_\phi(t,x)
=
-\Delta_{\mathcal H}^{\phi}\widetilde u_\phi(t,x),
& (t,x)\in(0,\infty)\times\Omega,\\
\widetilde u_\phi(t,x)=0,
& (t,x)\in(0,\infty)\times(\mathbb G\setminus\Omega),\\
u_\phi(0,x)=1,
& x\in\Omega.
\end{cases}
\end{equation}
The last formulation is understood in the appropriate weak or
$L^2$ sense unless additional regularity assumptions are imposed.
We are interested in understanding the large-time and small-time behavior of $\widetilde{H}^\phi_\Omega(t)$. While the large-time asymptotics would follow from the spectral theory of the subordinated sub-Laplacian on bounded domains, the derivation of small-time asymptotics would require some geometric conditions on the domain $\Omega$. Our first result states the following:
\begin{theorem}\label{thm:heat_content_large}
    Assume that $\phi$ satisfies \eqref{eq:Bernstein_cond} and $\Omega\subset\mathbb{G}$ is a bounded, open, connected subset. Then the first eigenvalue $\lambda_{1,\phi}$ of $\Delta^{\phi,\Omega}_{\mathcal H}$ is simple and admits an eigenfunction $f_\phi$ such that $f_\phi(x)>0$ for all $x\in\Omega$, and $\|f_\phi\|_{L^2(\mathbb G, dx)}=1$. Moreover, 
    \begin{align*}
        \lim_{t\to \infty} e^{t\lambda_{1,\phi}}\widetilde{H}^\phi_\Omega(t)=\langle f_\phi,\mathbbm{1}_\Omega\rangle^2_{L^2(\Omega,dx)}.
    \end{align*}
\end{theorem}

For the small-time asymptotics of the heat content, we only focus on the case when $\phi(u)=u^{\alpha/2}$, where $0<\alpha<2$. In this case, the subordinated sub-Laplacian is equal to the fractional sub-Laplacian, and we denote the corresponding heat content by $\widetilde{H}^{(\alpha)}_\Omega(t)$, and we call it the \emph{fractional heat content}. For the small-time asymptotic behavior of $\widetilde{H}^{(\alpha)}_\Omega(t)$, we require the following notions of horizontal perimeter and fractional horizontal perimeter.
\begin{definition}[Horizontal perimeter]
    For a measurable set $\Omega\subset \mathbb{G}$, its \emph{horizontal perimeter} is defined by
\begin{align*}
    \Per_{\mathcal H}(\Omega)=\sup\left\{\int_{\Omega}\sum_{i=1}^{d} X_i \rho_i: \|\sum_{i=1}^d \rho^2_i\|_\infty\le 1, \rho_i\in C_c^\infty(\mathbb G)\right\},
\end{align*}
where $\{X_1, \ldots, X_{d}\}$ is an orthonormal frame in $V_1$. A measurable set $\Omega\subset \mathbb{G}$ is called a \emph{Caccioppoli set} if $\Per_{\mathcal H}(\Omega)<\infty$.
\end{definition}
\begin{definition}[Fractional horizontal perimeter]\label{def:fract_horiz_per}
 For $\alpha\in (0,1)$, the $\alpha$-\emph{fractional horizontal perimeter} of a measurable set $\Omega$ is defined as 
 \begin{align*}
     \Per^{(\alpha)}_\mathcal{H}(\Omega)=\int_{\Omega}\int_{\Omega^c} \frac{1}{\|x^{-1}y\|_\alpha}dx dy,
 \end{align*}
 where $\|\cdot\|_\alpha$ is a homogeneous quasi-norm defined by
 \begin{align}
    \|x\|_\alpha=(\widetilde{R}_\alpha(x))^{-\frac{1}{\alpha+Q}}, \quad \widetilde{R}_\alpha(x)=\frac{\alpha}{2\Gamma\left(1-\frac\alpha 2\right)}\int_0^\infty t^{-\frac{\alpha}{2}-1} p(t,x) dt
\end{align}
with $p(t,\cdot)$ being the hypoelliptic heat kernel defined in \eqref{eq:heat_kernel}.
\end{definition}
The homogeneous quasinorm $\|\cdot\|_\alpha$ was originally introduced by Ferrari and Franchi \cite{FerrariFranchi2015}. This is equivalent to any homogeneous quasinorm, that is, for any homogeneous quasinorm $\|\cdot\|$ on $\mathbb G$, there exists a constant $c>0$ such that 
\begin{align}\label{eq:norm_equiv}
    c^{-1} \|x\| \le \|x\|_\alpha\le c \|x\|,
\end{align}
see \cite[Eq. (1.9)]{FFMPPS2028} for details.
Moreover, for any $u\in\mathcal{S}(\mathbb G)$, the Schwartz space on $\mathbb G$ and $\alpha\in (0,2)$ one has 
\begin{align*}
    \Delta_{\mathcal H}^{\alpha/2} u(x)=\mathrm{p.v.} \int_{\mathbb G} \frac{u(x)-u(y)}{\|x^{-1} y\|_\alpha^{Q+\alpha}} dy.
\end{align*}
When $\alpha\in [1,2)$, the following small time asymptotics for the heat content of fractional sub-Laplacian has been derived in \cite{Sarkar2026}, assuming that the boundary of the domain is non-characteristic (see Definition~\ref{def:characteristic_pt}).
\begin{theorem}[Theorem~1.1 in \cite{Sarkar2026}]\label{thm:spectral_heat1}
   Let $\Omega$ be a bounded, open subset of $\mathbb{G}$. If $\Omega$ has $C^2$ boundary with no characteristic points, then
    \begin{enumerate}[leftmargin=*]
        \item \label{sh1} For $1<\alpha<2$, 
        \begin{align*}
            \lim_{t\to 0}\frac{\Vol(\Omega)-\widetilde{H}^{(\alpha)}_\Omega(t)}{t^{\frac{1}{\alpha}}}= \mathbb{E}\left[\sup_{0\le t\le 1} Y_t\right]\Per_{\mathcal H}(\Omega),
        \end{align*}
        where $(Y_t)_{t\ge 0}$ is a one-dimensional symmetric $\alpha$-stable L\'evy process such that $\mathbb{E}[e^{i \xi Y_t}]=e^{-t|\xi|^\alpha}$ for all $t\ge 0$ and $\xi\in\mathbb R$.
        \item \label{sh2} When $\alpha=1$, 
        \begin{align*}
            \lim_{t\to 0}\frac{\Vol(\Omega)-\widetilde{H}^{(1)}_\Omega(t)}{t\log(1/t)}= \frac{1}{\pi}\Per_{\mathcal H}(\Omega).
        \end{align*}
    \end{enumerate}
\end{theorem}
As explained before, the technique used in \cite{Sarkar2026} to prove the above theorem does not extend to the case when $\alpha\in (0,1)$. In the following result,  we settle the question of the existence of limits when $\alpha\in (0,1)$. For this, we require a different regularity condition on the boundary, known as the volume density condition introduced in Definition~\ref{def:VDC}. We refer to Section~\ref{s:VDC} for more details.
\begin{theorem}\label{thm:spectral_heat}
 Let $\Omega\subset \mathbb G$ be bounded, open set satisfying the volume density condition (see Definition~\ref{def:VDC}). Then for all $0<\alpha<1$, 
        \begin{align}\label{eq:lim_0_1}
            \lim_{t\to 0}\frac{\Vol(\Omega)-\widetilde{H}^{(\alpha)}_\Omega(t)}{t}=\Per^{(\alpha)}_{\mathcal H}(\Omega)
        \end{align}
        whenever $\Per^{(\alpha)}_{\mathcal H}(\Omega)<\infty$.
         In particular, \eqref{eq:lim_0_1} holds for all bounded domains with $C^1$ boundary that does not have any characteristic points.
\end{theorem}
\begin{remark}
    The volume density condition is flexible enough to include various domains whose boundaries are non-smooth or contain characteristic points. We refer to Proposition~\ref{prop:VDC_ball} and Examples~\ref{ex:1}--\ref{ex:3} for details.
\end{remark}

\subsection{Relative heat content of subordinated sub-Laplacian} For a bounded open subset $\Omega\subset \mathbb{G}$, its relative heat content corresponding to the fractional sub-Laplacian $\Delta^{\alpha/2}_\mathcal{H}$ is defined as 
\begin{align*}
    H^{(\alpha)}_\Omega(t)=\int_{\Omega} v_\alpha(t,x)dx,
\end{align*}
where $v_\alpha(t,x)=e^{-t\Delta^{\alpha/2}_\mathcal{H}}\mathbbm{1}_\Omega(x)$ is the mild solution to the Cauchy problem
\begin{align*}
    \partial_tv_\alpha(t,x)&=-\Delta^{\alpha/2}_\mathcal{H} v_\alpha(t,x) \\
    v_\alpha(0,x)&=\mathbbm{1}_\Omega(x).
\end{align*}
When $\alpha=2$, we denote the relative heat content by $H_\Omega(t)$. The small time asymptotics of the relative heat content for compact sub-Riemannian manifolds have been recently obtained by Agrachev, Rizzi, and Rossi \cite{AgrachevRizziRossi2024}. Their method applies to Carnot groups, and one has the following asymptotic expansion
\begin{align*}
    H_\Omega(t)=\Vol(\Omega)-\sqrt{\frac{t}{\pi}}\Per_\mathcal{H}(\Omega)+ O(t^{3/2})
\end{align*}
whenever $\Omega$ has $C^\infty$ boundary with no characteristic points. In the next theorem, we generalize the above asymptotics for the fractional relative heat content.
\begin{theorem}\label{thm:heat_content}
     Let $\Omega$ be a bounded set in $\mathbb{G}$ whose $C^\infty$ boundary $\partial \Omega$ does contain no characteristic points. Then,  the following holds.
     \begin{enumerate}[leftmargin=*]
     \item\label{it:h1} For $1<\alpha< 2$,
     \begin{align*}
         \lim_{t\to 0} \frac{\Vol(\Omega) - H^{(\alpha)}_\Omega(t)}{t^{\frac{1}{\alpha}}}=\frac{1}{\pi}\Gamma\left(1-\frac{1}{\alpha}\right)\Per_\mathcal{H}(\Omega).
     \end{align*}
     \item\label{it:h2} When $\alpha=1$,
     \begin{align*}
         \lim_{t\to 0} \frac{\Vol(\Omega) - H^{(\alpha)}_\Omega(t)}{t\log(1/t)}=\frac{1}{\pi}\Per_\mathcal{H}(\Omega).
     \end{align*}
     \end{enumerate}
    For $0<\alpha<1$ and for any bounded open set $\Omega$ satisfying $\Per^{(\alpha)}_{\mathcal H}(\Omega)<\infty$ we have
    \begin{align*}
        \lim_{t\to 0} \frac{\Vol(\Omega) - H^{(\alpha)}_\Omega(t)}{t}=\Per^{(\alpha)}_{\mathcal H}(\Omega).
    \end{align*}
\end{theorem}
\begin{remark}
    When $\mathbb{G}$ is a Carnot group of step 2 and $\alpha=2$, Garofalo and Tralli \cite{GarofaloTralli2023} proved the heat content asymptotics for any bounded open set with finite perimeter. By the same technique used in the proof of the Theorem~\ref{thm:heat_content}, one can prove \eqref{it:h1} and \eqref{it:h2} for any bounded open set having finite perimeter provided that $\mathbb{G}$ is a Carnot group of step 2. It still remains an open question whether the heat content asymptotics hold (even when $\alpha=2$) for bounded open sets of general Carnot groups without assuming any regularity conditions on the boundary.
\end{remark}
We also note that as an interesting consequence of \cite[Theorem~1.1]{AgrachevRizziRossi2024}, one obtains the following version of the Bourgain-Brezis-Mironescu-D\'avila theorem for Carnot groups which was originally proved by Garofalo and Tralli \cite{GarofaloTralli2023} for step-2 Carnot groups.
\begin{corollary}\label{cor:BBMD}
    Let $\Omega\subset \mathbb{G}$ be a bounded domain with smooth boundary with no characteristic points. Then, 
    \begin{align*}
    \lim_{\alpha\nearrow 1} (1-\alpha) \Per^{(\alpha)}_\mathcal{H}(\Omega)=\frac{1}{\pi}\Per_\mathcal{H}(\Omega).
    \end{align*}
    Also, for any measurable set $\Omega$ with $\Vol(\Omega)<\infty$ and $\Per_{\mathcal{H}}(\Omega)<\infty$,
    \begin{align*}
    \lim_{\alpha\searrow 0} \Per^{(\alpha)}_{\mathcal H}(\Omega)=\Vol(\Omega).
    \end{align*}
\end{corollary}
\begin{remark}
The normalization used here differs from the heat-Besov normalization
of Garofalo and Tralli \cite{GarofaloTralli2023}. If
$\mathfrak P_{H,s}$ denotes their nonlocal horizontal perimeter, then
\begin{align*}
\Per_{\mathcal H}^{(\alpha)}(\Omega)
=
\frac{\alpha}{4\Gamma(1-\alpha/2)}
\mathfrak P_{H,\alpha/2}(\Omega).
\end{align*}
Hence, their limit formula
\begin{align*}
\lim_{s\nearrow1/2}
(1-2s)\mathfrak P_{H,s}(\Omega)
=
\frac{4}{\sqrt{\pi}}\Per_{\mathcal H}(\Omega)
\end{align*}
is equivalent to
\begin{align*}
\lim_{\alpha\nearrow1}
(1-\alpha)\Per_{\mathcal H}^{(\alpha)}(\Omega)
=
\frac{1}{\pi}\Per_{\mathcal H}(\Omega).
\end{align*}
Their result applies to finite-perimeter sets in step-two Carnot
groups, while ours holds in arbitrary Carnot groups under smoothness
and non-characteristic assumptions on the boundary.
\end{remark}

\section{Probabilistic preliminaries: hypoelliptic Brownian motion and subordination}\label{s:probability}
\subsection{Hypoelliptic Brownian motion} The negative of the sub-Laplacian $-\Delta_{\mathcal H}$ defined in \eqref{eq:sub-Laplacian} is the generator of a Markov process $B=(B_t)_{t\ge 0}$ on $\mathbb{G}$. This is known as the \emph{hypoelliptic Brownian motion} on $\mathbb G$, which satisfies the following properties analogous to those of Brownian motion in Euclidean spaces:
\begin{enumerate}[leftmargin=*]
\item For any $s\ge 0$, the processes $\left(B^{-1}_s B_{t+s}\right)_{t\ge 0}$ and $(B_t)_{t\ge 0}$ have the same law, and are independent.
\item The sample paths of $B$ are almost surely continuous.
\end{enumerate}
With the identification $\mathbb{G}\cong \mathbb{R}^{m_1}\times\cdots\times \mathbb{R}^{m_k}$, the hypoelliptic Brownian motion $B_t$ can be realized as the solution to the following Stratonovich stochastic differential equation:
\begin{align*}
    dB_t &=\sum_{i=1}^m X_i(B_t) \circ dW^{(i)}_t, \ X_0=x,
\end{align*}
where $W^{(1)}, \ldots, W^{(m)}$ are $m$ independent Brownian motions on $\mathbb{R}$ with $\mathrm{Var}(W^{(i)}_t)=2t$ for each $i=1,\ldots, m$. The heat kernel $p(t,x,y):=p(t,x^{-1} y)$ introduced in \eqref{eq:heat_kernel} coincides with the transition density of $B_t$, and for any bounded measurable function $f$ on $\mathbb G$, the heat semigroup $(P_t)_{t\ge 0}$ can be realized as 
\begin{align*}
    P_t f(x) =\mathbb{E}_x\left[f(B_t)\right] \quad \mbox{for all $x\in\mathbb G, \ t\ge 0$}.
\end{align*}

\subsection{Subordination}\label{s.Subordination}

Let $\phi$ be the Bernstein function introduced in
\eqref{eq:Bernstein}. There exists a convolution semigroup
$(\eta_t)_{t\geq0}$ of sub-probability measures on $[0,\infty)$ such
that
\begin{align}\label{eq:LK_subordinator}
\int_{[0,\infty)}e^{-us}\,\eta_t(ds)
=
e^{-t\phi(u)}
\qquad
\text{for every }t,u\geq0.
\end{align}
Moreover,
\begin{align*}
\eta_t([0,\infty))
=
e^{-t\phi(0)}
=
e^{-at}.
\end{align*}
In particular, if $a=0$, then each $\eta_t$ is a probability measure
and there exists a conservative, translation invariant, nonnegative and increasing Markov process
$S=(S_t)_{t\geq0}$ such that
\begin{align*}
\eta_t(ds)=\mathbb P(S_t\in ds)
\end{align*}
and
\begin{align*}
\mathbb E\left[e^{-uS_t}\right]
=
e^{-t\phi(u)}.
\end{align*}
Such Markov processes are known as \emph{subordinators}.
If $a>0$, the corresponding subordinator is killed at rate $a$. We
refer to \cite{Sato_Book} for the general theory of subordinators. By \cite[Theorem~31.5]{Sato_Book},  the generator of $S_t$ in $C_0([0,\infty))$ is given by 
\begin{equation}\label{eq:subordinator_generator}
\begin{aligned}
    &\mathcal{A}_\phi u(r)=-au(r)+b u'(r)+\int_0^\infty (u(s+r)-u(r))\nu(ds), \ \mbox{and} \\
    &C^2_0([0,\infty)):=\{u\in C^2([0,\infty)): u, u', u'' \ \mbox{vanish at $0$ and $\infty$}\}\subseteq \mathcal{D}(\mathcal{A}_\phi)
\end{aligned}
\end{equation}
where $a,b,\nu$ are given by \eqref{eq:Bernstein}.

Let $B=(B_t)_{t\geq0}$ be the hypoelliptic Brownian motion on
$\mathbb G$, whose transition semigroup is
\begin{align*}
P_t=e^{-t\Delta_{\mathcal H}},
\end{align*}
and assume that $B$ is independent of the subordinator. When $a=0$,
define the time-changed process
\begin{align}\label{eq:time-change}
B_t^\phi:=B_{S_t}.
\end{align}
When $a>0$, the same definition is used before the lifetime of the
subordinator, after which the process is sent to a cemetery state.
The transition semigroup of the subordinated process is given by
\begin{align*}
P_t^\phi f(x)
:=
\mathbb E_x
\bigl[f(B_t^\phi);\,t<\zeta_S\bigr],
\end{align*}
where $\zeta_S$ denotes the lifetime of the subordinator. By
conditioning with respect to the time change and using independence,
we obtain
\begin{align}\label{eq:Bochner_subordination}
P_t^\phi f
=
\int_{[0,\infty)}P_sf\,\eta_t(ds).
\end{align}
For $f\in L^2(\mathbb G,dx)$, the integral in
\eqref{eq:Bochner_subordination} is understood in the Bochner sense.
The semigroup $(P_t^\phi)_{t\geq0}$ is a strongly continuous
self-adjoint contraction semigroup on $L^2(\mathbb G,dx)$. By the
Bochner-Phillips functional calculus (see \cite{Phillips1952}),
\begin{align*}
P_t^\phi
=
e^{-t\phi(\Delta_{\mathcal H})}
=
e^{-t\Delta_{\mathcal H}^\phi},
\end{align*}
where $\Delta_{\mathcal H}^\phi$ is the non-negative self-adjoint
operator defined in \eqref{eq:subordinated_subL}. Thus, the
infinitesimal generator of $(P_t^\phi)_{t\geq0}$ is
$
-\Delta_{\mathcal H}^\phi.
$
From \eqref{eq:Bochner_subordination} it follows that the Haar measure satisfies
\begin{align*}
\int_{\mathbb G}P_t^\phi f(x)\,dx
=
e^{-at}
\int_{\mathbb G}f(x)\,dx
\end{align*}
for every $f\in L^1(\mathbb G,dx)$. Consequently, when $a=0$, the
Haar measure is invariant and the subordinated semigroup is
conservative. The process $B^\phi$ defined by \eqref{eq:time-change} is a possibly killed L\'evy process on
$\mathbb G$, that is,
\begin{enumerate}[leftmargin=*]
\item for every $s\geq0$, the process
\begin{align*}
\left(
(B_s^\phi)^{-1}B_{s+t}^\phi
\right)_{t\geq0}
\end{align*}
has the same distribution as $(B_t^\phi)_{t\geq0}$ and is independent
of the natural filtration up to time $s$, and 

\item the mapping
$t\mapsto B_t^\phi$ is continuous in probability. 
\end{enumerate}
We refer to
\cite{LiaoMingBookLevyProcessesinLieGroups} for a detailed account of the theory L\'evy processes on Lie groups.

When the integrability condition \eqref{eq:Bernstein_cond} holds, the transition measure of $B^\phi$ starting from the identity admits
the decomposition
\begin{align*}
P_t^\phi(e,dx)
=
\eta_t(\{0\})\delta_e(dx)
+
p^\phi(t,x)dx,
\end{align*}
where
\begin{align}\label{eqn.SubordinatedHK}
p^\phi(t,x)
:=
\int_{(0,\infty)}p(s,x)\,\eta_t(ds).
\end{align}
In particular, if $\eta_t(\{0\})=0$, then $p^\phi(t,\cdot)$ is the
transition density of $B^\phi$ with respect to the Haar measure.

If $b=0$, then, apart from the possible killing, $B^\phi$ is a
pure-jump process. If $b>0$, it also has a diffusion component with
generator $-b\Delta_{\mathcal H}$. According to
\cite[p.~16]{LiaoMingBookLevyProcessesinLieGroups} (see also \cite{AlbeverioGordina2007} for matrix Lie groups), the L\'evy
measure $\Pi_\phi$ of $B^\phi$ is characterized by
\begin{align*}
\lim_{t\searrow0}
\frac{P_t^\phi f(e)}{t}
=
\int_{\mathbb G\setminus\{e\}}
f(x)\,\Pi_\phi(dx)
\end{align*}
for every $f\in C_c^\infty(\mathbb G\setminus\{e\})$.
\begin{lemma}
    For any Borel set $A\subset \mathbb{G}\setminus \{e\}$, 
    \begin{align*}
        \Pi_\phi(A)=\int_{0}^\infty\left(\int_{A} p(s,x)dx\right)\nu(ds).
    \end{align*}
\end{lemma}
\begin{proof}
By \cite[Theorem~IV.4.2]{VaropoulosSaloff-CosteCoulhonBook1992} it is known that the function $t\mapsto p(t,x)$ is smooth for each $x\in\mathbb G$, and there exists $C>0$ such that
\begin{align}\label{eq:heat_deriv_bound}
    \left|\frac{\partial^m}{\partial t^m} p(t,x)\right|\le C t^{-m-\frac{Q}{2}} e^{-\frac{d_c(x)^2}{8t}} \quad \mbox{for all $x\in \mathbb G$},
\end{align}
where $d_c(x)$ is the Carnot-Carath\'eodory distance between the identity element $e$ and $x$.
Hence for any $f\in C^\infty_c(\mathbb G\setminus \{e\})$ the function
\begin{align*}
    t\longmapsto P_t f(e)=\int_{\mathbb G} f(x) p(t,x)dx\in C^\infty([0,\infty)),
\end{align*}
and all derivatives of $t\mapsto P_t f(e)$ vanish at $t=0$ and $t=\infty$. By \eqref{eq:subordinator_generator}, $t\mapsto P_t f(e)\in \mathcal{D}(\mathcal{A}_\phi)$. As a result, 
\begin{align*}
    \lim_{t\to 0}\frac{P^\phi_t f(e)}{t}=\lim_{t\to 0}\frac{\mathbb{E}_0[P_{S_t} f(e)]- P_0 f(e)}{t}
    =\int_0^\infty P_s f(e) \nu(ds).
\end{align*}
Since 
\begin{align*}
    \int_0^\infty P_s f(e) \nu(ds)=\int_0^\infty \int_{\mathbb G} f(x) p(s,x)dx\nu(ds),
\end{align*}
we conclude that $\Pi_\phi(A)=\int_{0}^\infty \int_{A}p(s,x)\nu(ds)$. This concludes the proof of the lemma.
\end{proof} 

Let $\Omega\subset\mathbb G$ be a bounded open set, and let
$\Delta_{\mathcal H}^{\phi,\Omega}$ be the non-negative subordinated
sub-Laplacian with exterior Dirichlet condition introduced in
Definition~\ref{def:dirichlet_subL}. Define
\begin{align*}
\tau_\Omega
:=
\inf\{t>0:B_t^\phi\notin\Omega\}
\end{align*}
and
\begin{align*}
B_t^{\phi,\Omega}
:=
\begin{cases}
B_t^\phi,
& t<\tau_\Omega,\\
\partial,
& t\geq\tau_\Omega,
\end{cases}
\end{align*}
where $\partial$ denotes the cemetery state.

The transition semigroup of the killed process is
\begin{align*}
P_t^{\phi,\Omega}f(x)
=
\mathbb E_x
\left[
f(B_t^\phi);\,t<\tau_\Omega
\right]
=
e^{-t\Delta_{\mathcal H}^{\phi,\Omega}}f(x).
\end{align*}
Consequently, the infinitesimal generator of the killed process is $-\Delta_{\mathcal H}^{\phi,\Omega}.$
In particular,
\begin{align*}
e^{-t\Delta_{\mathcal H}^{\phi,\Omega}}
\mathbbm{1}_\Omega(x)
=
\mathbb P_x(\tau_\Omega>t),
\qquad
x\in\Omega.
\end{align*}
Hence, the Dirichlet heat content of $\Omega$ associated with the
subordinated sub-Laplacian is
\begin{align*}
\widetilde H_\Omega^\phi(t)
=
\int_\Omega
\mathbb P_x(\tau_\Omega>t)\,dx.
\end{align*}
Similarly, the relative heat content associated with the global
subordinated process is
\begin{align*}
H_\Omega^\phi(t)
=
\int_\Omega
\mathbb P_x(B_t^\phi\in\Omega)\,dx
=
\left\langle
P_t^\phi\mathbbm 1_\Omega,
\mathbbm 1_\Omega
\right\rangle_{L^2(\mathbb G,dx)}.
\end{align*}
From the above representations it follows that for all $t\ge 0$,
\begin{align*}
    0\le \widetilde{H}^\phi_\Omega(t)\le H^\phi_\Omega(t).
\end{align*}

\subsection{Stable subordinators and fractional sub-Riemannian heat kernel}
As noted before in \eqref{eq:Bernstein_alpha}, for any $\alpha\in (0,2)$, the function $\phi_\alpha(u)=u^{\alpha/2}$ is a Bernstein function. The unique L\'evy subordinator $S^{(\alpha)}=(S^{(\alpha)}_t)_{t\ge 0}$ corresponding to $\phi_\alpha$ is also a stable process, that is, for any $t>0$,
\begin{align}\label{eq:self-similarity}
        S^{(\alpha)}_t\overset{d}{=} t^{\frac{2}{\alpha}} S^{(\alpha)}_1.
    \end{align}
Because of the above scaling property, $S^{(\alpha)}$ is called a \emph{stable subordinator}. Therefore, the time-changed hypoelliptic Brownian motion 
\begin{align*}
B^{(\alpha)}_t := B_{S^{(\alpha)}_t}
\end{align*}
is a Markov process with generator given by $-\Delta_{\mathcal H}^{\alpha/2}$.
\begin{notation}
Throughout the article, we denote the Markov semigroup corresponding to $B^{(\alpha)}_t$ by $P^{(\alpha)}_t$, and its transition kernel by $p^{(\alpha)}(t,x)$ for any $t>0$ and any $x\in \mathbb G$.
\end{notation}
\begin{definition}\label{def:frac_heat_kernel}
    The transition kernel $p^{(\alpha)}(t,\cdot)$ is called the \emph{fractional hypoelliptic heat kernel}.
\end{definition}
We end this section with an upper bound and asymptotics for the fractional hypoelliptic heat kernel.
\begin{proposition}\label{prop:heat_asymp}
For $0<\alpha<2$, there exists a constant $c>0$ such that for all $x\in \mathbb G$ and $t>0$,
\begin{align}\label{eq:frac_heat_bound}
     p^{(\alpha)}(t,x)\le  \frac{ct}{\|x\|^{Q+\alpha}_\alpha}.
\end{align}
Moreover, for all $x\in\mathbb{G}$,
\begin{align}\label{eq:heat_asymp}
    \lim_{t\to 0} \frac{p^{(\alpha)}(t,x)}{t}=\frac{1}{\|x\|^{Q+\alpha}_\alpha}.
\end{align}
\end{proposition}

\begin{proof}
    Using the two-sided heat kernel estimate \eqref{eq:heat_bound_1} for $p(t,\cdot)$ and the equivalence of norms in \eqref{eq:norm_equiv}, we have the following two-sided estimate of the sub-Riemannian heat kernel:
    \begin{align}\label{eq:heat_bound}
        c^{-1} t^{-\frac{Q}{2}} e^{-\frac{\|x\|^2_\alpha}{ct}}\le p(t,x)\le c t^{-\frac{Q}{2}} e^{-\frac{c\|x\|^2_\alpha}{t}}.
    \end{align}
    From \cite[Proposition~28.3]{Sato_Book} it follows that $S^{(\alpha)}_1$ is absolutely continuous, and let us denote its density by $\eta^{(\alpha)}_1$.
    In particular, from \cite[p.~97, Eq (5.19)]{Bogdan_et_al2009} it is known that for all $s>0$,
    \begin{align}\label{eq:stable_asymptotic}
        \eta^{(\alpha)}_1(s)\le C\min\{1, s^{-1-\frac{\alpha}{2}}\}
    \end{align}
    for some positive constant $C$.
    Using \eqref{eq:heat_bound} and \eqref{eq:stable_asymptotic} along with the self-similarity in \eqref{eq:self-similarity} we obtain
    \begin{align*}
        p^{(\alpha)}(t,x)\le t\int_0^\infty c s^{-\frac{Q}{2}} e^{-\frac{\|x\|^2_\alpha}{cs}} s^{-1-\frac{\alpha}{2}}ds.
    \end{align*}
    Evaluating the above integral, we conclude the proof of \eqref{eq:frac_heat_bound}.
    
    To prove \eqref{eq:heat_asymp}, let $\mathcal{A}_\alpha$ denote the infinitesimal generator of $S^{(\alpha)}_t$ in $C_0((0,\infty))$, the space of all continuous functions on $(0,\infty)$ vanishing at $0$ and $\infty$. Then by \eqref{eq:subordinator_generator}, $C^2_0((0,\infty))\subseteq \mathcal{D}(\mathcal{A}_\alpha)$, and
    \begin{align*}
        \mathcal{A}_\alpha u(t)=\frac{\alpha}{2\Gamma\left(1-\frac\alpha 2\right)}\int_0^\infty\frac{u(t+s)-u(t)}{s^{1+\frac{\alpha}{2}}} ds, \quad u\in C^2_0([0,\infty)).
    \end{align*}
     Since for any $x\in\mathbb G$, $t\mapsto p(t,x)\in C^\infty_0([0,\infty))$ and \eqref{eq:heat_deriv_bound} holds, it follows that $t\mapsto p(t,x)\in \mathcal{D}(\mathcal{A}_\alpha)$.
    Since $p(0,x)=0$ for all $x\in\mathbb G$, we have
    \begin{align*}
        \lim_{t\to 0}\frac{1}{t}p^{(\alpha)}(t,x)&=\lim_{t\to 0}\frac{1}{t}\int_0^\infty p(s,x)\mathbb P(S_t\in ds) \\
        &= \mathcal{A}_\alpha p(\cdot, x)(0) \\
        &=\frac{\alpha}{2\Gamma\left(1-\frac\alpha 2\right)}\int_0^\infty p(s,x) s^{-1-\frac{\alpha}{2}} ds \\
        &= \frac{1}{\|x\|^{Q+\alpha}_\alpha}
    \end{align*}
    This completes the proof of the proposition.
\end{proof}

\section{Regularity conditions on the boundary of the domain} \label{sec.Boundary}

\subsection{Characteristic points} In several results we assume that the boundary of the domain does not contain any characteristic points, which are defined formally as follows.
\begin{definition}[Characteristic points]\label{def:characteristic_pt}
    For an open set $\Omega\subset\mathbb{G}$ with $C^1$-boundary, a point $p\in\partial \Omega$ is said to be a \emph{characteristic point} if $\mathcal{H}_p\subseteq T_p(\partial\Omega)$, that is, the horizontal  normal ( i.e. the horizontal projection of the Euclidean normal) at $p$ to the boundary of $\Omega$ is zero. We say that $\partial \Omega$ is completely non-characteristic if it does not have any characteristic points.
\end{definition}
The boundary of a domain can be infinitely smooth and may still contain characteristic points. For example, the Kor\'anyi ball in the 3-dimensional Heisenberg group defined by
\begin{align*}
    \{(x,y,z): (x^2+y^2)^2+16z^2\le 1\}
\end{align*}
has $C^\infty$ boundary. Both the points $(0,0,1)$ and $(0,0,-1)$ are characteristic points. In sub-Riemannian geometry, characteristic points play a crucial role in determining smoothness of natural functions. For instance, by \cite[Theorem~1.1]{BossioRizziRossi2024Arxiv}, when $\Omega$ has $C^k$ boundary ($k\ge 2$) without characteristic points, the signed distance function $\delta$ defined by
\begin{align*}
    \delta(x)=\begin{cases}
        d(x,\partial\Omega) & \mbox{if $x\in\Omega$} \\
        -d(x,\partial\Omega) & \mbox{if $x\in\Omega^c$},
    \end{cases}
\end{align*}
    is also $C^{k}$ in a neighborhood of $\partial\Omega$. On the other hand, when $\Omega$ has $C^\infty$ boundary with characteristic points, $\delta$ may not even be Lipschitz with respect to natural coordinates, see \cite{AlbanoCannarsaScarinci2018} for details.

\subsection{The intrinsic exterior cone condition}\label{s:cones}
The following notion of intrinsic cones in Carnot groups was introduced by
Franchi and Serapioni \cite{FranchiSerapioni2016}; in the Heisenberg group
the definition goes back to \cite{MR2836591}. See also
\cite{MR3587666} for an introduction to the subject.

\begin{definition}
A subgroup $\mathbb K$ of a Carnot group is said to be \emph{homogeneous} if
\begin{align*}
\delta_\lambda(\mathbb K)=\mathbb K
\qquad\text{for every }\lambda>0.
\end{align*}
\end{definition}

\begin{definition}\label{def:oriented_cone}
Let $\mathbb G$ be a Carnot group and let $\mathbb H,\mathbb K$ be two
homogeneous subgroups of $\mathbb G$. We say that $\mathbb H$ and
$\mathbb K$ are \emph{complementary} if
\begin{align*}
\mathbb H\cap\mathbb K=\{e\}, \quad \mbox{and} \quad \mathbb G=\mathbb H\mathbb K.
\end{align*}
Assume moreover that $\mathbb K$ is a one-dimensional horizontal homogeneous
subgroup. Thus
\begin{align*}
\mathbb K=\exp(\mathbb RX)
\end{align*}
for some nonzero $X\in V_1$. Set
\begin{align*}
\mathbb K^+:=\{\exp(tX):t\geq0\}.
\end{align*}
For $p\in\mathbb G$ and $\beta,\rho>0$, we define the
\emph{oriented intrinsic cone} associated with the decomposition
$\mathbb G=\mathbb H\mathbb K$ by
\begin{align*}
\mathcal C_p^+(\beta,\rho;\mathbb H,\mathbb K)
:=
p\left\{
hk\in\mathbb H\mathbb K:
k\in\mathbb K^+,\ 
\|h\|\leq\beta\|k\|,\ 
\|k\|<\rho
\right\}.
\end{align*}
\end{definition}
In the above definition, $\beta, \rho$ are called the \emph{aperture} and \emph{height} of the cone respectively. It follows from the definition that the oriented intrinsic cones have the following left-translation property:
\begin{align*}
    \mathcal C_p^+(\beta,\rho;\mathbb H,\mathbb K)=L_p\left(\mathcal{C}_e^+(\beta,\rho;\mathbb H,\mathbb K)\right),
\end{align*}
where $e$ is the identity element, and $L_p$ denotes the left-translation map $L_p(x)=px$.
For the purposes of this section, it is convenient to use a canonical
complement for every horizontal direction. Let $X\in V_1$ satisfy $\langle X,X\rangle_{\mathcal H}=1$, and set
\begin{align*}
W_X:=\{Y\in V_1:\langle X,Y\rangle_{\mathcal H}=0\}.
\end{align*}
Then
\begin{align*}
\mathfrak{g}=\mathbb RX\oplus W_X\oplus V_2\oplus\cdots\oplus V_k.
\end{align*}
Define
\begin{align*}
\mathfrak h_X
:=
W_X\oplus V_2\oplus\cdots\oplus V_k.
\end{align*}
Since $[V_i,V_j]\subseteq V_{i+j}$ for all $1\le i,j\le k$, the subspace $\mathfrak h_X$ is a homogeneous ideal of $\mathfrak g$.
Hence
\begin{align*}
\mathbb H_X:=\exp(\mathfrak h_X)
\end{align*}
is a homogeneous normal subgroup of $\mathbb G$. Setting
\begin{align*}
\mathbb K_X:=\{\exp(tX):t\in\mathbb R\}.
\end{align*}
we note that $\mathbb H_X$ and $\mathbb K_X$ are homogeneous subgroups of $\mathbb G$, and by the Baker-Campbell-Dynkin-Hausdorff formula, they are also complementary subgroups, that is, 
\begin{align*}
\mathbb G=\mathbb H_X\mathbb K_X,
\qquad
\mathbb H_X\cap\mathbb K_X=\{e\}.
\end{align*}
In particular, every $g\in\mathbb G$ admits a unique decomposition
\begin{align*}
g=hk,
\qquad
h\in\mathbb H_X,\quad k\in\mathbb K_X,
\end{align*}
and the multiplication map
\begin{align*}
\Phi_X:\mathbb H_X\times\mathbb K_X\longrightarrow\mathbb G,
\qquad
\Phi_X(h,k)=hk,
\end{align*}
is a smooth diffeomorphism.

\begin{definition}
\label{def:oriented_cone_horiz}
Let $X$ be a unit vector in $V_1$. For $p\in\mathbb G$ and $\beta,\rho>0$, the \emph{oriented intrinsic horizontal cone} in the direction of $X$ is defined as
\begin{align*}
\mathcal C_p^+(\beta,\rho;X):=\mathcal{C}^+_p(\beta,\rho; \mathbb{H}_X,\mathbb{K}_X).
\end{align*}
\end{definition}

The following dilation property follows immediately from the definition.

\begin{lemma}\label{lem:cone_dilation}
For every $\lambda>0$,
\begin{align*}
\delta_\lambda
\bigl(\mathcal C_e^+(\beta,\rho;X)\bigr)
=
\mathcal C_e^+(\beta,\lambda\rho;X).
\end{align*}
\end{lemma}

\begin{proof}
Since $\mathbb H_X$ and $\mathbb K_X$ are homogeneous subgroups,
\begin{align*}
\delta_\lambda(hk)
=
\delta_\lambda(h)\delta_\lambda(k),
\qquad
\delta_\lambda(h)\in\mathbb H_X,
\quad
\delta_\lambda(k)\in\mathbb K_X^+.
\end{align*}
Moreover,
\begin{align*}
\|\delta_\lambda(h)\|
=
\lambda\|h\|,
\qquad
\|\delta_\lambda(k)\|
=
\lambda\|k\|.
\end{align*}
The conclusion follows directly.
\end{proof}

\begin{definition}
\label{def:cone_cond}
An open set $\Omega\subset\mathbb G$ is said to satisfy the
\emph{intrinsic exterior cone condition} if there exist $\beta,\rho>0$ such
that, for every $p\in\partial\Omega$, there exists a unit vector
$X_p\in V_1$ satisfying
\begin{align*}
\mathcal C_p^+(\beta,\rho;X_p)\setminus\{p\}
\subseteq
\mathbb G\setminus\overline{\Omega}.
\end{align*}
\end{definition}

Unlike Euclidean cones with fixed aperture and height, cones corresponding
to different horizontal directions need not be images of a single fixed
cone under measure-preserving graded automorphisms. Indeed, an orthogonal
transformation of $V_1$ mapping one horizontal direction into another need
not extend to a graded automorphism of $\mathbb G$. We shall nevertheless
obtain uniform lower bounds for both their volumes and the integrals needed
below.

For a unit vector $X\in V_1$, let $d\nu_X$ denote the Haar measure on
$\mathbb H_X$ obtained by transporting the Euclidean Lebesgue measure on
$\mathfrak h_X$ through the exponential map.

\begin{lemma}\label{lem:quotient_integral}
For every $X\in V_1$ with $|X|_{\mathcal H}=1$ and every $f\in L^1(\mathbb G, dx)$,
\begin{align*}
\int_{\mathbb G}f(g)\,dg
=
\int_{\mathbb R}
\int_{\mathbb H_X}
f\bigl(h\exp(tX)\bigr)\,d\nu_X(h)\,dt.
\end{align*}
\end{lemma}

\begin{proof}
Since
\begin{align*}
\mathfrak g=\mathfrak h_X\oplus\mathbb RX
\end{align*}
is an orthogonal direct sum and $|X|_{\mathcal H}=1$, the product of the Euclidean
Lebesgue measure on $\mathfrak h_X$ and the Lebesgue measure $dt$ is exactly
the Euclidean Lebesgue measure on $\mathfrak g$ under the linear
identification
\begin{align*}
(Y,t)\longmapsto Y+tX.
\end{align*}
We now consider the map
\begin{align*}
\Theta_X:\mathfrak h_X\times\mathbb R\longrightarrow\mathfrak g,
\qquad
\Theta_X(Y,t)
:=
\log\bigl(\exp(Y)\exp(tX)\bigr).
\end{align*}
Write
\begin{align*}
Y=Y_1+\cdots+Y_k,
\qquad
Y_1\in W_X,\quad Y_j\in V_j\quad (j\geq2).
\end{align*}
By the Baker-Campbell-Dynkin-Hausdorff formula, the component of
$\Theta_X(Y,t)$ in $V_1$ is exactly
\begin{align*}
Y_1+tX,
\end{align*}
whereas, for every $j\geq2$, its $V_j$-component has the form
\begin{align*}
Y_j+P_j(Y_1,\ldots,Y_{j-1},t),
\end{align*}
where $P_j$ is a polynomial depending only on coordinates of strictly
smaller homogeneous degree and on $t$. Thus, if the coordinates are ordered according to the stratification,
the differential of $\Theta_X$ is triangular with identity diagonal
blocks. Consequently,
\begin{align*}
\left|\det D\Theta_X(Y,t)\right|=1
\end{align*}
for every $(Y,t)\in\mathfrak h_X\times\mathbb R$.

Since the exponential map transports Euclidean Lebesgue measure on
$\mathfrak g$ to the fixed Haar measure $dg$, while it transports Euclidean
Lebesgue measure on $\mathfrak h_X$ to $d\nu_X$, the change-of-variables
formula gives
\begin{align*}
\int_{\mathbb G}f(g)\,dg
=
\int_{\mathbb R}\int_{\mathfrak h_X}
f\bigl(\exp(Y)\exp(tX)\bigr)\,dY\,dt,
\end{align*}
which is precisely the claimed formula.
\end{proof}

For later use, define
\begin{align*}
\alpha_X:=\|\exp(X)\|.
\end{align*}
By homogeneity,
\begin{align*}
\|\exp(tX)\|=\alpha_X|t|,
\qquad t\in\mathbb R.
\end{align*}

\begin{lemma}\label{lem:volumes}
Let $\mathbb G=\mathbb H_X\mathbb K_X$ be the decomposition from
Definition~\ref{def:oriented_cone_horiz}. Then
\begin{align}\label{eq:cone_volume}
\Vol(\mathcal C_e^+(\beta,\rho;X))
=
\frac{\beta^{Q-1}\rho^Q}{Q\alpha_X}
\nu_X\bigl(\mathbb B(e,1)\cap\mathbb H_X\bigr).
\end{align}
Moreover, for every $c>0$,
\begin{align}\label{eq:cone_integral}
\int_{\mathcal C_e^+(\beta,\rho;X)}
e^{-c\|g\|^2}\,dg
\geq
\frac{\beta^{Q-1}}{\alpha_X}
\nu_X\bigl(\mathbb B(e,1)\cap\mathbb H_X\bigr)
\int_0^\rho
r^{Q-1}e^{-c(1+\beta)^2r^2}\,dr.
\end{align}
\end{lemma}

\begin{proof}
By Lemma~\ref{lem:quotient_integral}, writing
$k=\exp(tX)$ with $t\geq0$, we obtain
\begin{align}\label{eq:volume_cone}
\Vol(\mathcal C_e^+(\beta,\rho;X))
=
\int_0^{\rho/\alpha_X}
\nu_X\bigl(
\mathbb B(e,\beta\alpha_Xt)\cap\mathbb H_X
\bigr)\,dt.
\end{align}
The homogeneous dimension of $\mathbb H_X$ is $Q-1$. Therefore
\begin{align*}
\nu_X\bigl(
\mathbb B(e,r)\cap\mathbb H_X
\bigr)
=
r^{Q-1}
\nu_X\bigl(
\mathbb B(e,1)\cap\mathbb H_X
\bigr).
\end{align*}
Substitution into \eqref{eq:volume_cone} gives
\begin{align*}
\begin{split}
\Vol(\mathcal C_e^+(\beta,\rho;X))
&=
\beta^{Q-1}\alpha_X^{Q-1}
\nu_X\bigl(\mathbb B(e,1)\cap\mathbb H_X\bigr)
\int_0^{\rho/\alpha_X}t^{Q-1}\,dt
\\
&=
\frac{\beta^{Q-1}\rho^Q}{Q\alpha_X}
\nu_X\bigl(\mathbb B(e,1)\cap\mathbb H_X\bigr),
\end{split}
\end{align*}
which proves \eqref{eq:cone_volume}.

To prove \eqref{eq:cone_integral}, let
\begin{align*}
g=h\exp(tX)\in\mathcal C_e^+(\beta,\rho;X)
\end{align*}
and set
\begin{align*}
r:=\|\exp(tX)\|=\alpha_Xt.
\end{align*}
Since $\|h\|\leq\beta r$, the triangle inequality gives
\begin{align*}
\|g\|
\leq
\|h\|+\|\exp(tX)\|
\leq
(1+\beta)r.
\end{align*}
Consequently,
\begin{align*}
e^{-c\|g\|^2}
\geq
e^{-c(1+\beta)^2r^2}.
\end{align*}
Using Lemma~\ref{lem:quotient_integral} once more,
\begin{align*}
\begin{split}
\int_{\mathcal C_e^+(\beta,\rho;X)}
e^{-c\|g\|^2}\,dg
&\geq
\int_0^{\rho/\alpha_X}
e^{-c(1+\beta)^2\alpha_X^2t^2}
\nu_X\bigl(
\mathbb B(e,\beta\alpha_Xt)\cap\mathbb H_X
\bigr)\,dt
\\
&=
\beta^{Q-1}
\nu_X\bigl(\mathbb B(e,1)\cap\mathbb H_X\bigr)
\alpha_X^{Q-1}
\\
&\qquad\times
\int_0^{\rho/\alpha_X}
t^{Q-1}
e^{-c(1+\beta)^2\alpha_X^2t^2}\,dt.
\end{split}
\end{align*}
The change of variables $r=\alpha_Xt$ yields
\eqref{eq:cone_integral}.
\end{proof}

\begin{lemma}
\label{lem:continuity_HM}
There exist positive constants $a_*,a^*,\alpha_*,\alpha^*$, depending only on
$\mathbb G$, the fixed inner product, and the homogeneous norm, such that
for every $X\in V_1$ with $|X|_{\mathcal H}=1$,
\begin{align*}
a_*
\leq
\nu_X\bigl(\mathbb B(e,1)\cap\mathbb H_X\bigr)
\leq
a^*
\end{align*}
and
\begin{align*}
\alpha_*
\leq
\alpha_X
\leq
\alpha^*.
\end{align*}
\end{lemma}

\begin{proof}
For
\begin{align*}
Y=Y_1+\cdots+Y_s,
\qquad
Y_j\in V_j,
\end{align*}
define the standard homogeneous quasinorm
\begin{align*}
N_0(\exp Y)
:=
\max_{1\leq j\leq s}|Y_j|^{1/j}.
\end{align*}
Since $N_0$ and $\|\cdot\|$ are continuous, positive away from the identity,
and homogeneous of degree one, by \cite[Proposition~5.1.4]{BonfiglioliLanconelliUguzzoniBook} there exist constants $c_0,C_0>0$ such that
\begin{align*}
c_0N_0(g)\leq\|g\|\leq C_0N_0(g)
\qquad\text{for every }g\in\mathbb G.
\end{align*}
Let $m_j=\dim V_j$ be defined as before. For every unit $X\in V_1$, the first layer of $\mathfrak h_X$ is
$W_X=X^\perp$, which has dimension $m_1-1$. Since $d\nu_X$ is induced by
the Euclidean Lebesgue measure on
\begin{align*}
X^\perp\oplus V_2\oplus\cdots\oplus V_k,
\end{align*}
there exists a constant $\kappa>0$, independent of $X$, such that, for
every $r>0$,
\begin{align}\label{eq:HM_equality}
\nu_X\bigl(
\{h\in\mathbb H_X:N_0(h)<r\}
\bigr)
=
\kappa r^{Q-1}.
\end{align}
Indeed, in exponential coordinates the set on the left-hand side is the
Cartesian product
\begin{align*}
\{|Y_1|<r\}\subset X^\perp,
\qquad
\{|Y_j|<r^j\}\subset V_j,
\quad j\geq2.
\end{align*}
Its Euclidean volume is therefore
\begin{align*}
\kappa
r^{m_1-1}
\prod_{j=2}^k r^{jm_j}
=
\kappa r^{Q-1},
\end{align*}
where $\kappa$ depends only on the dimensions of the layers. The equivalence of the two homogeneous quasinorms gives
\begin{align}\label{BallInclusion}
\{h\in\mathbb H_X:N_0(h)<C_0^{-1}\}
\subset
\mathbb B(e,1)\cap\mathbb H_X
\subset
\{h\in\mathbb H_X:N_0(h)<c_0^{-1}\}.
\end{align}
Combining \eqref{eq:HM_equality} and \eqref{BallInclusion}, we obtain
\begin{align}\label{eq:HM_ineq}
\kappa C_0^{-(Q-1)}
\leq
\nu_X\bigl(\mathbb B(e,1)\cap\mathbb H_X\bigr)
\leq
\kappa c_0^{-(Q-1)}
\end{align}
for every unit $X$. This proves the first assertion.

Finally, the map
\begin{align*}
X\longmapsto\alpha_X=\|\exp(X)\|
\end{align*}
is continuous and strictly positive on the Euclidean unit sphere of $V_1$.
Since this sphere is compact, there exist $\alpha_*,\alpha^*>0$ such that
\begin{align*}
\alpha_*
\leq
\alpha_X
\leq
\alpha^*
\end{align*}
for every $|X|=1$.
\end{proof}

Combining the preceding lemmas gives the following uniform estimates.

\begin{proposition}\label{prop:unif_lower_bound}
Let $B=(B_t)_{t\geq0}$ denote the hypoelliptic Brownian motion on
$\mathbb G$ starting from the identity $e$, and let
$\mathcal C_e^+(\beta,\rho;X)$ be as in
Definition~\ref{def:oriented_cone_horiz}. Then
\begin{align}\label{eq:min_HM}
\inf_{\substack{X\in V_1\\ |X|=1}}
\Vol(\mathcal C_e^+(\beta,\rho;X))
>0.
\end{align}
Moreover, for every fixed $t>0$,
\begin{align}\label{eq:int_cone_min}
\inf_{\substack{X\in V_1\\ |X|=1}}
\mathbb P_e\bigl(B_t\in\mathcal C_e^+(\beta,\rho;X)\bigr)
>0.
\end{align}
\end{proposition}

\begin{proof}
By \eqref{eq:cone_volume} and Lemma~\ref{lem:continuity_HM},
\begin{align*}
\Vol(\mathcal C_e^+(\beta,\rho;X))
\geq
\frac{\beta^{Q-1}\rho^Q}{Q\alpha^*}\,a_*,
\end{align*}
uniformly over all unit $X\in V_1$. This proves
\eqref{eq:min_HM}.

For the second assertion, let $p(t,\cdot)$ be the horizontal heat kernel of $B$. By the
Gaussian lower bound \eqref{eq:heat_bound_1}, there exist $C_1,c_1>0$ such
that
\begin{align*}
p(t,x)
\geq
C_1t^{-Q/2}
\exp\left(-c_1\frac{\|x\|^2}{t}\right).
\end{align*}
Hence
\begin{align*}
\mathbb P_e\bigl(B_t\in\mathcal C_e^+(\beta,\rho;X)\bigr)
\geq
C_1t^{-Q/2}
\int_{\mathcal C_e^+(\beta,\rho;X)}
\exp\left(-c_1\frac{\|x\|^2}{t}\right)dx.
\end{align*}
Applying \eqref{eq:cone_integral} with $c=c_1/t$ and then
Lemma~\ref{lem:continuity_HM}, we obtain
\begin{align*}
\begin{split}
\mathbb P_e\bigl(B_t\in\mathcal C_e^+(\beta,\rho;X)\bigr)
&\geq
C_1t^{-Q/2}
\frac{\beta^{Q-1}a_*}{\alpha^*}
\\
&\qquad\times
\int_0^\rho
r^{Q-1}
\exp\left(
-\frac{c_1(1+\beta)^2r^2}{t}
\right)\,dr.
\end{split}
\end{align*}
The right-hand side is strictly positive and independent of $X$, proving
\eqref{eq:int_cone_min}.
\end{proof}
We now establish the exterior cone condition for regular
non-characteristic boundaries.
\begin{theorem}
\label{thm:noncharacteristic-cones}
Let $\Omega\subset\mathbb{G}$ be a bounded domain with compact $C^1$ boundary and no
characteristic points.  Then $\Omega$ satisfies the intrinsic exterior cone
condition introduced in Definition~\ref{def:cone_cond}.
\end{theorem}

\begin{proof} 
Fix $p_0\in\partial\Omega$.  Choose an open neighborhood $U$ of $p_0$ and
a defining function $\Phi\in C^1(U)$ such that
\begin{align*}
\Omega\cap U=\{\Phi<0\}, \quad \partial\Omega\cap U=\{\Phi=0\}, \quad U\setminus\overline\Omega=\{\Phi>0\}.
\end{align*}
Since $p_0$ is non-characteristic, $\nabla_{\mathcal H}\Phi(p_0)\neq 0$. Define 
\begin{align*}
	X(q):= \frac{\nabla_{\mathcal H}\Phi(q)}{|\nabla_{\mathcal H} \Phi(q)|_{\mathcal H}}, \quad q\in\partial\Omega,
\end{align*} 
and we denote $X:=X(p_0)$. Due to continuity of the mapping $q\mapsto X(q)$, we can find an open boundary patch $\Sigma\Subset \partial\Omega\cap U$ such that
\begin{align*}
	\langle X(q), X\rangle_{\mathcal H}\ge \frac12 \quad \mbox{for all $q\in\Sigma$}.
\end{align*}
Also, by the continuity of the map $q\mapsto |\nabla_{\mathcal H} \Phi(q)|$ on $\partial\Omega$, there exists $\delta>0$ such that 
\begin{align}\label{eq:nabla_lb}
	|\nabla_{\mathcal H}\Phi(q)|_{\mathcal H}\ge \delta \quad \mbox{for all $q\in \Sigma$}.
\end{align}
 Therefore for all $q\in\Sigma$,
 \begin{align*}
 	X\Phi(q)=\langle \nabla_{\mathcal H}\Phi(q), X\rangle_{\mathcal H}\ge |\nabla_{\mathcal H} \Phi(q)|_{\mathcal H}\langle X(q), X\rangle_{\mathcal H} \ge \frac{\delta}{2},
 \end{align*}
where with an abuse of notation, we have also used $X$ to denote the left-invariant vector field generated by it. We will first prove that there exists $\beta_0,\rho_0>0$ such that
\begin{align*}
	\mathcal{C}^+_q(\beta_0,\rho_0; X)\setminus\{q\}\subset \mathbb{G}\setminus \Omega \qquad \mbox{for all $q\in\Sigma$}.
\end{align*}
We note the following uniform estimate: there are constants
$C_0>0$ and $\varepsilon_0>0$ such that
\begin{equation}\label{eq:cone-transverse-estimate}
|\Phi(qh)-\Phi(q)|\leq C_0\|h\|
\end{equation}
for every $q\in\Sigma$ and every $h\in\mathbb H_X$ with
$\|h\|<\varepsilon_0$.  To see this, write
$h=\exp(Y_1+\cdots+Y_k)$, $Y_j\in V_j$.  Equivalence with homogeneous quasinorms gives
\begin{align*}
|Y_j|\leq C \|h\|^j,
\qquad j=1,\ldots,k.
\end{align*}
For $\|h\|\leq1$, this implies
$|Z_1+\cdots+Z_s|_{\mathrm{Euc}}\leq C'\|h\|$, where $|\cdot|_{\mathrm{Euc}}$ denotes the Euclidean norm. The multiplication map on $\mathbb G$ is
smooth, so on the compact set
\begin{align*}
	\overline{\Sigma}\times \{h\in\mathbb{H}_X:\|h\|\le \varepsilon_0\},
\end{align*}
the Euclidean coordinate (via the global exponential map) distance between $qh$ and $q$ is bounded by a constant multiple of $\|h\|$.
Estimate \eqref{eq:cone-transverse-estimate} now follows from the boundedness
of the total derivative $D\Phi$ on $\overline U$. Let us write 
\begin{align*}
\alpha_X:=\|\exp(X)\|>0.
\end{align*}
By homogeneity,
\begin{equation}\label{eq:horizontal-axis-norm}
	\|\exp(sX)\|=s\alpha_X,
\qquad s>0.
\end{equation}
For every $q\in\Sigma$ we have
\[
X\Phi(q)
=
\langle\nabla_H\Phi(q),X\rangle_H
\geq \frac{\delta}{2}.
\]
Since $X\Phi$ is continuous and $\overline{\Sigma}$ is compact, there
exists a neighborhood $V\Subset U$ of $\overline{\Sigma}$ such that
\[
X\Phi(z)\geq\frac{\delta}{4}
\qquad\text{for every }z\in V.
\]
After decreasing $\varepsilon_0$ and $s_0$ if necessary, we may also
assume that
\[
qh\exp(rX)\in V
\]
whenever $q\in\Sigma$, $\|h\|<\varepsilon_0$, and $0\leq r\leq s_0$.

Choose $\beta_0>0$ and $\rho_0>0$ so that
\[
C_0\beta_0\alpha_X<\frac{\delta}{8},
\qquad
\rho_0<\min\left\{\alpha_Xs_0,\frac{\varepsilon_0}{\beta_0}\right\}.
\]
If $g\in \mathcal{C}_q^+(\beta_0,\rho_0;X)\setminus\{q\}$, then
$g=qh\exp(sX)$ with
\[
\|h\|\leq\beta_0\alpha_Xs,
\qquad
s<\frac{\rho_0}{\alpha_X}.
\]
Using \eqref{eq:cone-transverse-estimate}, the bound
$X\Phi\geq\delta/4$ on $V$ and the fundamental theorem of
calculus, we obtain
\begin{align*}
\Phi(g)
&=\Phi(qh)+\int_0^s X\Phi(qh\exp(rX))\,dr\\
&\geq -C_0\|h\|+\frac{\delta}{4}s\\
&\geq
\left(\frac{\delta}{4}-C_0\beta_0\alpha_X\right)s
>\frac{\delta}{8}s>0.
\end{align*}
Therefore,
\begin{align*}
\mathcal{C}_q^+(\beta_0,\rho_0;X)\setminus\{q\}
\subset U\setminus\Omega
\qquad\text{for every }q\in\Sigma.
\end{align*}

The compact boundary $\partial\Omega$ is covered by finitely many patches $(\Sigma_i)_{1\le i\le N }$ of this type,
with directions $X_1,\ldots,X_N$ and parameters
$(\beta_i,\rho_i)$.  Set
\begin{align*}
\beta:=\min_{1\leq i\leq N}\beta_i,
\qquad
\rho:=\min_{1\leq i\leq N}\rho_i.
\end{align*}
If $p$ belongs to the $i$-th patch, then decreasing aperture and height only
shrinks the cone, and hence
\begin{align*}
\mathcal{C}_p^+(\beta,\rho;X_i)\setminus \{p\}
\subset
\mathcal{C}_p^+(\beta_i,\rho_i;X_i)\setminus \{p\}
\subset
\mathbb{G}\setminus\overline\Omega.
\end{align*}
This proves the intrinsic exterior cone condition with uniform parameters.
\end{proof}

\subsection{The volume density condition}\label{s:VDC}
This is another type of regularity condition on the boundary, which is weaker than the intrinsic exterior cone condition. 
\begin{definition}\label{def:VDC}
    An open set $\Omega\subset\mathbb{G}$ is said to satisfy the \emph{volume density condition} if there exists $c>0$ such that for all $r>0$ and $x\in\partial\Omega$,
    \begin{align*}
        \Vol(\overline{\Omega}^c\cap \mathbb{B}(x,r))\ge c r^Q,
    \end{align*}
    where $Q$ is the homogeneous dimension of $\mathbb{G}$. 
\end{definition}
\begin{remark}
    The volume density condition for Euclidean spaces was introduced in \cite{Jang-MeiWu2002} to study harmonic measures for symmetric stable processes.
\end{remark}
\begin{proposition}\label{prop:VDC_conseq}
    Suppose that $\Omega$ is an open set in $\mathbb{G}$.
    \begin{enumerate}[leftmargin=*]
    \item If $\Omega$ satisfies the volume density condition then $\Vol(\partial\Omega)=0$.
    \item If $\Omega$ satisfies the intrinsic exterior cone condition and $\overline{\Omega}$ has finite volume, then it also satisfies the volume density condition. In particular, if $\Omega$ is open, bounded with $C^1$ boundary having no characteristic points, it satisfies the volume density condition.
    \end{enumerate}
\end{proposition}
\begin{proof}
    Assume that $\Vol(\partial\Omega)>0$. Since the Carnot-Carath\'eodory metric on $\mathbb{G}$ is volume doubling, by the Lebesgue differentiation theorem on volume doubling metric measure spaces (see \cite[Theorem~1.8]{HeinonenBook2001}), there exists $p\in\partial\Omega$ such that 
    \begin{align*}
        \lim_{r\to 0}\frac{\Vol(\partial\Omega\cap \mathbb{B}(p,r))}{\Vol(\mathbb{B}(p,r))}=1.
    \end{align*}
    On the other hand, due to the volume density condition, for any $r>0$ and $p\in\partial\Omega$,
    \begin{align*}
        \Vol(\partial\Omega\cap \mathbb{B}(p,r))\le \Vol(\mathbb{B}(p,r))-\Vol(\overline{\Omega}^c\cap \mathbb{B}(p,r))\le (1-c)\Vol(\mathbb{B}(p,r))
    \end{align*}
    for some $1-c\in [0,1)$, which shows that
    \begin{align*}
        \limsup_{r\to 0}\frac{\Vol(\partial\Omega\cap \mathbb{B}(p,r))}{\Vol(\mathbb{B}(p,r))}\le 1-c<1.
    \end{align*}
    This leads to a contradiction and hence $\Vol(\partial\Omega)=0$. 

  To prove the second assertion, we use the boundedness of $\Omega$. Since $\Omega$ satisfies the intrinsic exterior cone condition, there exists $r_0,\beta, \rho>0$ such that for each point $p\in\partial\Omega$, there exists an intrinsic cone $\mathcal{C}^+_p(\beta, \rho; X_p)$ in the sense of Definition~\ref{def:oriented_cone_horiz} such that 
    \begin{align*}
        \mathcal{C}^+_p(\beta, \rho; X_p)\setminus \{p\}\subset \overline{\Omega}^c.
    \end{align*}
    From the definition of the intrinsic cones it follows that for all $p\in\partial\Omega$,
    \begin{align*}
        \mathcal{C}^+_p\left(\beta, \frac{r}{1+\beta}; X_p\right)\subset \mathcal{C}^+_p(\beta, r; X_p)\cap \mathbb{B}(p,r).
    \end{align*}
    Therefore, for all $r\le\rho$, 
    \begin{equation}\label{eq:small_r}
    \begin{aligned}
        \Vol(\overline{\Omega}^c\cap \mathbb{B}(p,r))&\ge \Vol(\mathcal{C}^+_p(\beta,r; X_p)\cap \mathbb{B}(p, r)) \\
        &\ge \Vol\left(\mathcal{C}^+_p\left(\beta, \frac{r}{1+\beta}; X_p\right)\right) \\
        &=\Vol\left(\mathcal{C}^+_e\left(\beta,1; X_p\right)\right)\left(\frac{r}{1+\beta}\right)^Q
    \end{aligned}
    \end{equation}
    where the last identity follows from Lemma~\ref{lem:cone_dilation} along with the left translation invariance and the scaling property of the volume.  Since $\inf_{p\in\partial\Omega}\Vol(\mathcal{C}^+_e(\beta,1; X_p))>0$, thanks to Proposition~\ref{prop:unif_lower_bound}, \eqref{eq:small_r} implies that there exists a constant $c_1>0$ such that for all $p\in\partial\Omega$ and $r\le \rho$,
    \begin{align}\label{eq:small_r_1}
       \Vol(\overline{\Omega}^c\cap \mathbb{B}(p,r))\ge c_1 r^Q. 
    \end{align}
    On the other hand, for any $p\in\partial\Omega$ and $r>0$, 
    \begin{align*}
        \Vol(\overline{\Omega}^c\cap \mathbb{B}(p,r))\ge\Vol(\mathbb{B}(p,r))-\Vol(\overline{\Omega})=r^Q\Vol(\mathbb{B}(e,1))-\Vol(\overline{\Omega}).
    \end{align*}
    Therefore, there exists sufficiently large $K>0$ such that for all $r>K\rho$,
    \begin{align}\label{eq:large_r}
          \Vol(\overline{\Omega}^c\cap \mathbb{B}(p,r))\ge \frac{1}{2}r^Q\Vol(\mathbb{B}(e,1)).
    \end{align}
    Since $r\mapsto \Vol(\overline{\Omega}^c\cap \mathbb{B}(p,r))$ is an increasing function, for any $r\in (\rho, K\rho)$ and $p\in\partial\Omega$ we have
    \begin{align}\label{eq:intermediate_r}
        \Vol(\overline{\Omega}\cap \mathbb{B}(p,r))\ge  \Vol(\overline{\Omega}\cap \mathbb{B}(p,\rho))\ge c_1\rho^Q\ge \frac{c_1}{K^Q}r^Q.
    \end{align}
   Therefore, combining \eqref{eq:small_r_1}, \eqref{eq:large_r}, and \eqref{eq:intermediate_r} we conclude that $\Omega$ satisfies the volume density condition.
   
    Finally, by Theorem~\ref{thm:noncharacteristic-cones}, any bounded open set $\Omega$ with a $C^1$ non-characteristic boundary satisfies the intrinsic exterior cone condition, which implies the volume density condition for $\Omega$.
\end{proof}
A domain with a non-smooth boundary or with characteristic points may still satisfy the volume density condition. 

\begin{proposition}\label{prop:VDC_ball}
    Let $\mathbb{G}\cong\R^{m_1}\times\R^{m_2}$ be a Carnot group of step-2 and let $\|\cdot\|$ be any symmetric homogeneous quasinorm on $\mathbb{G}$ such that the unit ball at the identity
    \begin{align*}
        \mathbb{B}(e,1)=\{x\in\mathbb G: \|x\|<1\}
    \end{align*}
    is convex in Euclidean sense. Then for any $x\in\mathbb G$ and $R>0$, the metric balls $\mathbb B(x,R)$ satisfy the volume density condition.
\end{proposition}
\begin{remark}
     Our proof of the above proposition uses the central symmetry (in the Euclidean sense) of metric balls in step-2 Carnot groups. This is a consequence of the fact that the left-translation maps on step-2 Carnot groups are affine, which preserves central symmetry. This argument fails when $\mathbb G$ is a higher step Carnot group, as the left-translation maps become polynomial maps, see Proposition~\ref{prop:step_2} below.
\end{remark}
We include some examples of quasinorms on step-2 Carnot groups which satisfy the convexity condition in Proposition~\ref{prop:VDC_ball}. We use $|\cdot|$ to denote the Euclidean norm.
\begin{example}[Non-smooth cases]\label{ex:1}
    Let $\rho_1$ and $\rho_\infty$ denote the quasinorms on $\mathbb G$ defined by 
    \begin{align*}
        \rho_1(x)&=\left(|\xi_1|^2+|\xi_2|\right)^{\frac12}, \quad x=(\xi_1,\xi_2)\in\mathbb{R}^{m_1}\times \mathbb{R}^{m_2} \\
        \rho_\infty(x)&=\max\{|\xi_1|, \varepsilon|\xi_2|^{\frac 12}\}, \quad x=(\xi_1,\xi_2)\in \mathbb{R}^{m_1}\times \mathbb{R}^{m_2}.
    \end{align*}
    In the above definition of $\rho_\infty$, $\varepsilon$ is a positive constant. For a certain choice of $\varepsilon$, $\rho_\infty$ becomes a norm, see \cite[Theorem~5.1]{FranchiSerapioniSerra_Cassano2003a}.
    It is immediate that both $\rho_1$ and $\rho_\infty$ are symmetric, homogeneous, and the their metric balls defined through them are convex subsets of $\R^{m_1+m_2}$. Also, these metric balls do not have smooth boundaries.
\end{example}
\begin{example}[Cygan-Kor\'anyi type quasinorm]\label{ex:2}
Define the quasinorm $\rho_2$ on $\mathbb G$ as
\begin{align*}
    \rho_2(x)=\left(|\xi_1|^4+c|\xi_2|^2\right)^{\frac14}, \quad x=(\xi_1,\xi_2)\in\mathbb{R}^{m_1}\times \mathbb{R}^{m_2},
\end{align*}
where $c$ is a positive constant.
When $\mathbb G$ is an H-type group and $c=16$, $\rho_2$ is a norm and it is known as the Cygan-Kor\'anyi norm, see \cite{Cygan1981, Koranyi1985}.
Clearly, $\{x\in\mathbb G:\rho_2(x)<1\}$ is a convex subset of $\R^{m_1+m_2}$. The next lemma identifies the characteristic points on the boundary of metric balls with respect to $\rho_2$. 
\begin{lemma}\label{lem:characteristic_pt1}
    Let $\Omega=\mathbb{B}_{\rho_2}(e,R)$. A point $p\in\partial\Omega$ is a characteristic point if and only if $p=(0,\xi_2)$ with $c|\xi_2|^2=R^4$.
\end{lemma}
\begin{proof}
    Since $\mathbb G$ has step 2, the left-invariant vector fields generated by the basis elements $\{X_1,\ldots, X_{m_1}\}$ of $V_1$ can be written as
    \begin{align*}
        X_i = \partial_{\xi^i_1}+\frac{1}{2}\sum_{l=1}^{m_2}\langle J_l \xi_1, e_i\rangle\partial_{\xi^l_2}, \quad i=1,\ldots, m_1,
    \end{align*}
    where $(J_l)_{l=1}^{m_2}$ are skew-symmetric matrices on $\R^{m_1}$, and $\{e_1,\ldots, e_{m_1}\}$ is the canonical basis of $\R^{m_1}$. Writing 
    \begin{align*}
        F(\xi_1, \xi_2)=|\xi_1|^4+c |\xi_2|^2-R^4,
    \end{align*}
    a point $p=(\xi_1,\xi_2)$ is a characteristic point if $X_i F(p)=0$ for all $i=1,\ldots, m_1$.
    Note that
    \begin{align*}
        \partial_{\xi_1} F(\xi_1, \xi_2)=4|\xi_1|^2\xi_1, \quad \partial_{\xi_2} F(\xi_1,\xi_2)=2c\xi_2.
    \end{align*}
    Therefore, a simple computation shows that
    \begin{align*}
        X_i F(\xi_1, \xi_2)=4 |\xi_1|^2 \xi^i_1 +c\sum_{l=1}^{m_2}\xi^l_2\langle J_l\xi_1, e_i\rangle
    \end{align*}
    Denoting
    \begin{align*}
        J_{\xi_2}:=\sum_{l=1}^{m_2}\xi^l_2 J_l
    \end{align*}
    we can therefore write 
    \begin{align*}
        X_i F(\xi_1,\xi_2)=\langle 4|\xi_1|^2\xi_1+ cJ_{\xi_2}\xi_1, e_i\rangle, \quad 
        \nabla_{\mathcal H} F(\xi_1,\xi_2)= 4|\xi_1|^2\xi_1+ cJ_{\xi_2}\xi_1.
    \end{align*}
    Hence, 
    \begin{align*}
        \nabla_{\mathcal H} F(\xi_1, \xi_2)=0 \quad \iff \quad  \nabla_{\mathcal H} F(\xi_1,\xi_2)= 4|\xi_1|^2\xi_1+ cJ_{\xi_2}\xi_1=0.
    \end{align*}
    Since $J_{\xi_2}$ is skew-symmetric, the last identity implies that 
    \begin{align*}
        \langle 4|\xi_1|^2\xi_1+ cJ_{\xi_2}\xi_1,\xi_1\rangle=0 \quad \iff \quad 4|\xi_1|^4=0,
    \end{align*}
    which shows that $\xi_1=0$. Hence, $c|\xi_2|^2=R^4$. This completes the proof of the lemma.
\end{proof}
\end{example}
\begin{example}[Hebisch-Sikora norm]\label{ex:3} Hebisch and Sikora \cite{HebischSikora1990} proved the following result for arbitrary homogeneous Carnot groups:
\begin{theorem}[Theorem~2 in \cite{HebischSikora1990}]
    Let $\mathbb G$ be any homogeneous Carnot group equipped with the  dilations $(\delta_\lambda)_{\lambda>0}$. Then there exists $\varepsilon_0>0$ such that for every $0<\kappa<\varepsilon_0$, 
    \begin{align*}
        \rho(x):=\inf\{t>0: |\delta_{1/t} x|<\kappa \}
    \end{align*}
    defines a symmetric homogeneous norm on $\mathbb G$. In particular, the unit metric ball $\{x\in\mathbb G: \rho(x)<1\}$ coincides with the Euclidean ball $B_{\R^n}(0,\kappa)$.
\end{theorem}
From the above theorem ,it follows immediately that the unit ball centered at the identity with respect to the norm $\rho$ is convex. In the following we give the description of the characteristic points on the boundary of $\mathbb{B}_\rho(e,R)$ without a proof as it will follow from a similar computation in the proof of Lemma~\ref{lem:characteristic_pt1}.
\begin{lemma}
    Let $\Omega=\mathbb{B}_{\rho}(e,R)$ for some $R>0$. Then $p=(\xi_1,\xi_2)\in\partial\Omega$ is a characteristic point if and only if $\xi_1=0$ and $|\xi_2|=\kappa R^2$.
\end{lemma}
\end{example}
As discussed before, we use the following notion of central symmetry in the Euclidean sense. 
\begin{definition}\label{def:central_symmetry}
    A set $K\subset \R^n$ is called \emph{centrally symmetric} (in the Euclidean sense) with respect to $x_0\in K$ if for any $x\in K$, its reflection around $x_0$ also belongs to $K$, that is, $2x_0-x\in K$.
\end{definition}
\begin{proof}[Proof of Proposition~\ref{prop:VDC_ball}]
   Since $\mathbb G$ is a homogeneous Carnot group of step 2, by \cite[Section~3.2]{BonfiglioliLanconelliUguzzoniBook}, the group law can be written as 
   \begin{align*}
       (\xi_1,\xi_2)\star (\xi'_1, \xi'_2)=(\xi_1+\xi'_1, \xi_2+\xi'_2+B(\xi_1, \xi'_1)),
   \end{align*}
   where $B:\R^{m_1}\times \R^{m_1}\longrightarrow \R^{m_2}$ is a skew-symmetric bilinear form. This shows that for any fixed $x\in\mathbb G$, the left-translation map $L_x$ is an affine map on $\R^n$, where $n=m_1+m_2$. Therefore, $L_x$ maps convex sets to convex sets in $\R^n$. Moreover, the dilation map $\delta_\lambda: \R^n\longrightarrow \R^n$ is also affine. Hence the convexity assumption of $\mathbb B(e,1)$ implies that for any $x\in\mathbb G$ and $R>0$,
   \begin{align*}
       \mathbb{B}(x,R)=L_x\circ \delta_R (\mathbb B(e,1))
   \end{align*}
   is a convex subset of $\R^n$.

   Since $\|x\|=\|-x\|$ for any $x\in\mathbb G$, the metric ball $\mathbb{B}(e,r)$ is centrally symmetric with respect to $e$ in the Euclidean sense for any $r>0$. This means 
   \begin{align*}
    L_x(\mathbb B(e,r))=\mathbb{B}(x,r)
   \end{align*}
   is centrally symmetric with respect to $x$ (in the Euclidean sense) as $L_x:\R^n\longrightarrow\R^n$ is an affine map.

   Consider $p\in\partial\Omega$ where $\Omega=\mathbb{B}(x,R)$. Since $\Omega$ is convex, by the Hahn-Banach separating hyperplane theorem there exists a Euclidean hyperplane $H\subset \R^n$ passing through $p$ such that it divides $\R^n$ into two half-spaces $H^+$ and $H^-$, and $\Omega\subset H^-$. Therefore, $\mathring{H}^+\subset \overline{\Omega}^c$. Since any metric ball $\mathbb{B}(p,r)$ is centrally symmetric with respect to $p$, the hyperplane $H$ divides $\mathbb{B}(p,r)$ into two parts with equal volumes. As a consequence we get
   \begin{align*}
       \Vol(\overline{\Omega}^c\cap \mathbb B(p,r))\ge \Vol(\mathring{H}^+\cap \mathbb{B}(p,r))=\frac{1}{2}\Vol(\mathbb{B}(p,r))=\frac{1}{2}\Vol(\mathbb{B}(e,1)) r^Q.
   \end{align*}
   This proves the volume density condition for $\Omega$.
\end{proof}
We close this section with the following result which explains why the argument of central symmetry used above fails for Carnot groups of higher steps.
\begin{proposition}\label{prop:step_2}
    Let $\mathbb G=\R^n$ be a homogeneous Carnot group of step $k$, and let $\|\cdot\|$ be any homogeneous norm on $\mathbb G$. Assume that the metric balls with respect to this norm are centrally symmetric with respect to the center of the ball. Then $k\le 2$.
\end{proposition}
\begin{proof}
   Let us assume that $\mathbb B(x,r)$ is centrally symmetric with respect to $x$ for any $x\in\mathbb G$ and $r>0$. Then, for any $h\in\mathbb B(x,r)$, we have $2x-h\in\mathbb B(x,r)$. In other words, for any $y\in\mathbb B(e,r)$ and $x\in\mathbb G$, we have $x\star y \in\mathbb B(x,r)$ and therefore $2x-(x\star y) \in \mathbb B(x,r)$. Thus
   \begin{align*}
       F_x(y):=x^{-1}\star(2x-(x\star y))=(-x)\star(2x-(x\star y))\in \mathbb B(e,r).
   \end{align*}
   Let us choose $x=\exp(tX), y=\exp(\varepsilon Y)$, where $X,Y\in V_1$, and $\varepsilon>0$ is small enough so that $y\in \mathbb B(e,r)$. Using the Baker-Campbell-Dynkin-Hausdorff formula (see \cite[Theorem~1.2.1]{CorwinGreenleafBook}) and identifying the coordinates of $\mathbb G$ with respect to the coordinates of $\mathfrak g$ via the exponential map we get
   \begin{align*}
       F_{tX}(\varepsilon Y)=\mathrm{BCDH}(-tX, 2tX-\mathrm{BCDH}(tX,\varepsilon Y)),
   \end{align*}
   where 
   \begin{align*}
       \mathrm{BCDH}(X,Y)&=X+Y+\frac{1}{2}[X,Y]+\frac{1}{12}[X,[X,Y]]+\frac{1}{12}[Y, [Y,X]] \\
       &\ \ \ \ \ -\frac{1}{24}[Y, [X,[X,Y]]]+\cdots
   \end{align*}
   A direct computation shows that the $V_3$ component of $F_{tX}(\varepsilon Y)$ is given by
   \begin{align*}
       \frac{t\varepsilon^2}{6}[Y,[X,Y]].
   \end{align*}
   Assume that $[Y,[X,Y]]\neq 0$. Keeping $\varepsilon$ fixed and letting $t\to\infty$ then
   \begin{align*}
       \left|\frac{t\varepsilon^2}{6}[Y,[X,Y]]\right|\to \infty,
   \end{align*}
   which is a contradiction as 
   \begin{align*}
       F_{tX}(\varepsilon Y) \in \mathbb B(e,r) \quad \mbox{for any $t>0$}.
   \end{align*}
   This shows that $[Y, [X,Y]]=0$ for all $X,Y\in V_1$. Now replacing $Y$ by $Y+Z$ for arbitrary $Y,Z\in V_1$, the last identity shows
   \begin{align}\label{eq:lie_identity}
       [Y,[X,Z]]+[Z,[X,Y]]=0 \quad \mbox{for all $X,Y,Z\in V_1$}.
   \end{align}
   Let us write $T=[Y,[Z,X]]$. Since $[X,Z]=-[Z,X]$, \eqref{eq:lie_identity} yields
   \begin{align}\label{eq:lie2}
       T=[Y,[Z,X]]=[Z,[X,Y]].
   \end{align}
   Applying the permutation $(X,Y,Z)\mapsto (Z,Y,X)$ in \eqref{eq:lie_identity} yields
   \begin{align}\label{eq:lie3}
       T=[Y,[Z,X]]=[X,[Y,Z]].
   \end{align}
   Combining \eqref{eq:lie2} and \eqref{eq:lie3}, and using the Jacobi identity, we conclude that $3T=0$. Therefore,
   \begin{align*}
       [X,[Y,Z]]=0 \quad \mbox{for any $X,Y,Z\in V_1$}.
   \end{align*}
   This shows that the number of steps of $\mathbb G$ is at most $2$.
\end{proof}

\section{Finiteness of fractional horizontal perimeter}
 In this section we provide a sufficient condition for a set to have finite fractional horizontal perimeter introduced in Definition~\ref{def:fract_horiz_per}. 
\begin{proposition}\label{prop:frac_per_bound}
    Let $\Omega$ be any measurable subset of $\mathbb{G}$. Then for all $0<\alpha<1$,
    \begin{align}\label{eq:fract_per_bound}
        \Per^{(\alpha)}_{\mathcal H}(\Omega)\le C_\alpha \Vol(\Omega)^{1-\alpha} \Per_{\mathcal H}(\Omega)^\alpha
    \end{align}
    with 
    \begin{align*}
        C_\alpha=\frac{\left(2Q\right)^{\frac{\alpha}{2}}}{(1-\alpha)\Gamma(1-\frac{\alpha}{2})} ,
    \end{align*}
    where $Q$ is the homogeneous dimension of $\mathbb G$, and $m$ is the dimension of the first layer of the Lie algebra $\mathfrak{g}$.
     As a result, $\Per^{(\alpha)}_{\mathcal H}(\Omega)<\infty$ for all $\alpha\in (0,1)$ whenever $\Vol(\Omega)<\infty$ and $\Per_{\mathcal H}(\Omega)<\infty$. In particular, when $\Omega$ is bounded with a $C^1$ boundary, $\Per^{(\alpha)}_{\mathcal H}(\Omega)<\infty$ for all $\alpha\in (0,1)$.
\end{proposition}
To prove the above result, we first note the following alternative representation of the fractional horizontal perimeter.
\begin{lemma}\label{lem:fract_perimeter}
   For any $0<\alpha<1$ and a measurable set $\Omega\subset \mathbb G$,
    \begin{align*}
        \Per^{(\alpha)}_{\mathcal H}(\Omega)=\frac{\alpha}{2\Gamma(1-\frac{\alpha}{2})}\int_0^\infty (\Vol(\Omega)-H_\Omega(t)) t^{-\frac{\alpha}{2}-1} dt,
    \end{align*}
    where 
    \begin{align*}
        H_\Omega(t):=\int_{\Omega}\int_{\Omega}p(t,x,y)dx dy
    \end{align*}
    is the relative heat content of $\Omega$ corresponding to the sub-Laplacian $\Delta_{\mathcal H}$.
\end{lemma}
\begin{proof}
    Since 
    \begin{align*}
        \Vol(\Omega)-H_\Omega(t)= \int_{\Omega}\int_{\Omega ^c} p(t,x,y)dx dy=\int_{\Omega}\int_{\Omega ^c} p(t,x^{-1}y)dx dy,
    \end{align*}
    the proof of the lemma follows directly from the definition of fractional horizontal perimeter (see Definition~\ref{def:fract_horiz_per}) and Fubini's theorem.
\end{proof}

\begin{lemma}\label{prop:heat_content_upperbound}
    For any Caccioppoli set $\Omega$ and $t>0$, we have 
    \begin{align*}
        \frac{\Vol(\Omega)-H_\Omega(t)}{\sqrt{t}}\le \sqrt{\frac{Q}{2}}\Per_{\mathcal H}(\Omega).
    \end{align*}
\end{lemma}
\begin{proof} We start with the observation that for any Caccioppoli set $\Omega$ 
\begin{align*}
    \Vol(\Omega)-H_\Omega(t)=\langle P_t \mathbbm{1}_\Omega, \mathbbm{1}_{\Omega^c}\rangle_{L^2(\mathbb G, dx)}=\langle P_t\mathbbm{1}_{\Omega}-\mathbbm{1}_{\Omega}, \mathbbm{1}_{\Omega^c}\rangle_{L^2(\mathbb G, dx)}.
\end{align*}
From \cite[Theorem~3.1]{BaudoinBonnefont2016} it follows that for any Caccioppoli set $\Omega\subset\mathbb{G}$,
\begin{align}\label{eq:IP}
        t^{-\frac{1}{2}}\|P_t \mathbbm{1}_\Omega-\mathbbm{1}_\Omega\|_{L^1(\mathbb G, dx)}\le \sqrt{2Q} \Per_{\mathcal H}(\Omega),
    \end{align}
    where $P_t$ is the hypoelliptic heat semigroup on $\mathbb{G}$. Since
\begin{align*}
    \langle P_t\mathbbm{1}_{\Omega}-\mathbbm{1}_{\Omega}, \mathbbm{1}_{\Omega^c}\rangle_{L^2(\mathbb G, dx)}=\frac{1}{2}\|P_t\mathbbm{1}_\Omega-\mathbbm{1}_\Omega\|_{L^1(\mathbb G, dx)},
\end{align*}
the proof of the lemma now follows from \eqref{eq:IP}.
\end{proof}

\begin{proof}[Proof of Proposition~\ref{prop:frac_per_bound}] Let us write 
\begin{align*}
    c_\alpha=\frac{\alpha}{2\Gamma\left(1-\frac{\alpha}{2}\right)}.
\end{align*}
By Lemma~\ref{lem:fract_perimeter}, for any $\delta>0$ we have
\begin{align*}
    \Per^{(\alpha)}_{\mathcal H}(\Omega)&=c_\alpha\int_0^\delta (\Vol(\Omega)-H_\Omega(t))t^{-1-\frac{\alpha}{2}}dt + c_\alpha\int_\delta^\infty (\Vol(\Omega)-H_\Omega(t))t^{-1-\frac{\alpha}{2}}dt \\
    &=: I_1 + I_2.
\end{align*}
By Lemma~\ref{prop:heat_content_upperbound} we get 
\begin{align*}
    I_1\le c_\alpha\Per_{\mathcal H}(\Omega)\int_{0}^\delta \sqrt{\frac{Q}{2}}t^{-\frac{1}{2}-\frac{\alpha}{2}} dt= \frac{\sqrt{2Q}c_\alpha }{(1-\alpha)} \delta^{\frac{1}{2}-\frac{\alpha}{2}}\Per_{\mathcal H}(\Omega).
\end{align*}
On the other hand, since $0\le \Vol(\Omega)-H_\Omega(t)\le \Vol(\Omega)$, the second integral $I_2$ can be bounded as follows:
\begin{align*}
    I_2\le c_\alpha \Vol(\Omega) \int_\delta^\infty t^{-1-\frac{\alpha}{2}} dt=\frac{1}{\Gamma\left(1-\frac{\alpha}{2}\right)}\delta^{-\frac{\alpha}{2}}\Vol(\Omega).
\end{align*}
This shows that for any $\delta>0$,
\begin{align*}
    \Per^{(\alpha)}_{\mathcal H}(\Omega)\le \frac{\sqrt{2Q}c_\alpha }{(1-\alpha)} \delta^{\frac{1}{2}-\frac{\alpha}{2}}\Per_{\mathcal H}(\Omega)+\frac{1}{\Gamma\left(1-\frac{\alpha}{2}\right)}\delta^{-\frac{\alpha}{2}}\Vol(\Omega).
\end{align*}
Minimizing the above upper bound with respect to $\delta>0$, we conclude the proof of \eqref{eq:fract_per_bound}.

When $\Omega$ is bounded with $C^1$ boundary, it is known from \cite[Equation~(3.2)]{CapognaDanielliGarofalo1994} that
 	\begin{align*}
 		\Per_{\mathcal H}(\Omega)=\int_{\partial\Omega}\left[\sum_{i=1}^m \langle X_i,\nu\rangle^2\right]^{\frac12}d\mathcal{H}^{n-1},
 	\end{align*}
 	where $\nu$ is the Euclidean unit outward normal to the boundary, and $\mathcal{H}^{n-1}$ is the $(n-1)$-dimensional Euclidean Hausdorff measure. Hence, $\Per_{\mathcal H}(\Omega)<\infty$, which by \eqref{eq:fract_per_bound} implies that $\Per^{(\alpha)}_{\mathcal H}(\Omega)<\infty$.
\end{proof}

\section{Proof of Theorem~\ref{thm:eigenvalue_comparison} and Theorem~\ref{thm:heat_content_large}}\label{s:proof1}
We start by proving the following proposition which is the key ingredient to establish the two-sided bounds for the eigenvalues $\lambda_{n,\phi}$ in Theorem~\ref{thm:eigenvalue_comparison}.
\begin{proposition}\label{prop:unif_prob_bound}
Let $(B_t)_{t\ge 0}$ denote the hypoelliptic Brownian motion on $\mathbb{G}$. If $\Omega\subset \mathbb{G}$ is a bounded, open set satisfying the intrinsic exterior cone condition, there exists a constant $c=c(\Omega)\in (0,1)$ such that 
        \begin{align}\label{eq:A}
            \mathbb{P}_x(B_t\in \Omega)\le c \text{ for all } x\in\Omega^c, \ t>0.
        \end{align}
\end{proposition}
\begin{proof} Suppose that \eqref{eq:A} holds for all $x\in\partial \Omega$. Let $T$ denote the following stopping time 
    \begin{align*}
        T=\inf\{t>0: B_t\in \partial\Omega\}.
    \end{align*}
    Now for any $x\in\Omega^c$ and $t>0$, due to path continuity of $(B_t)_{t\ge 0}$ we note that 
    \begin{align}
        \mathbb{P}_x(B_t\in \Omega)&=\mathbb{P}_x(B_t\in \Omega, T<t) \notag \\
        &=\mathbb{P}_x(B_{T+(t-T)}\in \Omega, T<t) \notag \\
        &=\mathbb{E}_x\left[\mathbb{P}_x(B_{T+(t-T)}\in \Omega\mid \mathcal{F}_T)\mathbbm{1}_{\{T<t\}}\right] \notag\\
        &=\mathbb{E}_x\left[\mathbb{P}_{B_T}(B_{t-T}\in \Omega)\mathbbm{1}_{\{T<t\}}\right], \label{eq:conditional_prob}
    \end{align}
where the last identity follows from \cite[Exercise~8.17]{BlumenthalGetoor1968}. Due to path continuity of $B=(B_t)_{t\ge 0}$, we also have that for any $x\in\Omega^c$ and $t>0$,
\begin{align*}
    \mathbbm{1}_{\{T<t\}}\le \mathbbm{1}_{\{B_T\in\partial\Omega\}} \quad \mbox{$\mathbb P_x$-almost surely.}
\end{align*}
As $\mathbb{P}_y(B_t\in \Omega)\le c$ for all $(y,t)\in \partial \Omega\times (0,\infty)$, \eqref{eq:conditional_prob} implies that $\mathbb{P}_x(B_t\in \Omega)\le c$ for all $x\in\Omega^c$. Now, for any $x\in\partial \Omega$, we note that $\Omega\subset \mathbb{B}(x, r)$, where $r=\mathrm{diam}(\Omega)$. Therefore for any $t\ge 1$
\begin{align*}
    \mathbb{P}_x(B_t\in \Omega)&\le \mathbb{P}_x(B_t\in \mathbb{B}(x,r)) \\
    &= \mathbb{P}_e(B_t\in \mathbb{B}(e,r)) \\
    &= \mathbb{P}_e(B_1\in \mathbb{B}(e, r/\sqrt{t})) \\
    &\le \mathbb{P}_e(B_1\in \mathbb{B}(e,r))=:c_1.
\end{align*}
On the other hand, when $0\le t\le 1$, for any $x\in\partial \Omega$, let $\mathcal{C}^+_x(\beta,\rho;X_x)$ be an intrinsic horizontal cone satisfying $\mathcal{C}^+_x(\beta,\rho; X_x)\subset \Omega^c$.
\begin{align*}
    \mathbb{P}_x(B_t\notin \Omega)\ge \mathbb{P}_x(B_t\in \mathcal{C}^+_x(\beta,\rho; X_x))=\mathbb{P}_e(B_t\in \mathcal{C}^+_e(\beta,\rho; X_x)).
\end{align*}
Using the scaling property of Brownian motion and Lemma~\ref{lem:cone_dilation}, the last identity yields
\begin{align*}
     \mathbb{P}_x(B_t\notin \Omega)&\ge \mathbb{P}_e(B_1 \in\delta_{1/\sqrt{t}} \mathcal{C}^+_e(\beta,\rho; X_x)) \\
    &=\mathbb{P}_e(B_1\in \mathcal{C}^+_e(\beta,\rho/\sqrt{t}; X_x)) \\
    &\ge \mathbb{P}_e(B_1\in \mathcal{C}^+_e(\beta,\rho; X_x))\ge c_2,
\end{align*}
where the last inequality follows as $0\le t\le 1$ and 
\[
c_2=\inf_{\substack{X\in V_1 \\ \langle X,X\rangle=1}}\mathbb{P}_e(B_1\in \mathcal{C}^+_e(\beta,\rho; X))>0
\] 
by \eqref{eq:int_cone_min} in Proposition~\ref{prop:unif_lower_bound}. Taking $c=\max\{c_1,1-c_2\}$, we conclude that 
\begin{align*}
    \mathbb{P}_x(B_t\in \Omega)\le c<1 \quad \mbox{for all $x\in\partial\Omega$}.
\end{align*}
This completes the proof of the proposition.
\end{proof}
\begin{proof}[Proof of Theorem~\ref{thm:eigenvalue_comparison}] We will prove that the semigroup $P^{\phi,\Omega}_t=e^{-t\Delta^{\phi,\Omega}_{\mathcal H}}$ is compact for some $t>0$ whenever \eqref{eq:Bernstein_cond} holds. Since writing $\widetilde{\phi}(u)=\phi(u)-\phi(0)$ we have
\begin{align*}
    \Delta^\phi_{\mathcal H}=\phi(0)+\Delta^{\widetilde{\phi}}_{\mathcal H}, \quad \mbox{and} \quad \Delta^{\phi,\Omega}_{\mathcal H}=\phi(0)+\Delta^{\widetilde{\phi}, \Omega}_{\mathcal H},
\end{align*}
there is no loss of generality in assuming that $\phi(0)=0$.
Recalling Hunt's formula for killed Markov processes, for any bounded measurable function $f:\mathbb G\longrightarrow \R$ we have
\begin{align*}
    P^{\phi,\Omega}_t f(x)=P^{\phi}_t f(x)-\mathbb{E}_x\left[P^\phi_{t-\tau_\Omega} f(B^\phi_{\tau_\Omega}); \tau_\Omega<t\right],
\end{align*}
where $\tau_\Omega=\inf\{t\ge 0: B^\phi_t\in\Omega^c\}$, and $P^\phi$ denotes the semigroup associated with the subordinated hypoelliptic Brownian motion $B^\phi$. When $B^\phi_t$ admits a transition density with respect to the Lebesgue measure on $\mathbb G$ for some $t>0$, so does the process $B^{\phi,\Omega}$ killed upon exiting $\Omega$ and the above formula reads as
\begin{align*}
    p^\phi_\Omega(t,x,y)=p^\phi(t,x,y)-\mathbb{E}_x\left[p^\phi(t-\tau_\Omega,B^\phi_{\tau_\Omega},y); \tau_\Omega<t\right],
\end{align*}
where $p^\phi_\Omega(t,\cdot, \cdot)$ is the density of the killed subordinated hypoelliptic Brownian motion.
As a result, $0\le p^\phi_\Omega(t,x,y)\le p^\phi(t,x,y)$ for all $x,y\in\Omega$. It now remains to check that the transition density of $B^{\phi}_t$ exists and 
\begin{align*}
    \int_{\Omega}\int_{\Omega} p^\phi(t,x,y)^2dxdy<\infty.
\end{align*}
Invoking \eqref{eq:heat_bound_1} for the heat kernel bound for the sub-Laplacian on $\mathbb{G}$ we have 
\begin{align}\label{eq:heat_kernel_bound}
    p(t,x,y)\le C t^{-\frac{Q}{2}}
\end{align}
for all $x,y\in\mathbb G$ and $t>0$. Let $S=(S_t)_{t\ge 0}$ denote the subordinator associated with $\phi$. Then, for any nonnegative measurable function $f:\mathbb G\longrightarrow \R$, using Fubini's theorem we get
\begin{align*}
    P^\phi_t f(x)&=\int_0^\infty P_s f(x) \mathbb{P}(S_t\in ds) \\
    & = \int_{\mathbb G}\left(\int_0^\infty p(s,x,y)\mathbb{P}(S_t\in ds)\right) f(y) dy.
\end{align*}
This shows that $P^\phi_t$ admits a transition density with respect to the Lebesgue measure whenever 
\begin{align}\label{eq:subordination_bound}
    p^\phi(t,x,y):=\int_0^\infty p(s,x,y)\mathbb{P}(S_t\in ds)<\infty
\end{align}
for all $x,y\in\mathbb G$. Due to \eqref{eq:heat_kernel_bound}, a sufficient condition for \eqref{eq:subordination_bound} to hold is $\mathbb{E}\left[S^{-Q/2}_t\right]<\infty$. We recall the identity 
\begin{align*}
    S^{-Q/2}_t = \frac{1}{\Gamma\left(Q/2\right)} \int_0^\infty r^{Q/2-1} e^{-r S_t} dr,
\end{align*}
and therefore using Fubini's theorem, 
\begin{align*}
    \mathbb{E}\left[S^{-Q/2}_t\right]=\frac{1}{\Gamma\left(Q/2\right)}\int_0^\infty r^{Q/2-1} e^{-t\phi(r)} dr,
\end{align*}
 which is finite due to \eqref{eq:Bernstein_cond}. As $(x,y)\longmapsto p(t,x,y)$ is continuous, so is $(x,y)\longmapsto p^\phi(t,x,y)$ due to the dominated convergence theorem. Hence, 
 \begin{align*}
     \int_{\Omega}\int_\Omega p^\phi_\Omega(t,x,y)^2 dx dy\le \int_{\Omega} \int_{\Omega} p^\phi(t,x,y)^2 dx dy<\infty.
 \end{align*}
 This shows that $P^{\phi,\Omega}_t: L^2(\Omega, dx)\longrightarrow L^2(\Omega, dx)$ is a compact (in fact Hilbert-Schmidt) operator, which implies that $\sigma(P^{\phi,\Omega}_t)\setminus \{0\}$ consists of eigenvalues, and $0$ is the only accumulation point in the spectrum. From \cite[Theorem 2.1]{CarfagniniGordinaTeplyaev2024}, we see that $\Delta^{\phi,\Omega}_{\mathcal H}$ has a discrete spectrum. Since $P^{\phi,\Omega}_t$ is a positive operator on $L^2(\Omega, dx)$ satisfying 
 \begin{align*}
     \|P^{\phi,\Omega}_t f\|_{L^2(\Omega, dx)}\le \|f\|_{L^2(\Omega, dx)} \quad \mbox{for all $f\in L^2(\Omega, dx)$},
 \end{align*}
 $\sigma(P^{\phi,\Omega}_t)\subseteq [0,1]$. Therefore, by the spectral mapping theorem for semigroups, see \cite[\S2.6, Equation (2.7)]{EngelNagelBook2006}, we conclude that the $L^2(\Omega, dx)$-spectrum of $\Delta^{\phi,\Omega}_{\mathcal H}$ is a discrete subset of $[0,\infty)$. 
 It remains to prove that the $0$ is not an eigenvalue of $\Delta^{\phi,\Omega}_{\mathcal H}$, that is, the first eigenvalue is strictly positive. Due to \eqref{eq:Bernstein_cond}, we can choose $t_0>0$ large enough so that 
 \begin{align*}
     \frac{C}{\Gamma(Q/2)}\int_0^\infty u^{\frac{Q}{2}-1} e^{-t_0 \phi(u)}du<\frac{1}{\Vol(\Omega)^2},
 \end{align*}
where $C>0$ is the same constant as in \eqref{eq:heat_kernel_bound}. Then by identity \eqref{eq:subordination_bound}, we have $p^{\phi}(t_0,x,y)\le \Vol(\Omega)^{-2}$ for all $x,y\in\mathbb G$. Hence by \cite[Proposition~2.8]{CarfagniniGordinaTeplyaev2024}, we conclude that the first eigenvalue of $\Delta^{\phi,\Omega}_{\mathcal H}$ is strictly positive.

For any Bernstein function $\phi$ satisfying \eqref{eq:Bernstein_cond}, the eigenvalue estimate in \eqref{ineq.UpperBd.Eigenvalue_1} follows directly from \cite[p.~111--112]{ChenSong2005}.

 Using \cite[Theorem~3.4]{ChenSong2005}, for any complete Bernstein function $\phi$ satisfying \eqref{eq:Bernstein_cond} it follows that $\lambda_{n,\phi}\le \phi(\lambda_n)$ for all $n\ge 1$. On the other hand, due to the uniform bound \eqref{eq:A} in Proposition~\ref{prop:unif_prob_bound}, using \cite[Theorem~4.4]{ChenSong2005}, we conclude that $\lambda_{n,\phi}\ge c(\Omega)\phi(\lambda_n)$ for some constant $c(\Omega)$ independent of $n$ and $\phi$ and depending only on $\Omega$. This completes the proof of the theorem.
\end{proof}
We need the following result to prove Theorem~\ref{thm:heat_content_large}.
\begin{lemma}\label{lem:killed_transition_kernel}
    Let $\Omega$ be a bounded open connected subset of $\mathbb G$. Let $p^\phi_\Omega(t,\cdot, \cdot)$ be the transition density of the killed subordinated horizontal Brownian motion. Then, $p^\phi_\Omega(t,x,y)>0$ for all $x,y\in\Omega$ and $t>0$.
\end{lemma}
\begin{proof}
     Let $(X_t)_{t\ge 0}$ be the subordinated killed horizontal Brownian motion upon exiting $\Omega$, that is, 
\begin{align*}
    X_t = B^\Omega_{S_t},
\end{align*}
where $(B^\Omega_t)_{t\ge 0}$ is the horizontal Brownian motion killed upon exiting $\Omega$, and $S_t$ is the L\'evy subordinator with exponent $\phi$. Then $(X_t)_{t\ge 0}$ is a Markov process. Let $Q^{\phi,\Omega}_t$ denote the corresponding semigroup. Then using the same argument given in \cite[p.~581]{SongVondracek2003} we have
\begin{align*}
    P^{\phi,\Omega}_t f \ge Q^{\phi,\Omega}_t f
\end{align*}
for all nonnegative measurable functions $f$. This implies that $p^\phi_\Omega(t,x,y)\ge q^{\phi}_\Omega(t,x,y)$ for all $x,y\in\Omega$ where $q^{\phi}_\Omega(t,\cdot, \cdot)$ is the transition density of $Q^{\phi}_t$ that is given by
\begin{align*}
    q^{\phi}_\Omega(t,x,y)=\int_0^\infty p_{\Omega}(s,x,y)\mathbb{P}(S_t\in ds)
\end{align*}
where $p_\Omega(t,\cdot, \cdot)$ is the transition density of $B^{\Omega}_t$. Since $p_\Omega(t,x,y)>0$ for all $t>0$ and $x,y\in\Omega$, see \cite[Equation~3.16]{CarfagniniGordina2024}, we conclude that $q^{\phi}_\Omega(t,x,y)>0$ for all $x,y\in\Omega$. This completes the proof of the lemma.
\end{proof}
\begin{proof}[Proof of Theorem~\ref{thm:heat_content_large}] Let us first prove that $f_\phi(x)>0$ for all $x\in\Omega$. We note that by Theorem~\ref{thm:eigenvalue_comparison}, the spectral radius of $P^{\phi,\Omega}_t$ equals $e^{-t\lambda_{1,\phi}}$. Now consider the convex cone 
\begin{align*}
    \mathcal{K}:=\{f\in L^2(\Omega,dx): f\ge 0 \ \mbox{almost everywhere}\}\subset L^2(\Omega, dx).
\end{align*}
Then, $\mathcal{K}$ is a \emph{total cone}, that is, the closure of $\mathcal{K}-\mathcal{K}$ is $L^2(\Omega, dx)$. Since $P^{\phi,\Omega}_t$ is a Markov semigroup, we have
\begin{align}\label{eq:cone_L2}
    P^{\phi,\Omega}_t\mathcal{K}\subseteq \mathcal{K}.
\end{align}
By the Krein-Rutman theorem (see \cite{KreinRutman1948}) it follows that $f_{\phi}\in\mathcal K\setminus \{0\}$. Suppose, to the contrary that $f_\phi(x)=0$ for some $x\in\Omega$. Then,
\begin{align*}
    0=f_\phi(x)=e^{t\lambda_{1,\phi}}P^{\phi,\Omega}_t f_\phi(x)=e^{t\lambda_{1,\phi}}\int_{\Omega} f_\phi(y) p^\phi_\Omega(t,x,y)dy.
\end{align*}
Since $f_\phi\in\mathcal{K}\setminus \{0\}$, the set $A:=\{x\in\Omega: f_\phi(x)>0\}$ has positive Lebesgue measure, and the above identity implies that $p^\phi_\Omega(t,x,y)=0$ for almost every $y\in A$. This contradicts Lemma~\ref{lem:killed_transition_kernel} and hence $f_\phi(x)>0$ for all $x\in\Omega$.

The simplicity of the eigenvalue $\lambda_{1,\phi}$ would follow from the \emph{irreducibility} of the semigroup $P^{\phi,\Omega}_t$. By \cite[Example~14.11]{BaktaiBook} it is known that $(P^{\phi,\Omega}_t)_{t\ge 0}$ is irreducible if and only if its resolvent defined by
\begin{align}\label{eq:resolvent}
    R_\mu f(x):= \int_0^\infty e^{-t\mu}P^{\phi,\Omega}_t f(x) dt >0
\end{align}
whenever $f>0$ almost everywhere. We note that \eqref{eq:resolvent} indeed holds due to Lemma~\ref{lem:killed_transition_kernel}. Therefore, by \cite[Proposition~14.42(c)]{BaktaiBook} we conclude that $e^{-t\lambda_{1,\phi}}$ is a simple eigenvalue of $P^{\phi,\Omega}_t$, or equivalently, $\lambda_{1,\phi}$ is a simple eigenvalue of $\Delta^{\phi,\Omega}_{\mathcal H}$.

Since $P^{\phi,\Omega}_t : L^2(\Omega, dx)\longrightarrow L^2(\Omega, dx)$ is a compact, self-adjoint operator, it admits the following spectral expansion:
\begin{align}\label{eq:spectral_exp}
    P^{\phi,\Omega}_t f = \sum_{n=1}^\infty e^{-t\lambda_{n,\phi}} \langle f, \psi_n\rangle_{L^2(\Omega, dx)} \psi_n,
\end{align}
where $(\psi_n)_{n\ge 1}$ is a sequence of orthonormal eigenfunctions in $L^2(\Omega, dx)$ satisfying \begin{align*}P^{\phi,\Omega}_t\psi_n = e^{-t\lambda_{n,\phi}} \psi_n\end{align*} for all $t>0$ and $n\ge 1$. In particular, $\psi_1=f_\phi$. Taking $f=\mathbbm{1}_{\Omega}$ in \eqref{eq:spectral_exp}, we obtain
\begin{align*}
    \widetilde{H}^\phi_{\Omega}(t)=\langle P^{\phi,\Omega}_t\mathbbm{1}_\Omega, \mathbbm{1}_\Omega\rangle_{L^2(\Omega, dx)}=\sum_{n=1}^\infty e^{-t\lambda_{n,\phi}}\langle\psi_n,\mathbbm{1}_\Omega\rangle^2_{L^2(\Omega, dx)}.
\end{align*}
Using Parseval's identity, we obtain
\begin{align*}
    0\le \widetilde{H}^\phi_\Omega(t)-e^{-t\lambda_{1,\phi}}\langle f_\phi, \mathbbm{1}_\Omega\rangle^2_{L^2(\Omega, dx)}\le e^{-t\lambda_{2,\phi}}\Vol(\Omega).
\end{align*}
Since $\lambda_{1,\phi}$ is a simple eigenvalue of $P^{\phi,\Omega}_t$, we have $\lambda_{2,\phi}>\lambda_{1,\phi}$, and therefore,
\begin{align*}
    \lim_{t\to\infty} e^{t\lambda_{1,\phi}} \widetilde{H}^\phi_\Omega(t)=\langle f_\phi, \mathbbm{1}_\Omega\rangle^2_{L^2(\Omega, dx)}.
\end{align*}
This completes the proof of the theorem.
\end{proof}

\section{Proof of Theorem~\ref{thm:heat_content} and Corollary~\ref{cor:BBMD}}\label{s.proof2} From our discussion in Section~\ref{s.Subordination} it follows that for any $\alpha\in (0,2]$ the relative heat content is given by 
\begin{align*}
    H^{(\alpha)}_\Omega(t)=\int_\Omega\mathbb{P}_x(B^{(\alpha)}_t\in \Omega) dx=\int_{\Omega}\int_{\Omega}p^{(\alpha)}(t,x,y) dx dy,
\end{align*}
where $p^{(\alpha)}(t,\cdot, \cdot)$ is the fractional hypoelliptic heat kernel introduced in Definition~\ref{def:frac_heat_kernel}. 
For simplicity of notation, we write 
\begin{align*}
    H_\Omega(t):= H^{(2)}_\Omega(t).
\end{align*}
We first recall the result on small time asymptotics of the relative heat content for the sub-Laplacian on arbitrary Carnot groups, and this will be crucial in the proof of Theorem~\ref{thm:heat_content}. The next result follows directly from \cite[Theorem~1.1, Remark~1.2]{AgrachevRizziRossi2024}.
\begin{theorem}[Theorem~1.1 in \cite{AgrachevRizziRossi2024}]\label{t:heat_content}
    Let $\Omega$ be a bounded open subset of $\mathbb{G}$ such that $\partial \Omega$ is smooth and completely non-characteristic. Then, 
    \begin{align}\label{eq:rel_heat_asymp}
       \lim_{t\to 0} \frac{\Vol(\Omega)-H_\Omega(t)}{\sqrt{t}}=\sqrt{\frac{1}{\pi}}\Per_{\mathcal H}(\Omega).
    \end{align}
\end{theorem}
\begin{remark}
    When $\mathbb G$ is a Carnot group of step 2, one does not require any regularity condition on the boundary. It has been proved in \cite{GarofaloTralli2023} that for step-2 Carnot groups, \eqref{eq:rel_heat_asymp} holds for any bounded Caccioppoli set. The argument in their proof does not extend to arbitrary Carnot groups, while Agrachev, Rizzi, and Rossi \cite{AgrachevRizziRossi2024} proved \eqref{eq:rel_heat_asymp} for arbitrary compact sub-Riemannian manifolds equipped with a smooth volume measure, according to \cite[Remark~1.2]{AgrachevRizziRossi2024}, their method works for any Carnot group.
\end{remark}

\begin{proof}[Proof of Theorem~\ref{thm:heat_content}] Let $\eta^{(\alpha)}_t$ denote the density of the $\alpha/2$-subordinator $S^{(\alpha)}_t$. Using the self-similarity in \eqref{eq:self-similarity} we have 
\begin{align*}
    \eta_t(s)=t^{-\frac{2}{\alpha}}\eta_1(s t^{-\frac{2}{\alpha}}) \quad \mbox{for all $s\ge 0$}.
\end{align*}
Let us also denote 
\begin{align*}
    h^{(\alpha)}_\Omega(t):= \Vol(\Omega)-H^{(\alpha)}_\Omega(t)=\int_{\Omega}\int_{\Omega^c} p^{(\alpha)}(t,x,y) dx dy, \quad h_\Omega(t):= h^{(2)}_\Omega(t).
\end{align*}
Using Fubini's theorem, we can write
   \begin{align*}
       h^{(\alpha)}_\Omega(t)=\int_{\Omega}\mathbb{P}_x\left(B_{S^{(\alpha)}_t}\in \Omega^c\right) dx=\int_{0}^\infty h_\Omega(s)\eta^{(\alpha)}_t(s)ds.
   \end{align*}
   Therefore for any $t>0$ we have
    \begin{align*}
       t^{-\frac{1}{\alpha}} h^{(\alpha)}_{\Omega}(t)&=t^{-\frac{1}{\alpha}}\int_0^\infty h_\Omega(s) \eta^{(\alpha)}_t(s) ds \\
       &=\int_0^\infty \frac{h_\Omega(st^{2/\alpha})}{s^{\frac{1}{2}} t^{\frac{1}{\alpha}}} s^{\frac{1}{2}} \eta^{(\alpha)}_1(s) ds \\
       &=\int_{0}^\infty G(s,t) ds,
    \end{align*}
    where 
    \begin{align*}
        G(s,t)=\frac{h_\Omega(st^{2/\alpha})}{s^{\frac{1}{2}} t^{\frac{1}{\alpha}}} s^{\frac{1}{2}} \eta^{(\alpha)}_1(s).
    \end{align*}
    Now, by Theorem~\ref{t:heat_content}, we have
    \begin{align*}
        \lim_{t\to 0} G(s,t)=\sqrt{\frac{1}{\pi}} \Per_{\mathcal H}(\Omega) s^{\frac12} \eta^{(\alpha)}_1(s).
    \end{align*}
    Also, by Lemma~\ref{prop:heat_content_upperbound}, for all $s,t>0$,
    \begin{align*}
        0\le G(s,t)\le \sqrt{2Q}\Per_{\mathcal H}(\Omega) s^{\frac12} \eta^{(\alpha)}_1(s).
    \end{align*}
    Applying the dominated convergence theorem it follows that for $1<\alpha<2$
    \begin{align*}
        \lim_{t\to 0} \int_0^\infty G(s,t) ds&=\int_0^\infty \sqrt{\frac{1}{\pi}}\Per_{\mathcal H}(\Omega) s^{\frac12} \eta^{(\alpha)}_1(s) ds \\
        &=\frac{1}{\pi} \Gamma\left(1-\frac{1}{\alpha}\right) \Per_{\mathcal H}(\Omega).
    \end{align*}

When $\alpha=1$, we follow an argument similar to that used in the proof of \cite[Proposition~3.3]{ParkSong2019} with some modifications in our setting. From \cite[Lemma~3.2]{ParkSong2019}, for any $\delta>0$ one has 
    \begin{align}\label{eq:S_asymp}
        \lim_{t\to 0}\frac{\mathbb{E}\left[\left(S^{(1/2)}_1\right)^{\frac12}, 0<S^{(1/2)}_1<\delta t^{-2}\right]}{\log(1/t)}=\frac{1}{\sqrt{\pi}}.
    \end{align}
    Now, for any bounded $\Omega$ we write 
    \begin{align*}
        \Vol(\Omega)-H^{(1)}_\Omega(t)&=\int_0^\infty h_\Omega(st^2)\eta^{(1/2)}_1(s)ds \\
        &=\int_0^{\delta t^{-2}} \frac{h_\Omega(st^2)}{s^{\frac12}t} s^{\frac12} \eta^{(1/2)}_1(s) ds+\int_{\delta t^{-2}}^\infty h_\Omega(st^2)\eta^{(1/2)}_1(s) ds \\
        &=: I_1(t) + I_2(t),
    \end{align*}
where $\delta>0$ is to be determined. Since by Theorem~\ref{t:heat_content} 
\begin{align*}
    \lim_{t\to 0}\frac{h_\Omega(t)}{\sqrt{t}}=\frac{1}{\sqrt{\pi}}\Per_{\mathcal H}(\Omega),
\end{align*}
for any $\epsilon>0$, there exists $\delta>0$ such that 
\begin{align*}
    \frac{1}{\sqrt{\pi}}\Per_{\mathcal H}(\Omega)-\epsilon<\left|\frac{h_\Omega(t)}{\sqrt{t}}\right|< \frac{1}{\sqrt{\pi}}\Per_{\mathcal H}(\Omega)+\epsilon \quad \mbox{for all $0<t<\delta$}.
\end{align*}
As a result, using \eqref{eq:S_asymp} we get 
\begin{align*}
    \limsup_{t\to 0} \frac{I_1(t)}{t \log(1/t)} & \le \left(\frac{1}{\sqrt{\pi}}\Per_{\mathcal H}(\Omega)+\epsilon\right)\limsup_{t\to 0}\frac{1}{\log(1/t)}\int_0^{\delta t^{-2}} s^{\frac12} \eta^{(1/2)}_1(s)ds \\
    &=\frac{1}{\sqrt{\pi}}\left(\frac{1}{\sqrt{\pi}}\Per_{\mathcal H}(\Omega)+\epsilon\right).
\end{align*}
Similarly, we also get 
\begin{align*}
    \liminf_{t\to 0}\frac{I_1(t)}{t \log(1/t)}\ge \frac{1}{\sqrt{\pi}} \left(\frac{1}{\sqrt{\pi}}\Per_{\mathcal H}(\Omega)-\epsilon\right).
\end{align*}
As $\epsilon>0$ is arbitrary, we conclude that 
\begin{align*}
    \lim_{t\to 0} \frac{I_1(t)}{t\log(1/t)}=\frac{1}{\pi}\Per_{\mathcal H}(\Omega).
\end{align*}
On the other hand, since $h_{\Omega}(t)= \Vol(\Omega)-H_\Omega(t)\le \Vol(\Omega)$ for all $t>0$, using \eqref{eq:stable_asymptotic} we obtain 
\begin{align*}
    I_2(t)\le c\int_{\delta t^{-2}}^\infty \Vol(\Omega) s^{-3/2} ds= \frac{2c}{\sqrt{\delta}}\Vol(\Omega) t
\end{align*}
for some constant $c>0$. As a result, 
\begin{align*}
    \lim_{t\to 0}\frac{I_2(t)}{t\log(1/t)}=0.
\end{align*}
Therefore, we have
\begin{align*}
\lim_{t\to 0} \frac{\Vol(\Omega) - H^{(\alpha)}_\Omega(t)}{t\log(1/t)}=\frac{1}{\pi} \Per_{\mathcal H}(\Omega).
\end{align*}

Let us now consider the case when $0<\alpha<1$. Consider $\Omega\subset \mathbb G$ such that $\Per^{(\alpha)}_{\mathcal H}(\Omega)<\infty$. Since 
\begin{align*}
    \Vol(\Omega)-H^{(\alpha)}_\Omega(t)=\int_{\Omega}\int_{\Omega^c} p^{(\alpha)}(t,x^{-1}y) dy dx,
\end{align*}
from the estimates in \eqref{eq:frac_heat_bound} and \eqref{eq:heat_asymp} in Proposition~\ref{prop:heat_asymp} along with the dominated convergence theorem it follows that 
\begin{align*}
    \lim_{t\to 0}\frac{\Vol(\Omega)-H^{(\alpha)}_\Omega(t)}{t}&=\lim_{t\to 0}\int_{\Omega}\int_{\Omega^c} \frac{p^{(\alpha)}(t, x^{-1}y)}{t} dx dy \\
    &=\int_\Omega \int_{\Omega^c} \frac{1}{\|x^{-1}y\|_\alpha^{Q+\alpha}} dx dy\\
    &=\Per^{(\alpha)}_{\mathcal H}(\Omega).
\end{align*}
This concludes the proof of the theorem.
\end{proof}

\begin{proof}[Proof of Corollary~\ref{cor:BBMD}]
    We first prove the limit when $\alpha\nearrow 1$. Let $\Omega$ be a bounded open set having $C^\infty$ boundary with no characteristic points, and let us denote as before $h_\Omega(t)=\Vol(\Omega)-H_\Omega(t)$.  Then by Theorem~\ref{t:heat_content}, for any $\varepsilon>0$ there exists $\delta>0$ such that 
    \begin{align}\label{eq:h_omega_bound}
        \left|\frac{h_\Omega(t)}{\sqrt{t}}-\frac{1}{\sqrt{\pi}}\Per_{\mathcal H}(\Omega)\right|\le \varepsilon \quad \mbox{for all $0<t<\delta$}.
    \end{align}
    Using Lemma~\ref{lem:fract_perimeter} we can therefore write 
    \begin{align}\label{eq:per_decomp}
        \Per^{(\alpha)}_{\mathcal H}(\Omega)=c_\alpha \int_0^\delta \frac{h_\Omega(t)}{\sqrt{t}} t^{-\frac{\alpha}{2}-\frac{1}{2}} dt + c_\alpha\int_\delta^\infty h_\Omega(t) t^{-\frac{\alpha}{2}-1} dt,
    \end{align}
    where 
    \begin{align*}
        c_\alpha=\frac{\alpha}{2\Gamma(1-\frac{\alpha}{2})}.
    \end{align*}
    Since $0\le h_\Omega(t)\le \Vol(\Omega)$ for all $t>0$, we have 
    \begin{align}\label{eq:lim_bound_1}
        \limsup_{\alpha\nearrow 1} (1-\alpha)c_\alpha\int_\delta^\infty h_\Omega(t) t^{-\frac{\alpha}{2}-1} dt\le \limsup_{\alpha\nearrow 1}\frac{(1-\alpha)}{\Gamma(1-\frac{\alpha}{2})}\Vol(\Omega)\delta^{-\frac{\alpha}{2}}=0.
    \end{align}
    On the other hand noting that 
    \begin{align*}
        \lim_{\alpha\nearrow 1} c_\alpha=\frac{1}{2\sqrt{\pi}},
    \end{align*}
    \eqref{eq:h_omega_bound} implies 
    \begin{equation}\label{eq:lim_bound_2}
    \begin{aligned}
        \limsup_{\alpha\nearrow 1}(1-\alpha)c_\alpha \int_{0}^\delta \frac{h_\Omega(t)}{\sqrt{t}} t^{-\frac{\alpha}{2}-\frac12} dt\le \left(\frac{1}{\sqrt{\pi}}\Per_\mathcal{H}(\Omega)+\varepsilon\right)\frac{1}{\sqrt{\pi}}, \\
        \liminf_{\alpha\nearrow 1}(1-\alpha)c_\alpha \int_{0}^\delta \frac{h_\Omega(t)}{\sqrt{t}} t^{-\frac{\alpha}{2}-\frac12} dt\ge \left(\frac{1}{\sqrt{\pi}}\Per_\mathcal{H}(\Omega)-\varepsilon\right)\frac{1}{\sqrt{\pi}}.
    \end{aligned}
    \end{equation}
    Since $\varepsilon>0$ is arbitrary, combining \eqref{eq:lim_bound_1} and \eqref{eq:lim_bound_2}, we conclude that 
    \begin{align*}
        \lim_{\alpha\nearrow 1} (1-\alpha)\Per^{(\alpha)}_{\mathcal H}(\Omega)=\frac{1}{\pi}\Per_{\mathcal H}(\Omega).
    \end{align*}
    
    Let us now prove the limit when $\alpha\searrow 0$. From the heat kernel estimate in \eqref{eq:heat_bound_1} it follows that 
    \begin{align*}
        p(t,x,y)\le ct^{-\frac Q2} \quad \mbox{for all $x,y\in\mathbb G, \ t>0$}.
    \end{align*}
    As a result, 
    \begin{align*}
        H_\Omega(t)=\int_{\Omega}\int_{\Omega} p(t,x,y)dx dy\le ct^{-\frac Q2}\Vol(\Omega)^2,
    \end{align*}
    which shows that
    \begin{align*}
        \lim_{t\to \infty} H_\Omega(t)=0.
    \end{align*}
    Therefore, for any $\varepsilon>0$, there exists $\delta>0$ such that 
    \begin{align}\label{eq:heat_volume}
        \left|h_{\Omega}(t)-\Vol(\Omega)\right|\le \varepsilon \quad \mbox{for all $t>\delta$}.
    \end{align}
    We note that 
    \begin{align*}
        \lim_{\alpha\searrow 0} c_\alpha=0, \quad \mbox{and} \quad \lim_{\alpha\searrow 0} \frac{c_\alpha}{\alpha/2}=1.
    \end{align*}
     Using the same decomposition as in \eqref{eq:per_decomp} and invoking Lemma~\ref{prop:heat_content_upperbound} for any $\Omega$ satisfying $\Per_{\mathcal H}(\Omega)<\infty$ we get 
    \begin{align*} 
        \limsup_{\alpha\searrow 0}c_\alpha \int_0^\delta \frac{h_\Omega(t)}{\sqrt{t}} t^{-\frac{\alpha}{2}-\frac{1}{2}} dt\le \sqrt{2Q}\Per_{\mathcal H}(\Omega)\limsup_{\alpha\searrow 0} \frac{1}{2}c_\alpha(1-\alpha)=0.
    \end{align*}
    On the other hand, using \eqref{eq:heat_volume} we obtain
    \begin{align*}
        \limsup_{\alpha\searrow 0} c_\alpha\int_\delta^\infty h_\Omega(t) t^{-\frac{\alpha}{2}-1} dt\le (\Vol(\Omega)+\varepsilon)\limsup_{\alpha\searrow 0}\frac{c_\alpha}{\alpha/2}\delta^{-\frac{\alpha}{2}}=\Vol(\Omega)+\varepsilon, \\
        \liminf_{\alpha\searrow 0} c_\alpha\int_\delta^\infty h_\Omega(t) t^{-\frac{\alpha}{2}-1} dt\ge (\Vol(\Omega)-\varepsilon)\liminf_{\alpha\searrow 0}\frac{c_\alpha}{\alpha/2}\delta^{-\frac{\alpha}{2}}=\Vol(\Omega)-\varepsilon
    \end{align*}
    Since $\varepsilon>0$ is arbitrary, we conclude that 
    \begin{align*}
        \lim_{\alpha\searrow 0} \Per^{(\alpha)}_{\mathcal H}(\Omega)=\Vol(\Omega).
    \end{align*}
    This concludes the proof of the corollary.
\end{proof}

\section{Proof of Theorem~\ref{thm:spectral_heat}} \label{s:proof3}
An essential tool for proving \eqref{eq:lim_0_1} in Theorem~\ref{thm:spectral_heat} is the following powerful theorem due to Ikeda and Watanabe \cite{Ikeda-Watanabe}.

\begin{theorem}[Ikeda-Watanabe, 1962] \label{thm:Ikeda-Watanabe} Let $X=(X_t)_{t\ge 0}$ be a Feller process defined on a locally compact separable metric space $(M,d)$. Assume that there exists a positive kernel $\Pi(x,E)$, $x\in M$, $E\in\mathcal{B}(M)$ such that 
	\begin{enumerate}[leftmargin=*]
		\item \label{cond1} $\Pi(x,E)<\infty$ if $d(x,E)>0$, 
		\item \label{cond2} for any bounded $f\in C(M)$ and a bounded open set $D$ with $d(D,\supp(f))>0$, 
		\begin{align*}
			\frac{P_t f(x)}{t} \mbox{ is uniformly bounded in $x\in D$, $t>0$},
		\end{align*}
		and 
		$\lim_{t\downarrow 0}\frac{P_t f(x)}{t}=\int_M f(y)\Pi(x,dy)$ for every $x\in D$,
	\end{enumerate}
	then for any open set $\Omega\subset M$, $E\in\mathcal{B}(M)$ satisfying $d(E,\Omega)>0$, and $\lambda>0$,
	\begin{align}\label{eq:Laplace_transform}
		\mathbb{E}_x\left[e^{-\lambda \tau_\Omega}, X_{\tau_\Omega}\in E\right]=\int_\Omega G^\lambda_\Omega(x, dy)\Pi(y,E),
	\end{align}
	where $G^\lambda_\Omega(x,E)=\int_0^\infty e^{-\lambda s} p_\Omega(s,x,E)ds$. Therefore, the joint distribution of $(\tau_\Omega, X_{\tau_\Omega})$ on the event $\{X_{\tau_\Omega}\neq X_{\tau_\Omega -}\}$ is given by 
	\begin{align*}
		\mathbb{P}_x\left(\tau_\Omega\in ds, X_{\tau_\Omega}\in dz\right)=\int_\Omega p_\Omega(s,x,dy)\Pi(y, dz).
	\end{align*} 
	Moreover, \eqref{eq:Laplace_transform} holds for $\lambda=0$ if $\mathbb{E}_x[\tau_D]<\infty$.
\end{theorem}
In the following lemma we identify the positive kernel $\Pi$ in the case of the subordinated hypoelliptic Brownian motion $B^{(\alpha)}_t = B_{S^{(\alpha)}_t}$.

\begin{lemma}\label{lem:IW-check}
	For any $0<\alpha<2$, the subordinated process $B^{(\alpha)}$ satisfies conditions \eqref{cond1}-\eqref{cond2} of Theorem~\ref{thm:Ikeda-Watanabe} with 
	\begin{align*}
		\Pi(x,dy)=\frac{1}{\|x^{-1}y\|^{Q+\alpha}_\alpha}dy.
	\end{align*}
\end{lemma}
\begin{proof}
	Due to the equivalence of homogeneous metrics on $\mathbb{G}$, let us fix $d=d_c$, the Carnot-Carath\'eodory metric. For any $E\in\mathcal{B}(\mathbb G)$ and $x\in\mathbb{G}$ with $d_c(x,E)>0$, let us write $r_x=d_c(x,E)/2$. Then, $E\subset \mathbb B(x,r_x)^c$. Due to the equivalence between the quasinorm $\|\cdot\|_\alpha$ and the Carnot-Carath\'eodory norm $d_c(e,\cdot)$, we obtain
	\begin{align*}
		\int_E \frac{1}{\|x^{-1}y\|^{Q+\alpha}_\alpha} dy\le \int_{\mathbb B(x,r_x)^c}\frac{c_1}{d_c(x,y)^{Q+\alpha}} dy=\int_{\mathbb B(e,r_x)^c}\frac{c_1}{d_c(e,y)^{Q+\alpha}} dy
	\end{align*}
	for some constant $c_1>0$. Let us write 
    \begin{align*}
        A_k = \mathbb{B}(e,2^{k+1}r_x)\setminus \mathbb{B}(e,2^{k}r_x).
    \end{align*}
    Then, 
    \begin{align*}
        \int_{\mathbb B(e,r_x)^c}\frac{c_1}{d_c(e,y)^{Q+\alpha}} dy&=\sum_{k=0}^\infty \int_{A_k}\frac{c_1}{d_c(e,y)^{Q+\alpha}} dy \\
        &\le c_1\sum_{k=0}^\infty (2^k r_x)^{-Q-\alpha} \Vol(\mathbb{B}(e, 2^{k+1}r_x)) \\
        &=c_1\Vol(\mathbb{B}(e,1))\sum_{k=0}^\infty 2^{-k\alpha} r^{-\alpha}_x \\
        &=\frac{c_1\Vol(\mathbb{B}(e,1))}{1-2^{-\alpha}} r^{-\alpha}_x.
    \end{align*}
    This shows that 
	\begin{align}\label{eq:integral_estimate}
		\int_{E}\frac{1}{\|x^{-1}y\|^{Q+\alpha}_\alpha} dy \le \frac{c_2}{r^\alpha_x}
	\end{align}
    with $c_2=c_1\Vol(\mathbb{B}(e,1))/(1-2^{-\alpha})$, and hence \eqref{cond1} is satisfied.
    
	Let $D$ be a bounded open subset of $\mathbb{G}$ and $f$ be a bounded continuous function on $\mathbb{G}$ with $d_c(D, S_f)>0$, where $S_f=\mathrm{supp}(f)$. Let $P^{(\alpha)}=P^{(\alpha)}_t$ denote the semigroup associated with $B^{(\alpha)}$. Then for any $x\in\mathbb G$, 
	\begin{align*}
		P^{(\alpha)}_t f(x) =\int_{\mathbb G} f(y) p^{(\alpha)}(t,x,y)dy=\int_{\mathbb G} f(y) p^{(\alpha)}(t,x^{-1}y)dy.
	\end{align*}
	Using \eqref{eq:frac_heat_bound} in Proposition~\ref{prop:heat_asymp}, for any $x\in D$ we obtain
	\begin{align*}
		\left|\frac{P^{(\alpha)}_t f(x)}{t}\right|\le \int_{S_f}\frac{c_3\|f\|_\infty}{\|x^{-1}y\|^{Q+\alpha}_\alpha}dy\le \frac{2^\alpha c_2 c_3\|f\|_\infty}{d_c(x,S_f)^\alpha}\le \frac{2^\alpha c_2c_3 \|f\|_\infty}{d_c(D,S_f)^\alpha}
	\end{align*}
	where the penultimate inequality follows from \eqref{eq:integral_estimate}. Since $d_c(D,S_f)>0$, we conclude that $P^{(\alpha)}_t f(x)/t$ is uniformly bounded with respect to $x\in D$ and $t>0$. Finally, using \eqref{eq:heat_asymp} in Proposition~\ref{prop:heat_asymp} and the dominated convergence theorem, we conclude that \eqref{cond2} is also satisfied. This completes the proof of the lemma.
\end{proof}
In the next result we prove that under some regularity conditions on the boundary of a domain in the Carnot group, the subordinated hypoelliptic Brownian motion generated by the fractional sub-Laplacian operator does not hit the boundary upon exiting the domain. This result is known for Euclidean spaces from \cite{Bogdan1997}, and the following result is a generalization of this fact.
\begin{proposition}\label{prop:skip_boundary}
	Assume that $\Omega$ satisfies the volume density condition defined in Definition~\ref{def:VDC}. Then for any $\alpha\in (0,2)$,
	\begin{align*}
		\mathbb{P}_x\left( B^{(\alpha)}_{\tau_\Omega} \in \partial\Omega\right)=0.
	\end{align*}
\end{proposition}
The proof requires several intermediate results which are stated in a few lemmas below.
	\begin{lemma}\label{lem:finite_exp}
		Let $D\subset\mathbb{G}$ be an open ball with respect to some left-translation invariant homogeneous metric on $\mathbb{G}$. Then for $x\in D$, $\mathbb{E}_x[\tau_D]<\infty$.
	\end{lemma}
	\begin{proof}
		Due to left-translation invariance of the subordinated process $B^{(\alpha)}_t=B_{S^{(\alpha)}_t}$, without loss of generality we can assume that $D$ is centered at the identity element $e$, that is, $D=\mathbb B(e,r)$ for some $r>0$. It is easy to see that the horizontal coordinates of $B^{(\alpha)}$ have the same distribution as the subordinated Brownian motion on $\R^{m}$, where $m=\dim(V_1)$, $V_1$ being the first layer of the Lie algebra $\mathfrak{g}$. Since the homogeneous quasinorm defined by 
        \begin{align*}
            N(x)=\sup\{|x_j|^{1/j}, \ j=1,\ldots, k, \ x=(x_1,\ldots, x_k)\}
        \end{align*}
        is equivalent to any homogeneous norm on $\mathbb{G}$, for any $t>0$ we have
		\begin{align}\label{eq:exit_time_ineq}
			\mathbb{P}_x\left(\tau_D>t\right)\le \mathbb{P}_{x_1}\left(\tau'_{cr}>t\right)
		\end{align}
		for some $c>0$, where $x_1\in\R^{m}$ denotes the horizontal coordinate of $x\in\mathbb{G}$, and
		\begin{align*}
			\tau'_r=\inf\{t\ge 0: |\pi_1(B^{(\alpha)}_t)|>r\},
		\end{align*}
		where $\pi_1(B^{(\alpha)}_t)$ is the horizontal coordinate of $B^{(\alpha)}_t$.
		In fact, $\tau'_r$ is the first exit time of the  $\alpha$-stable isotropic L\'evy process from a Euclidean ball of radius $r$ centered at the origin (in $\R^{m}$). Therefore, \eqref{eq:exit_time_ineq} implies that $\mathbb{E}_x[\tau_D]\le \mathbb{E}_{x_1}[\tau'_{cr}]$. It was proved by Getoor \cite{Getoor1961b} that for any $\alpha\in (0,2)$ and $r>0$,
		\begin{align*}
			\mathbb{E}_{x_1}[\tau'_r]=\Gamma(m/2)[2^\alpha\Gamma(1+\alpha/2)\Gamma((m+\alpha)/2)]^{-1} (r^2-|x_1|^2)^{\alpha/2}_+,
		\end{align*}
		which is in particular finite for all $\alpha\in (0,2)$. This completes the proof of the lemma.
	\end{proof}
	\begin{lemma}\label{lem:skip_boundary1}
		Let $D$ be any open ball in $\Omega$ such that $d(D,\Omega^c)>0$. Then, 
		\begin{align*}
			\mathbb{P}_x\left( B^{(\alpha)}_{\tau_D} \in \partial\Omega\right)=0 \quad \text{for all $x\in D$}.
		\end{align*}
	\end{lemma}
	\begin{proof}
		To prove this, we again resort to Theorem~\ref{thm:Ikeda-Watanabe}. Since by Lemma~\ref{lem:finite_exp} $\mathbb{E}_x[\tau_D]<\infty$ for any $x\in D$, plugging $\lambda=0$ in \eqref{eq:Laplace_transform} we get that for any $A\subset\mathcal{B}(\mathbb{G})$ with $d(D,A)>0$,
		\begin{align}\label{eq:X_D_boundary}
			\mathbb{P}_x\left( B^{(\alpha)}_{\tau_D}\in A\right) = \int_A \int_{D}\frac{G^{(\alpha)}_{D}(x,y)}{\|y^{-1}z\|^{Q+\alpha}_\alpha}dy dz,
		\end{align}
		where $G_D$ is the Green's function defined by 
		\begin{align*}
			G^{(\alpha)}_D(x,y)=\int_0^\infty p^{(\alpha)}_D(s,x,y)ds.
		\end{align*}
		Since $d(D,\partial\Omega)>0$ and $\Vol(\partial\Omega)=0$ (see Lemma~\ref{prop:VDC_conseq}), by \eqref{eq:X_D_boundary}, $\mathbb{P}_x\left(B^{(\alpha)}_{\tau_D}\in \partial \Omega\right) =0$ for all $x\in D$.
	\end{proof}
	\begin{lemma}\label{lem:unif_bound} Let $x\in\Omega$ and $D_x=\mathbb B(x,r_x/2)$ where $r_x=d(x,\partial\Omega)$. Then,
		\begin{align}
			\inf_{x\in \Omega}\mathbb{P}_x\left(B^{(\alpha)}_{\tau_{D_x}}\in\Omega^c\right)>0.
		\end{align}
	\end{lemma}
	\begin{proof}
		Due to the volume density condition, there exists a constant $c_1>0$ such that 
		\begin{align*}
			\Vol\left(\overline{\Omega}^c\cap \mathbb B(x, 2r_x\right)\ge \Vol\left(\overline{\Omega}^c\cap \mathbb B(m_x, r_x\right) \ge c_1 r^Q_x,
		\end{align*}
        where $m_x$ is a point on the boundary nearest to $x$.
		Let us write $\Omega'_1=\overline{\Omega}^c\cap B(x, 2r_x)$.
		By \eqref{eq:X_D_boundary} we have
		\begin{align}\label{eq:green1}
			\mathbb{P}_x\left(B^{(\alpha)}_{\tau_{D_x}}\in \Omega'_1\right)= \int_{\Omega'_1}\int_{D_x} \frac{G^{(\alpha)}_{D_x}(x,y)}{\|y^{-1}z\|^{Q+\alpha}_\alpha} dy dz.
		\end{align}
		Due to the scaling property of the Dirichlet fractional sub-Laplacian $\Delta^{\alpha,\Omega}_\mathcal{H}$, it follows that for all $r,t>0$ and $x,y\in\Omega$
		\begin{align}
			p^{(\alpha)}_\Omega(t,x,y)=r^{-Q}p^{(\alpha)}_{\delta_{\frac1r}\Omega}\left(\frac{t}{r^\alpha}, \delta_{\frac1r}x, \delta_{\frac1r} y\right),
		\end{align}
		where $\delta$ is the dilation on $\mathbb{G}$, which shows that the Green's function has the following scaling property:
		\begin{align}\label{eq:green_dilation}
			G^{(\alpha)}_{\Omega}(x,y)=r^{\alpha-Q} G^{(\alpha)}_{\delta_{\frac1r}\Omega}\left(\delta_{\frac1r}x, \delta_{\frac1r} y\right) \quad \forall x,y\in\Omega.
		\end{align}
		Using the left-translation invariance of $B^{(\alpha)}$ we have 
		\begin{align}
			G^{(\alpha)}_{D_x}(x,y)=G^{(\alpha)}_{\mathbb B(e,r_x/2)} (e,x^{-1}y).
		\end{align}
        As the quasinorm $\|\cdot\|_\alpha$ is equivalent to the Carnot-Carath\'eodory norm, we have $\|y^{-1} z\|_\alpha\le c r_x$ for some constant $c$ independent of $x$. Therefore, using \eqref{eq:green1} and \eqref{eq:green_dilation} we obtain
        \begin{equation}\label{eq:green_prob}
		\begin{aligned}
			\mathbb{P}_x\left(B^{(\alpha)}_{\tau_{D_x}}\in \Omega^c\right)&\ge \mathbb{P}_x\left(B^{(\alpha)}_{\tau_{D_x}}\in \Omega'_1\right) \\
			&\ge \int_{\Omega'_1}\int_{\mathbb B(e,r_x/2)} \frac{G^{(\alpha)}_{\mathbb B(e,r_x/2)}(e,y)}{(cr_x)^{Q+\alpha} }dy dz \\
			& = \int_{\Omega'_1}\int_{\mathbb B(e,r_x/2)} r_x^{\alpha-Q}\frac{G^{(\alpha)}_{\mathbb B(e,1/2)}(e,\delta_{1/r_x}y)}{(cr_x)^{Q+\alpha} }dy dz \\
			& = (cr_x)^{-Q}\int_{\Omega'_1}\int_{\mathbb B(e,1/2)} G^{(\alpha)}_{\mathbb B(e,1/2)}(e,y) dy dz,
		\end{aligned}
        \end{equation}
        where the last equality holds due to the change of variable $y\mapsto \delta_{1/r_x}y$.
		Since 
        \begin{align*}
        \int_{\mathbb B(e,1)} G^{(\alpha)}_{\mathbb B(e,1/2)}(e,y)dy=\mathbb{E}_e[\tau_{\mathbb B(e,1/2)}]
        \end{align*}
        and $\Vol(\Omega'_1)\ge c_1 r^Q_x$ for some $c_1>0$, we conclude the proof of the lemma using \eqref{eq:green_prob}.
	\end{proof}  
	\begin{proof}[Proof of Proposition~\ref{prop:skip_boundary}] We follow a similar argument used in the proof of \cite[Lemma~A.1]{Bogdan_et_al2020} adapted to our setting. For any $x\in\Omega$, let us denote $D_x=\mathbb B(x,r_x/2)\in \Omega$ with $r_x=d(x,\partial\Omega)$. Then for any $x\in \Omega$ we have
        \begin{align*}
            \mathbb P_x\left(B^{(\alpha)}_{\tau_\Omega}\in \partial\Omega\right)=\mathbb P_x\left(B^{(\alpha)}_{\tau_{D_x}}\in \partial\Omega\right)+\mathbb{P}_x\left(B^{(\alpha)}_{\tau_{D_x}}\in \Omega, B^{(\alpha)}_{\tau_\Omega}\in \partial\Omega\right).
        \end{align*}
        The first term is identically equal to $0$ due to Lemma~\ref{lem:skip_boundary1}. For the second term, using strong Markov property for $B^{(\alpha)}$ we obtain
        \begin{equation}\label{eq:Markov_eq}
        \begin{aligned}
            &\ \ \mathbb{P}_x\left(B^{(\alpha)}_{\tau_{D_x}}\in \Omega, B^{(\alpha)}_{\tau_\Omega}\in \partial\Omega\right) \\
            &=\mathbb{E}_x\left[\mathbb{P}_{B^{(\alpha)}_{\tau_{D_x}}}\left(B^{(\alpha)}_{\tau_\Omega}\in\partial \Omega\right), \ B^{(\alpha)}_{\tau_{D_x}}\in\Omega\right]
        \end{aligned}
        \end{equation}
    By Lemma~\ref{lem:unif_bound}, there exists $\gamma\in [0,1)$ such that
    \begin{align*}
        \sup_{x\in\Omega}\mathbb P_x\left(B^{(\alpha)}_{\tau_{D_x}}\in\Omega\right)\le \gamma.
    \end{align*}
    Therefore, \eqref{eq:Markov_eq} yields,
    \begin{align*}
        \sup_{x\in\Omega} \mathbb{P}_x\left(B^{(\alpha)}_{\tau_{\Omega}}\in\partial\Omega\right) \le \gamma \sup_{x\in\Omega} \mathbb{P}_x\left(B^{(\alpha)}_{\tau_{\Omega}}\in\partial\Omega\right). 
    \end{align*}
    This shows that $\mathbb{P}_x\left(B^{(\alpha)}_{\tau_{\Omega}}\in\partial\Omega\right)=0$ for all $x\in\Omega$, which completes the proof of the proposition.
    \end{proof}
    
We are now ready to prove Theorem~\ref{thm:spectral_heat}.
\begin{proof}[Proof of Theorem~\ref{thm:spectral_heat}]
	We only need to prove \eqref{eq:lim_0_1}. We adapt the technique in \cite{GrzywnyParkSong2019} in the context of L\'evy processes in Euclidean spaces. We note that for any $x\in\Omega$, 
	\begin{align*}
		\mathbb{P}_x(\tau_\Omega>t)=\mathbb{P}_x(B^{(\alpha)}_t\in\Omega)-\mathbb{P}_x(\tau_\Omega<t, B^{(\alpha)}_t\in\Omega).
	\end{align*}
	Using the strong Markov property and \cite[Exercise~8.17]{BlumenthalGetoor1968}, we get 
	\begin{align*}
		\mathbb{P}_x(\tau_\Omega<t, B^{(\alpha)}_t\in\Omega)&=\mathbb{E}_x\left[\tau_\Omega< t, \mathbb{P}_{B^{(\alpha)}_{\tau_\Omega}}(B^{(\alpha)}_{t-\tau_\Omega}\in\Omega) \right] \\
		& = \mathbb{E}_x\left[\tau_\Omega< t, B^{(\alpha)}_{\tau_\Omega}\in\overline{\Omega}^c, \mathbb{P}_{B^{(\alpha)}_{\tau_\Omega}}(B^{(\alpha)}_{t-\tau_\Omega}\in\Omega)\right] \\
		& +  \mathbb{E}_x\left[\tau_\Omega< t, B^{(\alpha)}_{\tau_\Omega}\in\partial\Omega, \mathbb{P}_{B^{(\alpha)}_{\tau_\Omega}}(B^{(\alpha)}_{t-\tau_\Omega}\in\Omega)\right].
	\end{align*}
	Let us define 
	\begin{align*}
		I(t)&=\int_{\Omega} \mathbb{E}_x\left[\tau_\Omega< t, B^{(\alpha)}_{\tau_\Omega}\in\overline{\Omega}^c, \mathbb{P}_{B^{(\alpha)}_{\tau_\Omega}}(B^{(\alpha)}_{t-\tau_\Omega}\in\Omega)\right] dx, \\
		II(t)&= \int_{\Omega}\mathbb{E}_x\left[\tau_\Omega< t, B^{(\alpha)}_{\tau_\Omega}\in\partial\Omega, \mathbb{P}_{B^{(\alpha)}_{\tau_\Omega}}(B^{(\alpha)}_{t-\tau_\Omega}\in\Omega)\right] dx.
	\end{align*}
	Then, 
	\begin{align}\label{eq:H_decomp}
		\widetilde{H}^{(\alpha)}_\Omega(t)=H^{(\alpha)}_\Omega(t) - I(t) -II(t) \text{ for all $t\ge 0$.}
	\end{align}
	We first prove that $\lim_{t\to 0} I(t)/t =0$. In this part, we carefully use the Ikeda-Watanabe theorem. By Proposition~\ref{prop:heat_asymp} and Lemma~\ref{lem:IW-check}, when $B^{(\alpha)}$ is the subordinated hypoelliptic Brownian motion on $\mathbb{G}$, $\Pi$ is given by 
	\begin{align*}
		\Pi(x,dy)=\Pi(x,y)dy=\frac{dy}{\|x^{-1}y\|^{Q+\alpha}_\alpha}.
	\end{align*}
	Using Theorem~\ref{thm:Ikeda-Watanabe} we get 
	\begin{align*}
		I(t)=\int_\Omega\int_0^t \int_\Omega\int_{\overline{\Omega}^c} \mathbb{P}_z\left(B^{(\alpha)}_{t-s}\in \Omega\right)\frac{p^{(\alpha)}_\Omega(s,x,y)}{\|y^{-1} z\|^{Q+\alpha}_\alpha} dz dy ds dx,
	\end{align*}
	where $p^{(\alpha)}_\Omega(s,x,y)$ is the transition density of the subordinated Brownian motion killed upon exiting $\Omega$. We note that 
	\begin{align*}
		\int_\Omega\int_{\overline{\Omega}^c} \frac{\mathbb{P}_z\left(B^{(\alpha)}_{t-s}\in \Omega\right)}{\|y^{-1}z\|^{Q+\alpha}_Q} p^{(\alpha)}_\Omega(s,x,y) dz dy&= \int_\Omega\left(\int_{\overline{\Omega}^c} \frac{\mathbb{P}_{z}(B^{(\alpha)}_{t-s}\in \Omega)}{\|y^{-1}z\|^{Q+\alpha}_\alpha} dz\right) p^{(\alpha)}_\Omega(s,x,y)dy \\
		& \le P^{(\alpha)}_s(g_{t-s})(x),
	\end{align*}
	where \begin{align*}g_t(y)=\mathbbm{1}_\Omega(y)\int_{\overline{\Omega}^c} \frac{\mathbb{P}_{z}(B^{(\alpha)}_t\in \Omega)}{\|y^{-1}z\|^{Q+\alpha}_\alpha} dz,\end{align*} 
	and $P^{\Omega}$ is the semigroup associated with the killed subordinated process. Therefore, we have 
	\begin{align*}
		I(t)\le  \int_0^t \int_\Omega P^{\Omega}_s (g_{t-s})(x) dx ds.
	\end{align*}
	Also note that for all $t>0$,
	\begin{align*}
		\int_\Omega g_t(y) dy\le \int_{\Omega}\int_{\Omega^c}\frac{1}{\|y^{-1} z\|^{Q+\alpha}_\alpha} dydz =\Per^{(\alpha)}_{\mathcal H}(\Omega)<\infty.
	\end{align*}
	Hence,
	\begin{align*}
		I(t)\le \int_0^t \int_{\Omega} P^{\Omega}_s(\mathbbm{1}_\Omega)(x) g_{t-s}(x)dx ds \le \int_0^t \int_{\Omega}g_s(x) dx ds,
	\end{align*}
	which shows that 
	\begin{align*}
		\limsup_{t\to 0}\frac{I(t)}{t}\le \limsup_{t\to 0} \int_\Omega g_t(x) dx \le \int_\Omega\int_{\overline{\Omega}^c}\frac{\limsup_{t\to 0}\mathbb{P}_z(B^{(\alpha)}_t\in\Omega)}{\|y^{-1}z\|^{Q+\alpha}_\alpha} dz dy =0,
	\end{align*}
	where the last inequality follows from the dominated convergence theorem along with the fact $\lim_{t\to 0} \mathbb{P}_z(B^{(\alpha)}_t\in\Omega)=0$ for any $z\in\overline{\Omega}^c$ as $B^{(\alpha)}$ is right continuous.
	
	On the other hand, due to Proposition~\ref{prop:skip_boundary}, $II(t)=0$ for all $t>0$. Invoking \eqref{eq:H_decomp}, we conclude the proof of \eqref{eq:lim_0_1}. 

    When $\Omega$ is bounded with $C^1$ boundary having no characteristic points, by Theorem~\ref{thm:noncharacteristic-cones}, $\Omega$ satisfies the intrinsic exterior cone condition defined in Definition~\ref{def:cone_cond}. By Proposition~\ref{prop:VDC_conseq} it follows that $\Omega$ also satisfies the volume density condition, and hence \eqref{eq:lim_0_1} holds in this case. This completes the proof of the theorem. 
\end{proof}

\subsection*{Acknowledgment} The authors thank Sasha Teplyaev for many insightful discussions during the preparation of this work.

\appendix
\section{Dirichlet form associated with subordinated semigroup}
Let $(\mathcal{E}^\phi,\mathcal{D}(\mathcal{E}^\phi))$ be the Dirichlet form defined in \eqref{eq:subordinated_DF} which corresponds to the subordinated semigroup $(P^\phi_t)_{t\ge 0}$ defined in \eqref{eq:subordinated_semigroup}. Let us also recall that 
\begin{align*}
    -\Delta^\phi_{\mathcal H}=-\phi(\Delta_{\mathcal H})
\end{align*}
is the generator of $(P^\phi_t)_{t\ge 0}$ in $L^2(\mathbb G, dx)$.
\begin{lemma}\label{lem:form_core_1} For any Bernstein function $\phi$, 
\begin{align}\label{eq:chain_inclusion}
    C^\infty_c(\mathbb G)\subset \mathcal{D}(\Delta^\phi_{\mathcal H})\subset \mathcal{D}(\mathcal{E}^\phi).
\end{align}
Moreover, $C^\infty_c(\mathbb G)$ is a form core for $\mathcal{E}^\phi$.
\end{lemma}
\begin{proof}
    By Phillips' subordination theorem, see for instance \cite[Theorem~13.6]{SchillingSongVondracekBook}, for any Bernstein function $\phi$, $\mathcal{D}(\Delta_{\mathcal H})$ is a core for $(\Delta^\phi_{\mathcal H}, \mathcal{D}(\Delta^\phi_{\mathcal H}))$. As noted before,  $(\Delta_{\mathcal H}, C^\infty_c(\mathbb G))$ is essentially self-adjoint in $L^2(\mathbb G, dx)$, that is, $C^\infty_c(\mathbb G)$ is a core for $(\Delta_{\mathcal H},\mathcal{D}(\Delta_{\mathcal H}))$, which implies that 
    \begin{align*}
    C^\infty_c(\mathbb G)\subset \mathcal{D}(\Delta^\phi_{\mathcal H}).
\end{align*}
Also, by \cite[Corollary~1.3.1]{Fukushima_et_alBook}, $\mathcal{D}(\Delta^\phi_{\mathcal H})\subset \mathcal{D}(\mathcal{E}^\phi)$. This proves \eqref{eq:chain_inclusion}.

By \cite[p.~16]{SchillingSongVondracekBook}, for any Bernstein function $\phi$, there exists $C>0$ such that 
\begin{align}\label{eq:Bernstein_bound}
    \phi(\lambda)\le C(1+\lambda) \quad \mbox{for all $\lambda\ge 0$}.
\end{align}
Therefore, for any $f\in \mathcal{D}(\Delta_{\mathcal H})$,
\begin{align}\label{eq:norm_ineq}
    \|f\|_{\mathcal{E}^\phi}=(\|f\|^2_{L^2(\mathbb G,dx)}+\mathcal{E}^\phi(f,f))^{\frac12} \le C(\|f\|_{L^2(\mathbb G, dx)}+\|\Delta_{\mathcal H} f\|_{L^2(\mathbb G, dx)}),
\end{align}
for some constant $C>0$. We note that the semigroup
\begin{align*}
e^{-t(\Delta^\phi_{\mathcal H})^{1/2}}
\end{align*}
is also obtained by subordinating the semigroup $P^\phi_t$ with the Bernstein function $u\mapsto \sqrt{u}$, and therefore by Phillips' subordination theorem, $\mathcal{D}(\Delta^\phi_{\mathcal H})$ is a core for $((\Delta^\phi_{\mathcal H})^{1/2}, \mathcal{D}((\Delta^\phi_{\mathcal H})^{1/2}))$. Since $C^\infty_c(\mathbb G)$ is a core for $(\Delta_{\mathcal H}, \mathcal{D}(\Delta_{\mathcal H}))$ in $L^2(\mathbb G, dx)$, the above inequality in \eqref{eq:norm_ineq} implies that $C^\infty_c(\mathbb G)$ is a form core for $\mathcal{E}^\phi$. This completes the proof.
\end{proof}
Let $\Omega\subset\mathbb G$ be a measurable subset. Consider the part of the Dirichlet form $((\mathcal{E}^\phi)^\Omega, (\mathcal{F}^\phi)^\Omega)$ with
\begin{align*}
     (\mathcal{F}^\phi)^\Omega=\{f\in\mathcal{D}(\mathcal{E}^\phi): \widetilde{f}=0 \quad \mbox{$\mathcal{E}^\phi$-q.e. on $\mathbb{G}\setminus \Omega$}\},
\end{align*}
where $\widetilde{f}$ denote the quasi-continuous version of $f$, see \cite[p.~69]{Fukushima_et_alBook}.
\begin{lemma}\label{lem:form_core_2}
    Let $\Omega$ be an open subset of $\mathbb G$. Then, $((\mathcal{E}^\phi)^\Omega, (\mathcal{F}^\phi)^\Omega)$ is a regular Dirichlet form on $L^2(\Omega,dx)$. Moreover, $C^\infty_c(\Omega)$ is a form core for $\mathcal{E}^{\phi,\Omega}$. As a result, the generator associated with $\mathcal{E}^{\phi,\Omega}$ is the Friedrichs extension of $(\mathcal{L}^\phi_\Omega, C^\infty_c(\Omega))$ in $L^2(\Omega, dx)$, where for any $f\in C^\infty_c(\Omega)$,
    \begin{align*}
        \mathcal{L}^\phi_\Omega f = \left.\Delta^\phi_\Omega f^0\right|_{\Omega},
    \end{align*}
    with 
    \begin{align*}
        f^0(x)=\begin{cases}
            f(x) & \mbox{if $x\in\Omega$} \\ 
            0 & \mbox{if $x\in\Omega^c$}.
        \end{cases}
    \end{align*}
\end{lemma}
\begin{proof}
    We first note that $\mathcal{E}^\phi$ is a regular Dirichlet form on $L^2(\mathbb G, dx)$ as $C^\infty_c(\mathbb G)\subset \mathcal{D}(\mathcal{E}^\phi)$, due to Lemma~\ref{lem:form_core_1}. Since $\Omega$ is open, by \cite[Theorem~4.4.3(i)]{Fukushima_et_alBook}, $(\mathcal{E}^{\phi,\Omega}, (\mathcal{F}^\phi)^\Omega)$ is a regular Dirichlet form on $L^2(\mathbb G, dx)$. Also by Lemma~\ref{lem:form_core_1}, $C^\infty_c(\mathbb G)$ is a core for $\mathcal{E}^\phi$. Moreover, it is a \emph{special standard core} for $\mathcal{E}^\phi$, see \cite[p.~6]{Fukushima_et_alBook}. Therefore, invoking \cite[Theorem~4.4.3(i)]{Fukushima_et_alBook} once again, the space of functions
    \begin{align*}
        \mathcal{C}_\Omega:=\{f\in C^\infty_c(\mathbb G): \supp(f)\subset\Omega\}
    \end{align*}
    is a form core for $\mathcal{E}^{\phi,\Omega}$. Since $\mathcal{C}_\Omega=C^\infty_c(\Omega)$, we conclude that the generator of $\mathcal{E}^{\phi,\Omega}$ in $L^2(\Omega, dx)$ is the Friedrichs extension of $(\mathcal{L}^\phi_\Omega, C^\infty_c(\Omega))$ in $L^2(\Omega, dx)$. This completes the proof.
\end{proof}

\bibliographystyle{plain}
\bibliography{SubordinatedCarnot}

@article {MR2836591,
    AUTHOR = {Franchi, Bruno and Serapioni, Raul and Serra Cassano,
              Francesco},
     TITLE = {Differentiability of intrinsic {L}ipschitz functions within
              {H}eisenberg groups},
   JOURNAL = {J. Geom. Anal.},
  FJOURNAL = {Journal of Geometric Analysis},
    VOLUME = {21},
      YEAR = {2011},
    NUMBER = {4},
     PAGES = {1044--1084},
      ISSN = {1050-6926,1559-002X},
   MRCLASS = {22E30 (58C20)},
  MRNUMBER = {2836591},
MRREVIEWER = {Davide\ Vittone},
       DOI = {10.1007/s12220-010-9178-4},
       URL = {https://doi.org/10.1007/s12220-010-9178-4},
}

@incollection {MR3587666,
    AUTHOR = {Serra Cassano, Francesco},
     TITLE = {Some topics of geometric measure theory in {C}arnot groups},
 BOOKTITLE = {Geometry, analysis and dynamics on sub-{R}iemannian manifolds.
              {V}ol. 1},
    SERIES = {EMS Ser. Lect. Math.},
     PAGES = {1--121},
 PUBLISHER = {Eur. Math. Soc., Z\"urich},
      YEAR = {2016},
      ISBN = {978-3-03719-162-0},
   MRCLASS = {53C17 (28A75 49Q20)},
  MRNUMBER = {3587666},
MRREVIEWER = {Davide\ Vittone},
}

@Book{LiaoMingBookLevyProcessesinLieGroups,
  Author                   = {Liao, Ming},
  Title                    = {L\'evy processes in {L}ie groups},
  Year                     = {2004},
  Volume                   = {162},
  Pages                    = {x+266},
  Address                  = {Cambridge},
  ISBN                     = {0-521-83653-0},
  Mrclass                  = {60-02 (22E30 43A80 60G51)},
  Mrnumber                 = {MR2060091 (2005e:60004)},
  Mrreviewer               = {H. Heyer},
  Publisher                = {Cambridge University Press},
  Series                   = {Cambridge Tracts in Mathematics}
}

@Article{AlbeverioGordina2007,
  author       = {Albeverio, Sergio and Gordina, Maria},
  journal      = {Bull. Sci. Math.},
  title        = {L\'evy processes and their subordination in matrix {L}ie groups},
  year         = {2007},
  issn         = {0007-4497},
  number       = {8},
  pages        = {738--760},
  volume       = {131},
  coden        = {BSMQA9},
  creationdate = {2025-03-03T19:00:36},
  fjournal     = {Bulletin des Sciences Math\'ematiques},
  mrclass      = {60B15 (43A05 43A80 58D20 60G51)},
  mrnumber     = {MR2372464 (2008m:60010)},
  mrreviewer   = {David Applebaum},
}

@Article{ChenSong2005,
  author       = {Chen, Zhen-Qing and Song, Renming},
  journal      = {J. Funct. Anal.},
  title        = {Two-sided eigenvalue estimates for subordinate processes in domains},
  year         = {2005},
  issn         = {0022-1236},
  number       = {1},
  pages        = {90--113},
  volume       = {226},
  creationdate = {2025-03-03T19:15:15},
  doi          = {10.1016/j.jfa.2005.05.004},
  fjournal     = {Journal of Functional Analysis},
  mrclass      = {60J25 (35P15 60J35)},
  mrnumber     = {2158176},
  mrreviewer   = {Wilhelm Stannat},
  url          = {https://doi.org/10.1016/j.jfa.2005.05.004},
}

@Article{TysonWangJ2018,
  author       = {Tyson, Jeremy and Wang, Jing},
  journal      = {Comm. Partial Differential Equations},
  title        = {Heat content and horizontal mean curvature on the {H}eisenberg group},
  year         = {2018},
  issn         = {0360-5302},
  number       = {3},
  pages        = {467--505},
  volume       = {43},
  creationdate = {2025-03-03T19:23:40},
  doi          = {10.1080/03605302.2018.1446166},
  fjournal     = {Communications in Partial Differential Equations},
  mrclass      = {53C17 (35K05 35R03 58J65)},
  mrnumber     = {3804205},
  mrreviewer   = {Andrea Pinamonti},
  url          = {https://doi.org/10.1080/03605302.2018.1446166},
}

@Article{RizziRossi2021,
  author       = {Rizzi, Luca and Rossi, Tommaso},
  journal      = {J. Math. Pures Appl. (9)},
  title        = {Heat content asymptotics for sub-{R}iemannian manifolds},
  year         = {2021},
  issn         = {0021-7824},
  pages        = {267--307},
  volume       = {148},
  creationdate = {2025-03-03T19:24:25},
  doi          = {10.1016/j.matpur.2020.12.004},
  fjournal     = {Journal de Math\'{e}matiques Pures et Appliqu\'{e}es. Neuvi\`eme S\'{e}rie},
  mrclass      = {53C17 (35R01 58J35 58J60)},
  mrnumber     = {4223354},
  mrreviewer   = {Jing Wang},
  url          = {https://doi.org/10.1016/j.matpur.2020.12.004},
}

@Article{ParkSong2022,
  author       = {Park, Hyunchul and Song, Renming},
  journal      = {Electron. J. Probab.},
  title        = {Spectral heat content for {$\alpha$}-stable processes in {$C^{1,1}$} open sets},
  year         = {2022},
  pages        = {Paper No. 22, 19},
  volume       = {27},
  creationdate = {2025-03-04T16:17:05},
  doi          = {10.1214/22-ejp752},
  fjournal     = {Electronic Journal of Probability},
  mrclass      = {60G52 (60J76)},
  mrnumber     = {4379201},
  url          = {https://doi.org/10.1214/22-ejp752},
}

@Article{BaudoinBonnefont2016,
  author       = {Baudoin, Fabrice and Bonnefont, Michel},
  journal      = {Nonlinear Anal.},
  title        = {Reverse {P}oincar\'e inequalities, isoperimetry, and {R}iesz transforms in {C}arnot groups},
  year         = {2016},
  issn         = {0362-546X},
  pages        = {48--59},
  volume       = {131},
  creationdate = {2024-01-21T07:49:13},
  fjournal     = {Nonlinear Analysis. Theory, Methods \& Applications. An International Multidisciplinary Journal},
  mrclass      = {43A77 (47D07)},
  mrnumber     = {3427969},
  mrreviewer   = {Alexander Isaakovich Shtern},
  url          = {https://doi.org/10.1016/j.na.2015.10.014},
}

@article {FerrariFranchi2015,
    AUTHOR = {Ferrari, Fausto and Franchi, Bruno},
     TITLE = {Harnack inequality for fractional sub-{L}aplacians in {C}arnot
              groups},
   JOURNAL = {Math. Z.},
  FJOURNAL = {Mathematische Zeitschrift},
    VOLUME = {279},
      YEAR = {2015},
    NUMBER = {1-2},
     PAGES = {435--458},
      ISSN = {0025-5874,1432-1823},
   MRCLASS = {35R03 (35B45 35B65 35K08 35R11 53C17)},
  MRNUMBER = {3299862},
MRREVIEWER = {Maochun\ Zhu},
       DOI = {10.1007/s00209-014-1376-5},
       URL = {https://doi.org/10.1007/s00209-014-1376-5},
}

@article {FFMPPS2028,
    AUTHOR = {Ferrari, Fausto and Miranda, Jr., Michele and Pallara, Diego
              and Pinamonti, Andrea and Sire, Yannick},
     TITLE = {Fractional {L}aplacians, perimeters and heat semigroups in
              {C}arnot groups},
   JOURNAL = {Discrete Contin. Dyn. Syst. Ser. S},
  FJOURNAL = {Discrete and Continuous Dynamical Systems. Series S},
    VOLUME = {11},
      YEAR = {2018},
    NUMBER = {3},
     PAGES = {477--491},
      ISSN = {1937-1632,1937-1179},
   MRCLASS = {35R11 (35K08 35R03)},
  MRNUMBER = {3732178},
MRREVIEWER = {Serena\ Dipierro},
       DOI = {10.3934/dcdss.2018026},
       URL = {https://doi.org/10.3934/dcdss.2018026},
}

@book {Sato_Book,
    AUTHOR = {Sato, Ken-iti},
     TITLE = {L\'evy processes and infinitely divisible distributions},
    SERIES = {Cambridge Studies in Advanced Mathematics},
    VOLUME = {68},
   EDITION = {Revised},
      NOTE = {Translated from the 1990 Japanese original},
 PUBLISHER = {Cambridge University Press, Cambridge},
      YEAR = {2013},
     PAGES = {xiv+521},
      ISBN = {978-1-107-65649-9},
   MRCLASS = {60G51 (60E07 60G18 60G52 60J45)},
  MRNUMBER = {3185174},
}

@article {ParkSong2019,
    AUTHOR = {Park, Hyunchul and Song, Renming},
     TITLE = {Small time asymptotics of spectral heat contents for
              subordinate killed {B}rownian motions related to isotropic
              {$\alpha$}-stable processes},
   JOURNAL = {Bull. Lond. Math. Soc.},
  FJOURNAL = {Bulletin of the London Mathematical Society},
    VOLUME = {51},
      YEAR = {2019},
    NUMBER = {2},
     PAGES = {371--384},
      ISSN = {0024-6093,1469-2120},
   MRCLASS = {60J75},
  MRNUMBER = {3937594},
MRREVIEWER = {Wojciech\ Cygan},
       DOI = {10.1112/blms.12235},
       URL = {https://doi.org/10.1112/blms.12235},
}

@Book{BlumenthalGetoor1968,
  author       = {Blumenthal, R. M. and Getoor, R. K.},
  publisher    = {Academic Press, New York-London},
  title        = {Markov processes and potential theory},
  year         = {1968},
  series       = {Pure and Applied Mathematics, Vol. 29},
  creationdate = {2025-07-06T05:17:15},
  mrclass      = {60.62 (60.60)},
  mrnumber     = {0264757},
  mrreviewer   = {J. L. Doob},
  pages        = {x+313},
}

@Article{CarfagniniGordina2024,
  author       = {Carfagnini, Marco and Gordina, Maria},
  journal      = {Int. Math. Res. Not. IMRN},
  title        = {Dirichlet {S}ub-{L}aplacians on {H}omogeneous {C}arnot {G}roups: {S}pectral {P}roperties, {A}symptotics, and {H}eat {C}ontent},
  year         = {2024},
  issn         = {1073-7928},
  number       = {3},
  pages        = {1894--1930},
  volume       = {3},
  creationdate = {2025-07-06T05:17:58},
  doi          = {10.1093/imrn/rnad065},
  fjournal     = {International Mathematics Research Notices. IMRN},
  mrclass      = {Prelim},
  mrnumber     = {4702267},
  url          = {https://doi.org/10.1093/imrn/rnad065},
}

@Book{BonfiglioliLanconelliUguzzoniBook,
  author       = {Bonfiglioli, Andrea and Lanconelli, Ermanno and Uguzzoni, Francesco},
  publisher    = {Springer, Berlin},
  title        = {Stratified {L}ie groups and potential theory for their sub-{L}aplacians},
  year         = {2007},
  isbn         = {978-3-540-71896-3; 3-540-71896-6},
  series       = {Springer Monographs in Mathematics},
  creationdate = {2025-11-05T17:44:12},
  mrclass      = {22E30 (31C45 35-02 35H10 43A80)},
  mrnumber     = {2363343},
  mrreviewer   = {Maria Stella Fanciullo},
  pages        = {xxvi+800},
}

@Misc{BossioRizziRossi2024Arxiv,
  author        = {Tania Bossio and Luca Rizzi and Tommaso Rossi},
  title         = {Tubes in sub-Riemannian geometry and a Weyl's invariance result for curves in the Heisenberg groups},
  year          = {2024},
  archiveprefix = {arXiv},
  creationdate  = {2025-12-09T14:40:35},
  eprint        = {2408.16838},
  primaryclass  = {math.DG},
  url           = {https://arxiv.org/abs/2408.16838},
}

@Article{AgrachevRizziRossi2024,
  author       = {Agrachev, Andrei and Rizzi, Luca and Rossi, Tommaso},
  journal      = {Anal. PDE},
  title        = {Relative heat content asymptotics for sub-{R}iemannian manifolds},
  year         = {2024},
  issn         = {2157-5045,1948-206X},
  number       = {9},
  pages        = {2997--3037},
  volume       = {17},
  creationdate = {2025-12-09T15:01:39},
  doi          = {10.2140/apde.2024.17.2997},
  fjournal     = {Analysis \& PDE},
  mrclass      = {53C17 (35K15 35R01 58J35 58J60)},
  mrnumber     = {4818197},
  mrreviewer   = {Robert\ Weston\ Neel},
  url          = {https://doi.org/10.2140/apde.2024.17.2997},
}

@Article{Hunt1956a,
  author       = {Hunt, G. A.},
  journal      = {Trans. Amer. Math. Soc.},
  title        = {Semi-groups of measures on {L}ie groups},
  year         = {1956},
  issn         = {0002-9947},
  pages        = {264--293},
  volume       = {81},
  creationdate = {2025-12-09T16:02:09},
  fjournal     = {Transactions of the American Mathematical Society},
  mrclass      = {46.2X},
  mrnumber     = {MR0079232 (18,54a)},
  mrreviewer   = {K. Yosida},
}

@Misc{CarfagniniGordinaTeplyaev2024,
  author        = {Marco Carfagnini and Maria Gordina and Alexander Teplyaev},
  title         = {Dirichlet metric measure spaces: spectrum, irreducibility, and small deviations},
  year          = {2024},
  archiveprefix = {arXiv},
  creationdate  = {2025-12-09T16:08:31},
  eprint        = {2409.07425},
  primaryclass  = {math.PR},
  url           = {https://arxiv.org/abs/2409.07425},
}

@Book{CapognaDanielliPaulsTysonBook,
  author       = {Capogna, Luca and Danielli, Donatella and Pauls, Scott D. and Tyson, Jeremy T.},
  publisher    = {Birkh\"auser Verlag, Basel},
  title        = {An introduction to the {H}eisenberg group and the sub-{R}iemannian isoperimetric problem},
  year         = {2007},
  isbn         = {978-3-7643-8132-5; 3-7643-8132-9},
  series       = {Progress in Mathematics},
  volume       = {259},
  creationdate = {2025-12-09T16:37:00},
  mrclass      = {53C17 (22E30 30C65 32T27 32V15 49Q15)},
  mrnumber     = {2312336 (2009a:53053)},
  mrreviewer   = {Piotr Haj{\l}asz},
  pages        = {xvi+223},
}

@article {GrzywnyParkSong2019,
    AUTHOR = {Grzywny, Tomasz and Park, Hyunchul and Song, Renming},
     TITLE = {Spectral heat content for {L}\'evy processes},
   JOURNAL = {Math. Nachr.},
  FJOURNAL = {Mathematische Nachrichten},
    VOLUME = {292},
      YEAR = {2019},
    NUMBER = {4},
     PAGES = {805--825},
      ISSN = {0025-584X,1522-2616},
   MRCLASS = {60J75 (35K05 35K08)},
  MRNUMBER = {3937619},
MRREVIEWER = {Vassili\ N.\ Kolokol\cprime tsov},
       DOI = {10.1002/mana.201800035},
       URL = {https://doi.org/10.1002/mana.201800035},
}

@Article{Hoermander1967a,
  author       = {H{\"o}rmander, Lars},
  journal      = {Acta Math.},
  title        = {Hypoelliptic second order differential equations},
  year         = {1967},
  issn         = {0001-5962},
  pages        = {147--171},
  volume       = {119},
  creationdate = {2026-03-17T10:02:57},
  doi          = {10.1007/bf02392081},
  fjournal     = {Acta Mathematica},
  mrclass      = {35.48 (47.00)},
  mrnumber     = {0222474 (36 \#5526)},
  mrreviewer   = {J. Smoller},
}

@Article{GordinaLaetsch2017,
  author       = {Gordina, Maria and Laetsch, Thomas},
  journal      = {Trans. Amer. Math. Soc.},
  title        = {A convergence to {B}rownian motion on sub-{R}iemannian manifolds},
  year         = {2017},
  issn         = {0002-9947},
  number       = {9},
  pages        = {6263--6278},
  volume       = {369},
  creationdate = {2026-03-17T12:01:58},
  doi          = {10.1090/tran/6831},
  fjournal     = {Transactions of the American Mathematical Society},
  mrclass      = {60J65 (58J65)},
  mrnumber     = {3660220},
  url          = {http://dx.doi.org/10.1090/tran/6831},
}

@Article{Bogdan1997,
  author       = {Bogdan, Krzysztof},
  journal      = {Studia Math.},
  title        = {The boundary {H}arnack principle for the fractional {L}aplacian},
  year         = {1997},
  issn         = {0039-3223,1730-6337},
  number       = {1},
  pages        = {43--80},
  volume       = {123},
  creationdate = {2026-03-17T12:22:42},
  doi          = {10.4064/sm-123-1-43-80},
  fjournal     = {Studia Mathematica},
  mrclass      = {31B25 (31C05 35S99 60J65)},
  mrnumber     = {1438304},
  mrreviewer   = {Zhong\ Xin\ Zhao},
  url          = {https://doi.org/10.4064/sm-123-1-43-80},
}

@article{Sarkar2026,
      title={Fractional heat content asymptotics for Carnot groups}, 
      author={Rohan Sarkar},
      year={2026},
      eprint={2601.04088},
      archivePrefix={arXiv},
      primaryClass={math.AP},
      url={https://arxiv.org/abs/2601.04088}, 
      journal={arXiv:2601.04088}
}

@article {Ikeda-Watanabe,
    AUTHOR = {Ikeda, Nobuyuki and Watanabe, Shinzo},
     TITLE = {On some relations between the harmonic measure and the
              {L}\'evy measure for a certain class of {M}arkov processes},
   JOURNAL = {J. Math. Kyoto Univ.},
  FJOURNAL = {Journal of Mathematics of Kyoto University},
    VOLUME = {2},
      YEAR = {1962},
     PAGES = {79--95},
      ISSN = {0023-608X},
   MRCLASS = {60.60 (60.62)},
  MRNUMBER = {142153},
MRREVIEWER = {J.\ Elliott},
       DOI = {10.1215/kjm/1250524975},
       URL = {https://doi.org/10.1215/kjm/1250524975},
}

@Book{EngelNagelBook2006,
  author     = {Engel, Klaus-Jochen and Nagel, Rainer},
  publisher  = {Springer, New York},
  title      = {A short course on operator semigroups},
  year       = {2006},
  isbn       = {978-0387-31341-2; 0-387-31341-9},
  series     = {Universitext},
  mrclass    = {47-01 (47D03 47D06)},
  mrnumber   = {2229872},
  mrreviewer = {Jacek Banasiak},
  pages      = {x+247},
}

@book{VaropoulosSaloff-CosteCoulhonBook1992,
    AUTHOR = {Varopoulos, N. Th. and Saloff-Coste, L. and Coulhon, T.},
     TITLE = {Analysis and geometry on groups},
    SERIES = {Cambridge Tracts in Mathematics},
    VOLUME = {100},
 PUBLISHER = {Cambridge University Press, Cambridge},
      YEAR = {1992},
     PAGES = {xii+156},
      ISBN = {0-521-35382-3},
   MRCLASS = {43A80 (47D03 47F05 58G11 60B15)},
  MRNUMBER = {1218884},
MRREVIEWER = {A.\ Hulanicki},
}

@Book{AgrachevBarilariBoscainBook2020,
  author       = {Agrachev, Andrei and Barilari, Davide and Boscain, Ugo},
  publisher    = {Cambridge University Press, Cambridge},
  title        = {A comprehensive introduction to sub-{R}iemannian geometry},
  year         = {2020},
  isbn         = {978-1-108-47635-5},
  note         = {From the Hamiltonian viewpoint, With an appendix by Igor Zelenko},
  series       = {Cambridge Studies in Advanced Mathematics},
  volume       = {181},
  creationdate = {2026-06-24T11:05:38},
  mrclass      = {53C17},
  mrnumber     = {3971262},
  mrreviewer   = {Luca Rizzi},
  pages        = {xviii+745},
}

@Article{ChenChenLi2026,
  author       = {Chen, Hua and Chen, Hong-Ge and Li, Jin-Ning},
  journal      = {J. Differential Equations},
  title        = {Weyl's law and estimates of eigenvalues for fractional sub-{L}aplacian operators on {C}arnot groups},
  year         = {2026},
  issn         = {0022-0396,1090-2732},
  number       = {part 2},
  pages        = {Paper No. 113861, 46},
  volume       = {453},
  creationdate = {2026-06-24T11:09:07},
  doi          = {10.1016/j.jde.2025.113861},
  fjournal     = {Journal of Differential Equations},
  mrclass      = {35P15 (35J70 35R03 58C40)},
  mrnumber     = {4975440},
  url          = {https://doi.org/10.1016/j.jde.2025.113861},
}

@Article{FranchiSerapioni2016,
  author       = {Franchi, Bruno and Serapioni, Raul Paolo},
  journal      = {J. Geom. Anal.},
  title        = {Intrinsic {L}ipschitz graphs within {C}arnot groups},
  year         = {2016},
  issn         = {1050-6926},
  number       = {3},
  pages        = {1946--1994},
  volume       = {26},
  creationdate = {2026-06-24T11:11:10},
  doi          = {10.1007/s12220-015-9615-5},
  fjournal     = {Journal of Geometric Analysis},
  mrclass      = {49Q15 (22E25 53C17 58C20 58C35)},
  mrnumber     = {3511465},
  mrreviewer   = {Jingzhi Tie},
  url          = {https://doi.org/10.1007/s12220-015-9615-5},
}

@article{Strichartz1986,
  author  = {Strichartz, Robert S.},
  title   = {Sub-Riemannian geometry},
  journal = {Journal of Differential Geometry},
  volume  = {24},
  number  = {2},
  pages   = {221--263},
  year    = {1986},
  doi     = {10.4310/jdg/1214440436}
}

@Article{FranchiSerapioniSerra_Cassano2003a,
  author       = {Franchi, Bruno and Serapioni, Raul and Serra Cassano, Francesco},
  journal      = {J. Geom. Anal.},
  title        = {On the structure of finite perimeter sets in step 2 {C}arnot groups},
  year         = {2003},
  issn         = {1050-6926},
  number       = {3},
  pages        = {421--466},
  volume       = {13},
  creationdate = {2026-06-24T11:11:48},
  doi          = {10.1007/BF02922053},
  fjournal     = {The Journal of Geometric Analysis},
  mrclass      = {49Q15 (53C17)},
  mrnumber     = {1984849},
  mrreviewer   = {J. E. Brothers},
  url          = {https://doi.org/10.1007/BF02922053},
}

@Article{Getoor1961b,
  author       = {Getoor, R. K.},
  journal      = {Trans. Amer. Math. Soc.},
  title        = {First passage times for symmetric stable processes in space},
  year         = {1961},
  issn         = {0002-9947,1088-6850},
  pages        = {75--90},
  volume       = {101},
  creationdate = {2026-06-24T11:15:39},
  doi          = {10.2307/1993412},
  fjournal     = {Transactions of the American Mathematical Society},
  mrclass      = {60.60},
  mrnumber     = {137148},
  mrreviewer   = {J.\ Elliott},
  url          = {https://doi.org/10.2307/1993412},
}

@Book{HeinonenBook2001,
  author       = {Heinonen, Juha},
  publisher    = {Springer-Verlag, New York},
  title        = {Lectures on analysis on metric spaces},
  year         = {2001},
  isbn         = {0-387-95104-0},
  series       = {Universitext},
  creationdate = {2026-06-24T11:17:11},
  doi          = {10.1007/978-1-4613-0131-8},
  mrclass      = {30C65 (28A75 28A78 46E35)},
  mrnumber     = {1800917},
  mrreviewer   = {Christopher Bishop},
  pages        = {x+140},
  url          = {http://dx.doi.org/10.1007/978-1-4613-0131-8},
}

@Book{SchillingSongVondracekBook,
  author       = {Schilling, Ren\'{e} L. and Song, Renming and Vondraček, Zoran},
  publisher    = {Walter de Gruyter \& Co., Berlin},
  title        = {Bernstein functions},
  year         = {2012},
  edition      = {Second},
  isbn         = {978-3-11-025229-3; 978-3-11-026933-8},
  note         = {Theory and applications},
  series       = {De Gruyter Studies in Mathematics},
  volume       = {37},
  creationdate = {2026-06-24T11:35:39},
  doi          = {10.1515/9783110269338},
  mrclass      = {60E07 (31C05 43A35 44A10 47A57 47D06 60E10 60Jxx)},
  mrnumber     = {2978140},
  mrreviewer   = {David\ Applebaum},
  pages        = {xiv+410},
  url          = {https://doi.org/10.1515/9783110269338},
}

@Article{LeDonne2017,
  author       = {Le Donne, Enrico},
  journal      = {Anal. Geom. Metr. Spaces},
  title        = {A primer on {C}arnot groups: homogenous groups, {C}arnot-{C}arath\'{e}odory spaces, and regularity of their isometries},
  year         = {2017},
  number       = {1},
  pages        = {116--137},
  volume       = {5},
  creationdate = {2026-06-24T11:55:19},
  doi          = {10.1515/agms-2017-0007},
  fjournal     = {Analysis and Geometry in Metric Spaces},
  mrclass      = {53C17 (22E25 22F30 43A80)},
  mrnumber     = {3742567},
  mrreviewer   = {Andrea Pinamonti},
  url          = {https://doi.org/10.1515/agms-2017-0007},
}

@article {GarofaloTralli2023,
    AUTHOR = {Garofalo, Nicola and Tralli, Giulio},
     TITLE = {A {B}ourgain-{B}rezis-{M}ironescu-{D}\'avila theorem in
              {C}arnot groups of step two},
   JOURNAL = {Comm. Anal. Geom.},
  FJOURNAL = {Communications in Analysis and Geometry},
    VOLUME = {31},
      YEAR = {2023},
    NUMBER = {2},
     PAGES = {321--341},
      ISSN = {1019-8385,1944-9992},
   MRCLASS = {53C17 (46E35 49Q15)},
  MRNUMBER = {4685023},
MRREVIEWER = {Andrea\ Pinamonti},
       DOI = {10.4310/cag.2023.v31.n2.a3},
       URL = {https://doi.org/10.4310/cag.2023.v31.n2.a3},
}

@article {AlbanoCannarsaScarinci2018,
	AUTHOR = {Albano, Paolo and Cannarsa, Piermarco and Scarinci, Teresa},
	TITLE = {Regularity results for the minimum time function with
	{H}\"ormander vector fields},
	JOURNAL = {J. Differential Equations},
	FJOURNAL = {Journal of Differential Equations},
	VOLUME = {264},
	YEAR = {2018},
	NUMBER = {5},
	PAGES = {3312--3335},
	ISSN = {0022-0396,1090-2732},
	MRCLASS = {35F30 (35D40 35F21)},
	MRNUMBER = {3741391},
	DOI = {10.1016/j.jde.2017.11.016},
	URL = {https://doi.org/10.1016/j.jde.2017.11.016},
}

@article {Phillips1952,
    AUTHOR = {Phillips, R. S.},
     TITLE = {On the generation of semigroups of linear operators},
   JOURNAL = {Pacific J. Math.},
  FJOURNAL = {Pacific Journal of Mathematics},
    VOLUME = {2},
      YEAR = {1952},
     PAGES = {343--369},
      ISSN = {0030-8730,1945-5844},
   MRCLASS = {46.3X},
  MRNUMBER = {50797},
MRREVIEWER = {K.\ Yosida},
       URL = {http://projecteuclid.org/euclid.pjm/1103051781},
}

@book {Bogdan_et_al2009,
    AUTHOR = {Bogdan, Krzysztof and Byczkowski, Tomasz and Kulczycki,
              Tadeusz and Ryznar, Michal and Song, Renming and Vondraček,
              Zoran},
     TITLE = {Potential analysis of stable processes and its extensions},
    SERIES = {Lecture Notes in Mathematics},
    VOLUME = {1980},
    EDITOR = {Graczyk, Piotr and Stos, Andrzej},
 PUBLISHER = {Springer-Verlag, Berlin},
      YEAR = {2009},
     PAGES = {x+187},
      ISBN = {978-3-642-02140-4},
   MRCLASS = {60J45 (31C05 31C25 35K08 35P15 60-02 60G52 60J50)},
  MRNUMBER = {2569321},
MRREVIEWER = {Wilhelm\ Stannat},
       DOI = {10.1007/978-3-642-02141-1},
       URL = {https://doi.org/10.1007/978-3-642-02141-1},
}

@book {EngelNagelBook2000,
    AUTHOR = {Engel, Klaus-Jochen and Nagel, Rainer},
     TITLE = {One-parameter semigroups for linear evolution equations},
    SERIES = {Graduate Texts in Mathematics},
    VOLUME = {194},
      NOTE = {With contributions by S. Brendle, M. Campiti, T. Hahn, G.
              Metafune, G. Nickel, D. Pallara, C. Perazzoli, A. Rhandi, S.
              Romanelli and R. Schnaubelt},
 PUBLISHER = {Springer-Verlag, New York},
      YEAR = {2000},
     PAGES = {xxii+586},
      ISBN = {0-387-98463-1},
   MRCLASS = {47D06 (34G10 35K90 47N20)},
  MRNUMBER = {1721989},
MRREVIEWER = {Charles\ Batty},
}

@article {DGS-C2009,
    AUTHOR = {Driver, Bruce K. and Gross, Leonard and Saloff-Coste, Laurent},
     TITLE = {Holomorphic functions and subelliptic heat kernels over {L}ie
              groups},
   JOURNAL = {J. Eur. Math. Soc. (JEMS)},
  FJOURNAL = {Journal of the European Mathematical Society (JEMS)},
    VOLUME = {11},
      YEAR = {2009},
    NUMBER = {5},
     PAGES = {941--978},
      ISSN = {1435-9855,1435-9863},
   MRCLASS = {32W30 (22E30 35H20)},
  MRNUMBER = {2538496},
MRREVIEWER = {Maria\ Gordina},
       DOI = {10.4171/JEMS/171},
       URL = {https://doi.org/10.4171/JEMS/171},
}

@article {SongVondracek2003,
    AUTHOR = {Song, Renming and Vondraček, Zoran},
     TITLE = {Potential theory of subordinate killed {B}rownian motion in a
              domain},
   JOURNAL = {Probab. Theory Related Fields},
  FJOURNAL = {Probability Theory and Related Fields},
    VOLUME = {125},
      YEAR = {2003},
    NUMBER = {4},
     PAGES = {578--592},
      ISSN = {0178-8051,1432-2064},
   MRCLASS = {60J45 (31C25 60J75)},
  MRNUMBER = {1974415},
MRREVIEWER = {Zhen-Qing\ Chen},
       DOI = {10.1007/s00440-002-0251-1},
       URL = {https://doi.org/10.1007/s00440-002-0251-1},
}

@book {BaktaiBook,
    AUTHOR = {B\'atkai, Andr\'as and Kramar Fijavž, Marjeta and Rhandi,
              Abdelaziz},
     TITLE = {Positive operator semigroups},
    SERIES = {Operator Theory: Advances and Applications},
    VOLUME = {257},
      NOTE = {From finite to infinite dimensions,
              With a foreword by Rainer Nagel and Ulf Schlotterbeck},
 PUBLISHER = {Birkh\"auser/Springer, Cham},
      YEAR = {2017},
     PAGES = {xvii+364},
      ISBN = {978-3-319-42811-6; 978-3-319-42813-0},
   MRCLASS = {47-02 (15-02 47B65 47D06)},
  MRNUMBER = {3616245},
MRREVIEWER = {Christoph\ Kriegler},
       DOI = {10.1007/978-3-319-42813-0},
       URL = {https://doi.org/10.1007/978-3-319-42813-0},
}

@article {KreinRutman1948,
    AUTHOR = {Kre\u in, M. G. and Rutman, M. A.},
     TITLE = {Linear operators leaving invariant a cone in a {B}anach space},
   JOURNAL = {Uspehi Matem. Nauk (N.S.)},
  FJOURNAL = {Uspehi Matem. Nauk (N.S.)},
    VOLUME = {3},
      YEAR = {1948},
    NUMBER = {1(23)},
     PAGES = {3--95},
   MRCLASS = {46.3X},
  MRNUMBER = {27128},
MRREVIEWER = {M.\ M.\ Day},
}

@article {CapognaDanielliGarofalo1994,
	AUTHOR = {Capogna, Luca and Danielli, Donatella and Garofalo, Nicola},
	TITLE = {The geometric {S}obolev embedding for vector fields and the
	isoperimetric inequality},
	JOURNAL = {Comm. Anal. Geom.},
	FJOURNAL = {Communications in Analysis and Geometry},
	VOLUME = {2},
	YEAR = {1994},
	NUMBER = {2},
	PAGES = {203--215},
	ISSN = {1019-8385,1944-9992},
	MRCLASS = {46E35 (26D20 35J99 58G03)},
	MRNUMBER = {1312686},
	MRREVIEWER = {James\ E.\ Ross},
	DOI = {10.4310/CAG.1994.v2.n2.a2},
	URL = {https://doi.org/10.4310/CAG.1994.v2.n2.a2},
}

@article{HebischSikora1990,
 author={Hebisch, Waldemar and Sikora, Adam},
  title={A smooth subadditive homogeneous norm on a homogeneous group},
  journal={Studia Mathematica},
  volume={96},
  pages={231--236},
  year={1990},
  publisher={Instytut Matematyczny Polskiej Akademii Nauk}
}

@Article {BergGilkey1994,
	AUTHOR = {van den Berg, M. and Gilkey, Peter B.},
	TITLE = {Heat content asymptotics of a {R}iemannian manifold with
	boundary},
	JOURNAL = {J. Funct. Anal.},
	FJOURNAL = {Journal of Functional Analysis},
	VOLUME = {120},
	YEAR = {1994},
	NUMBER = {1},
	PAGES = {48--71},
	ISSN = {0022-1236,1096-0783},
	MRCLASS = {58G11 (58G18 58G20)},
	MRNUMBER = {1262245},
	MRREVIEWER = {Patrice\ Sawyer},
	DOI = {10.1006/jfan.1994.1022},
	URL = {https://doi.org/10.1006/jfan.1994.1022},
}

@article {Cygan1981,
    AUTHOR = {Cygan, Jacek},
     TITLE = {Subadditivity of homogeneous norms on certain nilpotent {L}ie
              groups},
   JOURNAL = {Proc. Amer. Math. Soc.},
  FJOURNAL = {Proceedings of the American Mathematical Society},
    VOLUME = {83},
      YEAR = {1981},
    NUMBER = {1},
     PAGES = {69--70},
      ISSN = {0002-9939,1088-6826},
   MRCLASS = {22E30 (35H05 43A80)},
  MRNUMBER = {619983},
       DOI = {10.2307/2043893},
       URL = {https://doi.org/10.2307/2043893},
}

@article {Koranyi1985,
    AUTHOR = {Kor\'anyi, Adam},
     TITLE = {Geometric properties of {H}eisenberg-type groups},
   JOURNAL = {Adv. in Math.},
  FJOURNAL = {Advances in Mathematics},
    VOLUME = {56},
      YEAR = {1985},
    NUMBER = {1},
     PAGES = {28--38},
      ISSN = {0001-8708},
   MRCLASS = {53C22 (58F19)},
  MRNUMBER = {782541},
MRREVIEWER = {Edward\ N.\ Wilson},
       DOI = {10.1016/0001-8708(85)90083-0},
       URL = {https://doi.org/10.1016/0001-8708(85)90083-0},
}

@book {CorwinGreenleafBook,
    AUTHOR = {Corwin, Lawrence J. and Greenleaf, Frederick P.},
     TITLE = {Representations of nilpotent {L}ie groups and their
              applications. {P}art {I}},
    SERIES = {Cambridge Studies in Advanced Mathematics},
    VOLUME = {18},
      NOTE = {Basic theory and examples},
 PUBLISHER = {Cambridge University Press, Cambridge},
      YEAR = {1990},
     PAGES = {viii+269},
      ISBN = {0-521-36034-X},
   MRCLASS = {22E27 (22-01 22E25 22E30)},
  MRNUMBER = {1070979},
MRREVIEWER = {Jeffrey\ Fox},
}

@article {Savo1999,
	AUTHOR = {Savo, Alessandro},
	TITLE = {Uniform estimates and the whole asymptotic series of the heat
	content on manifolds},
	JOURNAL = {Geom. Dedicata},
	FJOURNAL = {Geometriae Dedicata},
	VOLUME = {73},
	YEAR = {1998},
	NUMBER = {2},
	PAGES = {181--214},
	ISSN = {0046-5755,1572-9168},
	MRCLASS = {58J37 (58J35)},
	MRNUMBER = {1652049},
	DOI = {10.1023/A:1005016122695},
	URL = {https://doi.org/10.1023/A:1005016122695},
}

@article {Jang-MeiWu2002,
    AUTHOR = {Wu, Jang-Mei},
     TITLE = {Harmonic measures for symmetric stable processes},
   JOURNAL = {Studia Math.},
  FJOURNAL = {Studia Mathematica},
    VOLUME = {149},
      YEAR = {2002},
    NUMBER = {3},
     PAGES = {281--293},
      ISSN = {0039-3223,1730-6337},
   MRCLASS = {60J45 (31B15 31C45)},
  MRNUMBER = {1893056},
MRREVIEWER = {Ren\ Ming\ Song},
       DOI = {10.4064/sm149-3-5},
       URL = {https://doi.org/10.4064/sm149-3-5},
}

@article {Sztonyk2000,
    AUTHOR = {Sztonyk, Pawe\l \},
     TITLE = {On harmonic measure for {L}\'evy processes},
   JOURNAL = {Probab. Math. Statist.},
  FJOURNAL = {Probability and Mathematical Statistics},
    VOLUME = {20},
      YEAR = {2000},
    NUMBER = {2},
     PAGES = {383--390},
      ISSN = {0208-4147,2300-8113},
   MRCLASS = {60J45},
  MRNUMBER = {1825650},
MRREVIEWER = {Mamoru\ Kanda},
}

@article {Bogdan_et_al2020,
    AUTHOR = {Bogdan, Krzysztof and Grzywny, Tomasz and Pietruska-Pa\l uba,
              Katarzyna and Rutkowski, Artur},
     TITLE = {Extension and trace for nonlocal operators},
   JOURNAL = {J. Math. Pures Appl. (9)},
  FJOURNAL = {Journal de Math\'ematiques Pures et Appliqu\'ees. Neuvi\`eme
              S\'erie},
    VOLUME = {137},
      YEAR = {2020},
     PAGES = {33--69},
      ISSN = {0021-7824,1776-3371},
   MRCLASS = {46E35 (31C05 35A15 35C15 35J25 60J45)},
  MRNUMBER = {4088505},
MRREVIEWER = {William\ E.\ Gryc},
       DOI = {10.1016/j.matpur.2019.09.005},
       URL = {https://doi.org/10.1016/j.matpur.2019.09.005},
}

@book {Fukushima_et_alBook,
    AUTHOR = {Fukushima, Masatoshi and Oshima, Yoichi and Takeda, Masayoshi},
     TITLE = {Dirichlet forms and symmetric {M}arkov processes},
    SERIES = {De Gruyter Studies in Mathematics},
    VOLUME = {19},
   EDITION = {extended},
 PUBLISHER = {Walter de Gruyter \& Co., Berlin},
      YEAR = {2011},
     PAGES = {x+489},
      ISBN = {978-3-11-021808-4},
   MRCLASS = {60J25 (28A12 31C45 60F10 60J40 60J45 60J55)},
  MRNUMBER = {2778606},
}

@book {ReedSimonBook,
    AUTHOR = {Reed, Michael and Simon, Barry},
     TITLE = {Methods of modern mathematical physics. {I}.
              {F}unctional analysis},
   EDITION = {Second},
 PUBLISHER = {Academic Press, Inc. [Harcourt Brace Jovanovich, Publishers]},
   ADDRESS = {New York},
      YEAR = {1980},
     PAGES = {xv+400},
      ISBN = {0-12-585050-6},
   MRCLASS = {46-01 (47-01 81-01)},
  MRNUMBER = {0751959},
}
\end{document}